\documentclass[11pt]{article}
\usepackage{amsmath,amssymb}
\usepackage{graphicx,psfrag}
\usepackage{amsthm}
\usepackage{enumerate,placeins,wrapfig,booktabs,chngcntr}
\usepackage{subcaption}

\newtheorem{thm}{Theorem}[section] 
\newtheorem{conj}[thm]{Conjecture}

\newtheorem{rmk}[thm]{Remark}
\newtheorem{cor}[thm]{Corollary}
\newtheorem{lem}[thm]{Lemma}
\newtheorem{prop}[thm]{Proposition}

\newtheorem{defn}[thm]{Definition}

  \allowdisplaybreaks[4]

\begin{document}
	\title{Accessibility and Ergodicity of Partially Hyperbolic Diffeomorphisms without Periodic Points}
	\author{Ziqiang Feng\thanks{Department of Mathematics, Southern University of Science and Technology, Shenzhen, 518000 and Beijing International Center for Mathematical Research, Peking University, Beijing, 100871, China. 12031293@mail.sustech.edu.cn; zqfeng@pku.edu.cn}, Ra\'{u}l Ures\thanks{Department of Mathematics and Shenzhen International Center for Mathematics, Southern University of Science and Technology, Shenzhen, 518000, China. ures@sustech.edu.cn}\thanks{This work was partially supported by National Key R\&D Program of China 2022YFA1005801, NSFC 12071202, and NSFC 12161141002.}}

	\maketitle
	

	\begin{abstract}

		We prove that every $C^2$ conservative partially hyperbolic diffeomorphism of a closed 3-manifold without periodic points is ergodic, which gives an affirmative answer to the Ergodicity Conjecture by Hertz-Hertz-Ures \cite{2008nil} in the absence of periodic points. We also show that a partially hyperbolic diffeomorphism of a closed 3-manifold $M$ with no periodic points is accessible if the non-wandering set is all of $M$ and the fundamental group $\pi_1(M)$ is not virtually solvable.

	\end{abstract}

	{\bf Keywords}: Partial hyperbolicity, accessibility, ergodicity, foliations.
	
	{\bf MSC}: 37A25; 37C25; 37C40; 37C86; 37D30; 57R30.

\tableofcontents


\section{Introduction}

In 2008, Hertz, Hertz, and Ures \cite{2008nil} introduced a conjecture that has since played a significant role in the field of partially hyperbolic dynamics. They posited that any $ C^r $ (with $ r > 1 $) conservative partially hyperbolic diffeomorphism on a closed 3-manifold is ergodic, provided there is no embedded 2-torus tangent to the $ E^s \oplus E^u $ bundle (the direct sum of the stable and unstable distributions). In this work, we focus on this conjecture, specifically for the class of partially hyperbolic diffeomorphisms that lack periodic points.

The concept of \textbf{ergodicity}, introduced by Boltzmann in the late 19th century in the context of thermodynamics, was originally formulated to describe the statistical behavior of gas particles. Over time, it evolved into a fundamental property in dynamical systems, characterizing the equidistribution of typical orbits. Formally, a dynamical system $ f $ is said to be ergodic with respect to an invariant measure $ \mu $ if, for every continuous function $ \phi: M \to \mathbb{R} $, the time average of $ \phi $ along the orbit of $ \mu $-almost every point $ x $ equals its spatial average:
$$
\lim_{n \to \infty} \frac{1}{n} \sum_{i=0}^{n-1} \phi(f^i(x)) = \int_M \phi \, d\mu \quad \mu\text{-a.e. } x.
$$
Oxtoby and Ulam \cite{OU1941} demonstrated that ergodicity is a generic property for volume-preserving homeomorphisms. However, extending this genericity to smooth systems proved impossible with the advent of \textbf{KAM theory} \cite{KAM54, Arnold65, Moser62}. KAM theory reveals that, in smooth integrable systems of elliptic type, invariant tori with positive but not full volume persist under small perturbations. This persistence creates open sets of volume-preserving diffeomorphisms that are not ergodic, providing examples of \textbf{stably non-ergodic systems}.

In \cite{Hopf1939}, Hopf introduced a technique known as the \emph{Hopf argument}, which was used to establish the ergodicity of the geodesic flow on a class of surfaces with negative curvature. Later, Anosov and Sinai \cite{Anosov1967, AnosovSinai67} refined and extended the Hopf argument to prove the ergodicity of all $C^2$ volume-preserving uniformly hyperbolic systems, including geodesic flows on compact manifolds of negative curvature.  

A natural generalization of uniform hyperbolicity, where the Hopf argument might be extended, is the concept of partial hyperbolicity. Partial hyperbolicity was independently introduced by Brin and Pesin \cite{BrinPesin74} in the context of skew products and frame flows, and by Pugh and Shub \cite{PughShub72} in their study of Anosov actions.  

Let $M$ be a closed Riemannian manifold. A diffeomorphism $f: M \rightarrow M$ is said to be (pointwise) \emph{partially hyperbolic} if the tangent bundle $TM$ admits a non-trivial $Df$-invariant splitting $TM = E^s \oplus E^c \oplus E^u$ such that  
\begin{equation*}
	\|Df(x)v^s\| < 1 < \|Df(x)v^u\| 
	\quad \text{and} \quad
	\|Df(x)v^s\| < \|Df(x)v^c\| < \|Df(x)v^u\|
\end{equation*}
holds for every $x \in M$ and unit vectors $v^\sigma \in E^\sigma_x$ ($\sigma = s, c, u$).  
Here, the stable bundle $E^s$ is contracting, the unstable bundle $E^u$ is expanding, and the intermediate bundle $E^c$, called the center bundle, is neither as expanded as the unstable bundle nor as contracted as the stable one.

Following the foundational work of \cite{GPS94}, Pugh and Shub proposed their Stable Ergodicity Conjecture, which states that ``stable ergodicity is dense among conservative partially hyperbolic diffeomorphisms.'' Their conjecture was divided into two parts: (1) (essential) accessibility implies ergodicity for a $C^2$ conservative partially hyperbolic diffeomorphism; and (2) accessibility holds in an open and dense set of partially hyperbolic diffeomorphisms.  

Significant progress has been made in understanding stable ergodicity for partially hyperbolic diffeomorphisms. In the case where the center bundle $E^c$ is one-dimensional, the Pugh-Shub conjecture was resolved by Hertz, Hertz, and Ures \cite{08invent}. Part (1) of the conjecture, which asserts that (essential) accessibility implies ergodicity, was proved by Burns and Wilkinson \cite{BW10annals} under a technical bunching condition; this condition is automatically satisfied when the center bundle is one-dimensional. Later, Avila, Crovisier, and Wilkinson \cite{ACW21} established the $C^1$ denseness of stably ergodic partially hyperbolic systems. For a comprehensive overview of the Pugh-Shub conjecture and related topics in partially hyperbolic dynamics, we refer the reader to the survey article \cite{2018survey}.

As articulated in the Pugh-Shub conjecture, ergodicity is a prevalent property among volume-preserving partially hyperbolic diffeomorphisms. This paper is motivated by the following conjecture, which provides a complete characterization of manifolds admitting non-ergodic dynamics:  

\begin{conj}[Weak Ergodicity Conjecture]\cite{2008nil}
	Let $f: M^3 \rightarrow M^3$ be a $C^r$ ($r > 1$) non-ergodic partially hyperbolic diffeomorphism on an orientable closed 3-manifold. Then $M^3$ must be one of the following:  
	
	(1) the 3-torus $\mathbb{T}^3$;  
	
	(2) the mapping torus of $-id: \mathbb{T}^2 \rightarrow \mathbb{T}^2$;  
	
	(3) the mapping tori of Anosov diffeomorphisms of $\mathbb{T}^2$.  
\end{conj}

This conjecture establishes a connection between dynamical properties and the topology of the ambient manifold, suggesting that the manifold's structure fundamentally influences whether a conservative partially hyperbolic diffeomorphism exhibits ergodic behavior.  

Hertz, Hertz, and Ures further proposed a stronger version of their ergodicity conjecture:  

\begin{conj}[Strong Ergodicity Conjecture]\cite{2008nil}
	If a $C^r$ ($r > 1$) conservative partially hyperbolic diffeomorphism on a 3-manifold is non-ergodic, then there exists an embedded 2-torus tangent to $E^s \oplus E^u$.  
\end{conj}

The Strong Ergodicity Conjecture posits that the sole obstruction to ergodicity is the presence of a proper compact accessibility class. As the nomenclature implies, this version is stronger than the Weak Conjecture, a conclusion supported by the result in \cite{2011TORI} (see Theorem \ref{maptori}).

Over the past few decades, significant effort has been devoted to characterizing ambient manifolds in order to verify the conjectures above. In \cite{2008nil}, the authors confirmed the Weak Ergodicity Conjecture for manifolds with (virtually) solvable fundamental groups. Notably, they identified the first class of manifolds where all conservative partially hyperbolic diffeomorphisms are ergodic. Building on results from \cite{HamU14CCM}, Gan and Shi \cite{GS20DA} resolved the Strong Ergodicity Conjecture for manifolds with (virtually) solvable $\pi_1(M)$. Consequently, non-solvable manifolds have garnered significant attention in the pursuit of verifying the Ergodicity Conjecture. Hammerlindl, Hertz and Ures \cite{2020Seifert} validated the conjecture for conservative partially hyperbolic diffeomorphisms isotopic to the identity on Seifert manifolds. Recent work by Fenley and Potrie demonstrated that the isotopy-class constraint in \cite{2020Seifert} can be removed if the manifold has a hyperbolic base and the induced base action of the diffeomorphism is pseudo-Anosov \cite{FP21accessible,FP-gafa}.  Furthermore, they established ergodicity for partially hyperbolic diffeomorphisms on closed hyperbolic 3-manifolds \cite{FP_hyperbolic}. 

As articulated in the Pugh-Shub conjecture, the set of conservative non-ergodic partially hyperbolic diffeomorphisms is meager. A central challenge for the Hertz-Hertz-Ures Ergodicity Conjecture lies in refining this meager set to well-understood cases by restricting both the dynamical systems and the ambient manifolds.

In this paper, we affirm the Hertz-Hertz-Ures Strong Ergodicity Conjecture for partially hyperbolic diffeomorphisms without periodic points. This result is more powerful, as it precludes all non-ergodic scenarios—even those admitting an embedded $su$-torus.

\begin{thm}\label{ergodic}
	Any $C^2$ conservative partially hyperbolic diffeomorphism on a closed 3-dimensional manifold with no periodic points is ergodic.
\end{thm}

Among techniques for proving ergodicity, accessibility—inspired by the first part of the Pugh-Shub conjecture—has emerged as one of the most effective tools (see \cite{08invent,BW10annals}). Specifically, investigating ergodicity can be reduced to a geometric problem rooted in the study of accessibility.

Hertz, Hertz, and Ures \cite{2008nil} demonstrated that the only 3-manifold with (virtually) nilpotent fundamental group admitting a non-accessible partially hyperbolic diffeomorphism is the 3-torus, provided the nonwandering set is the entire manifold. Accessibility has also been established for conservative partially hyperbolic diffeomorphisms in the following settings: (1) on closed Seifert 3-manifolds, either in the isotopy class of the identity \cite{2020Seifert} or acting pseudo-Anosov on the base \cite{FP-gafa}; and (2) on hyperbolic 3-manifolds \cite{FP_hyperbolic}.

An alternative approach to resolving accessibility and ergodicity involves specializing the dynamics. For manifolds with non-solvable fundamental groups, Fenley and Potrie \cite{FP_hyperbolic} proved accessibility for discretized Anosov flows. Subsequently, in \cite{FP21accessible}, they expanded this result to a broader class of partially hyperbolic diffeomorphisms under the conditions that $\pi_1(M)$ is not virtually solvable and $NW(f) = M$. This paper further addresses the classification of accessible partially hyperbolic systems.  

\begin{thm}\label{1accessible}
	Let $f: M \rightarrow M$ be a $C^1$ partially hyperbolic diffeomorphism of a closed 3-manifold without periodic points. If $\pi_1(M)$ is not virtually solvable and $NW(f) = M$, then $f$ is accessible.  
\end{thm}  

Notably, accessibility can be established without assuming the non-wandering or conservative properties, as shown in \cite{DW03,BHHTU08,HS21DA,FP21accessible}. In \cite{FU2}, we extend the techniques developed here to prove that Theorem \ref{1accessible} remains valid even when $NW(f) \neq M$.  

The proof of Theorem \ref{ergodic} hinges on Theorem \ref{1accessible}, which relies on deep results from the geometry and topology of codimension-one foliations in 3-manifolds. We outline our strategy below:  

After passing to a finite cover and iteration, we may assume the manifold and all invariant subbundles are orientable, with orientations preserved by the dynamics. Crucially, this assumption does not affect the accessibility property.

Assuming the non-wandering property, the non-accessible case implies one of two possibilities: either the bundle $E^s \oplus E^u$ integrates into a foliation with non-compact leaves, or there exists a 2-torus tangent to $E^s \oplus E^u$ (see Theorem \ref{1su-foliation}). The latter scenario would yield an Anosov torus, forcing the manifold's fundamental group to be solvable. This reduces the problem to studying foliations tangent to $E^s \oplus E^u$. Additionally, the foliations tangent to $E^s \oplus E^c$ and $E^c \oplus E^u$ contain no compact leaves (Corollary \ref{1nocompact}). However, the intersection dynamics of these foliations—and their behavior within each foliation—remain poorly understood. A primary objective of this work is to derive topological constraints on these intersections. Specifically, we aim to exclude the existence of an $su$-foliation by demonstrating a global product structure among the leaves of the three foliations in the universal cover, a phenomenon that can only occur if the manifold has a solvable fundamental group.

In Section \ref{top}, we first establish that every lifted center leaf intersects each lifted $su$-leaf at precisely one point in the universal cover. This follows from the absence of compact $su$-leaves, which precludes multiple intersection points between a center leaf and an $su$-leaf. Within center-stable leaves, we prove that any pair of center and stable leaves intersects non-trivially, inducing a leafwise product structure between stable and center leaves. Analogously, this holds for center-unstable leaves. Extending this product structure to the universal cover guarantees the intersection of center leaves and $su$-leaves, reinforcing the global geometric constraints.

Every $su$-leaf is Gromov hyperbolic unless the manifold's fundamental group is virtually solvable (Corollary \ref{hyperbolicleaf}). Consequently, when lifted to the universal cover, each $su$-leaf can be compactified into a hyperbolic disk with an ideal circle at infinity. By analyzing the stable subfoliation within lifted $su$-leaves—induced by their intersection with the lifted center-stable foliation—we demonstrate in Section \ref{nondenselimit} that every ray in a stable leaf accumulates at a unique point on this ideal circle.  

We now show that the set of ideal points of stable leaves is dense in the ideal circle of any $su$-leaf. This result is established in two principal steps:  First, we get rid of the case where every pair of lifted $su$-leaves lies at bounded Hausdorff distance from one another (a scenario we term a \emph{uniform} $su$-foliation).
Second, if the $su$-foliation is \emph{non-uniform}, we demonstrate that the manifold must be a torus bundle over the circle under the density failure condition.  

Next, within each lifted $su$-leaf, we prove that every stable leaf possesses two distinct ideal points (Proposition \ref{twodistinctiidealpoint}). This result hinges critically on the global product structure of the $su$-foliation and center foliation. To achieve this, we analyze the distance between a stable leaf and the corresponding geodesic connecting its endpoints in the hyperbolic disk. This analysis establishes two key properties: the leaf space of the stable subfoliation within lifted $su$-leaves is Hausdorff, and stable leaves are uniform quasi-geodesics in the $su$-leaves.  

The proof of Theorem \ref{1accessible} is presented in Section \ref{proof}, while Section  \ref{pf_ergodic} is dedicated to the proof of Theorem \ref{ergodic}.


\section{Preliminaries}\label{preliminaries}

\subsection{Partial hyperbolicity}

Let \( M \) be a compact Riemannian manifold. A diffeomorphism \( f: M \to M \) is \emph{partially hyperbolic} if the tangent bundle splits into three nontrivial \( Df \)-invariant subbundles \( TM = E^s \oplus E^c \oplus E^u \), such that for all \( x \in M \) and unit vectors \( v^\sigma \in E^\sigma_x \) (\( \sigma = s, c, u \)):
\begin{equation*}
    \begin{aligned}
        \|Df(x)v^s\| &< 1 < \|Df(x)v^u\| \\
        \text{and} \quad \|Df(x)v^s\| &< \|Df(x)v^c\| < \|Df(x)v^u\|.
    \end{aligned}
\end{equation*}

\begin{thm}{\cite{BrinPesin74, HPS77}}
    Let \( f: M \to M \) be a partially hyperbolic diffeomorphism. Then there exist unique invariant foliations \( \mathcal{F}^s \) and \( \mathcal{F}^u \), tangent to \( E^s \) and \( E^u \), respectively.
\end{thm}

This classical result establishes the unique integrability of the strong stable and unstable bundles. The foliations \( \mathcal{F}^s \) and \( \mathcal{F}^u \) are called the \emph{stable} and \emph{unstable foliations}. However, the center bundle \( E^c \) is not always integrable, and its integrability remains a longstanding open problem. A partially hyperbolic diffeomorphism is \emph{dynamically coherent} if there exist invariant foliations \( \mathcal{F}^{cs} \) and \( \mathcal{F}^{cu} \), tangent to \( E^s \oplus E^c \) and \( E^c \oplus E^u \), respectively. Dynamical coherence is not guaranteed; see \cite{Wilkinson98, BurnsWilkinson08, 2016example, BGHP3} for non-coherent examples. The foliations \( \mathcal{F}^{cs} \) and \( \mathcal{F}^{cu} \) are called the \emph{center-stable} and \emph{center-unstable foliations}. Under dynamical coherence, intersecting \( \mathcal{F}^{cs} \) and \( \mathcal{F}^{cu} \) yields a one-dimensional \( f \)-invariant foliation tangent to \( E^c \). If \( E^c \) is integrable, the resulting foliation \( \mathcal{F}^c \) is called the \emph{center foliation}. 

The bundle \( E^c \) is \emph{uniquely integrable} if there exists a \( C^0 \)-foliation \( \mathcal{F}^c \) with \( C^1 \)-leaves such that every \( C^1 \)-embedded curve \( \sigma: [0,1] \to M \), satisfying \( \dot{\sigma}(t) \in E^c(\sigma(t)) \), lies entirely within the leaf \( \mathcal{F}^c(\sigma(0)) \). Unique integrability of \( E^c \) implies dynamical coherence. We refer to \cite{BurnsWilkinson08} for further discussion. Notably, dynamical coherence does not preclude the existence of center-stable surfaces outside the foliation \( \mathcal{F}^{cs} \); thus, unique integrability is strictly stronger (see \cite{2018survey}).

\begin{thm}\cite{2015Center}\label{nocompactleaf}
    Let \( f: M^3 \to M^3 \) be a dynamically coherent partially hyperbolic diffeomorphism. Then \( \mathcal{F}^{cs} \) and \( \mathcal{F}^{cu} \) have no compact leaves.
\end{thm}

Three-dimensional manifolds admitting compact tori tangent to \( E^s \oplus E^c \), \( E^c \oplus E^u \), or \( E^s \oplus E^u \) are classified below. Such tori force the fundamental group to be (virtually) solvable.

\begin{thm} \cite{2011TORI}\label{maptori}
    Let \( f: M^3 \to M^3 \) be a partially hyperbolic diffeomorphism on a closed orientable 3-manifold. If there exists an \( f \)-invariant 2-torus \( T \) tangent to \( E^s \oplus E^u \), \( E^c \oplus E^u \), or \( E^c \oplus E^s \), then \( M^3 \) must be:
    \begin{enumerate}
        \item The 3-torus \( \mathbb{T}^3 \);
        \item The mapping torus of \( -\mathrm{id}: \mathbb{T}^2 \to \mathbb{T}^2 \); or
        \item The mapping torus of a hyperbolic automorphism of \( \mathbb{T}^2 \).
    \end{enumerate}
\end{thm}

A set is \emph{s-saturated} (resp.~\emph{u-saturated}) if it is a union of stable (resp.~unstable) leaves. A set is \emph{su-saturated} if it is both $s$- and $u$-saturated. The \emph{accessibility class} \( AC(x) \) of \( x \in M \) is the minimal $su$-saturated set containing \( x \). Points in the same accessibility class are connected by an \emph{$su$-path} (piecewise tangent to \( E^s \cup E^u \)). The diffeomorphism \( f \) is \emph{accessible} if \( AC(x) = M \) for all \( x \), i.e., any two points are connected by an $su$-path. Let \( \Gamma(f) \) denote the set of non-open accessibility classes; then \( f \) is accessible if and only if \( \Gamma(f) = \emptyset \).

The bundles \( E^s \) and \( E^u \) are \emph{jointly integrable} at \( x \in M \) if there exists \( \delta > 0 \) such that for all \( y \in W^s_\delta(x) \) and \( z \in W^u_\delta(x) \), the intersection \( W^u_{\mathrm{loc}}(y) \cap W^s_{\mathrm{loc}}(z) \) is nonempty.

\begin{prop}{\cite{08invent}}\label{jointint}
    \( \Gamma(f) \) is a compact, codimension-one invariant set laminated by accessibility classes. A point \( x \) belongs to \( \Gamma(f) \) if and only if \( E^s \) and \( E^u \) are jointly integrable at all points of \( AC(x) \). Moreover, \( E^s \) and \( E^u \) are globally jointly integrable if and only if \( \Gamma(f) = M \).
\end{prop}

\begin{prop} \cite{2006Some}\label{Gamma(f)=M}
    Let \( f: M \to M \) be a partially hyperbolic diffeomorphism with one-dimensional \( E^c \). If \( f \) has no periodic points and \( E^s \), \( E^u \) are jointly integrable, then \( E^c \) is uniquely integrable.
\end{prop}

\subsection{Non-wandering property}

We now state key results for partially hyperbolic diffeomorphisms whose non-wandering set is the whole manifold, which will be used throughout this paper. Notably, in the non-wandering setting, we need not assume integrability of the center bundle.

\begin{thm}{\cite{2006Some}}\label{1integrable}
    Let \( f \) be a partially hyperbolic diffeomorphism of a closed 3-manifold. If \( f \) has no periodic points, then the following hold:
    \begin{enumerate}
        \item \( NW(f) = M \) if and only if \( f \) is transitive.
        \item If \( f \) is transitive and has no periodic points, then \( E^c \) is uniquely integrable.
    \end{enumerate}
\end{thm}

Consequently, \( f \) is dynamically coherent, admitting \( f \)-invariant foliations \( \mathcal{F}^{cs} \) and \( \mathcal{F}^{cu} \) tangent to \( E^s \oplus E^c \) and \( E^c \oplus E^u \), respectively. When \( NW(f) = M \), the following structure theorem applies:

\begin{thm}{\cite{2008nil}}\label{1su-foliation}
    Let \( f \) be a partially hyperbolic diffeomorphism on an orientable 3-manifold \( M \) with \( NW(f) = M \). Assume \( E^s \), \( E^c \), \( E^u \) are orientable and \( f \) is not accessible. Then one of the following holds:
    \begin{enumerate}
        \item There exists an incompressible torus tangent to \( E^s \oplus E^u \);
        \item There exists a nontrivial \( f \)-invariant lamination \( \Gamma(f) \subsetneq M \) tangent to \( E^s \oplus E^u \), extendable to a foliation without compact leaves. Boundary leaves of \( \Gamma(f) \) are periodic and contain dense periodic points;
        \item There exists an \( f \)-invariant foliation tangent to \( E^s \oplus E^u \) without compact leaves.
    \end{enumerate}
\end{thm}

By Theorems \ref{nocompactleaf} and \ref{1integrable}, \( \mathcal{F}^{cs} \) and \( \mathcal{F}^{cu} \) have no compact leaves. In Case 1 of Theorem \ref{1su-foliation}, \( M \) has solvable fundamental group (Proposition \ref{nocompact}). Case 2 implies \( f \) has periodic points, see Proposition \ref{1sublamination}. Thus, under aperiodicity and non-solvability assumptions, only Case 3 persists:

\begin{cor}\label{1nocompact}
    Under the assumptions of Theorem \ref{1su-foliation}, if \( f \) has no periodic points and \( \pi_1(M) \) is not virtually solvable, then \( M \) admits three minimal foliations:
    \begin{itemize}
        \item \( \mathcal{F}^{cs} \) (tangent to \( E^s \oplus E^c \)),
        \item \( \mathcal{F}^{cu} \) (tangent to \( E^c \oplus E^u \)),
        \item \( \mathcal{F}^{su} \) (tangent to \( E^s \oplus E^u \)),
    \end{itemize}
    none of which have compact leaves. Minimality follows from Lemma \ref{1minimal}.
\end{cor}

\begin{prop}{\cite[Proposition A.5]{08invent}}\label{1sublamination}
    Let \( \Lambda \) be an \( f \)-invariant sublamination tangent to \( E^s \oplus E^u \). If \( NW(f) = M \), then the boundary leaves of \( \Lambda \) are periodic and contain dense periodic points.
\end{prop}

\subsection{Foliation by non-compact leaves}

We summarize key results on codimension-one foliations, particularly 2-dimensional foliations in 3-manifolds.  

A \emph{Reeb component} of a foliation is a solid torus whose interior is foliated by planes transverse to the core, with all leaves limiting on the boundary torus (which is itself a leaf). A foliation is \emph{Reebless} if it contains no Reeb components. A foundational result is:

\begin{thm}[Novikov]\label{Novikov}
    Let \( M \) be a compact orientable 3-manifold, and \( \mathcal{F} \) a Reebless, transversally orientable, codimension-one foliation. Then:
    \begin{enumerate}
        \item[(i)] \( \mathcal{F} \) admits no null-homotopic closed transversal;
        \item[(ii)] Every non-nullhomotopic closed path in a leaf of \( \mathcal{F} \) remains non-nullhomotopic in \( M \).
    \end{enumerate}
\end{thm}

Let \( \widetilde{\mathcal{F}} \) denote the lift of \( \mathcal{F} \) to the universal cover \( \widetilde{M} \). A foliation \( \mathcal{F} \) is \emph{\(\mathbb{R}\)-covered} if the leaf space of \( \widetilde{\mathcal{F}} \) is Hausdorff and homeomorphic to \( \mathbb{R} \). A foliation is \emph{taut} if there exists a closed transversal intersecting every leaf. For codimension-one foliations:

\begin{thm}{\cite[Corollary 3.3.8, Corollary 6.3.4]{CC00I}}
    Let \( \mathcal{F} \) be a transversely oriented, codimension-one foliation on a compact 3-manifold \( M \). If \( \mathcal{F} \) has no compact leaves, it is taut.
\end{thm}

The next theorem synthesizes properties of taut foliations, drawing from Novikov's work, \cite{Roussarie71}, \cite{Palmeira}, and \cite{Calegari07book}. Analogous results hold for essential laminations \cite{GabaiOertel89}.

\begin{thm}\label{taut}
    Let \( M \) be a closed 3-manifold not finitely covered by \( \mathbb{S}^2 \times \mathbb{S}^1 \). If \( M \) admits a taut foliation \( \mathcal{F} \), then:
    \begin{itemize}
        \item Every leaf of \( \widetilde{\mathcal{F}} \) is a properly embedded plane dividing \( \widetilde{M} \) into two connected half-spaces.
        \item The universal cover \( \widetilde{M} \) is homeomorphic to \( \mathbb{R}^3 \).
        \item Every closed transversal to \( \mathcal{F} \) is homotopically non-trivial.
        \item Transversals to \( \widetilde{\mathcal{F}} \) intersect each leaf at most once.
    \end{itemize}
\end{thm}


\section{Geometrical Structure}\label{top}

Throughout this paper, \( M \) denotes a closed 3-manifold, and \( f: M \to M \) is a partially hyperbolic diffeomorphism with no periodic points and not accessible. After passing to a finite cover and iterating \( f \) if necessary, we assume:
\begin{itemize}
    \item \( M \) and the subbundles \( E^\sigma \) (\( \sigma = c, s, u \)) are orientable;
    \item \( f \) preserves the orientations of \( E^\sigma \).
\end{itemize}
These assumptions do not affect the non-accessibility of \(f\). Our goal is to prove Theorem \ref{1accessible} in subsequent sections.  

Key examples of non-accessible systems include:  
\begin{itemize}
    \item \textbf{Suspension Anosov flows:} The time-\(\alpha\) map (\(\alpha \in \mathbb{R} \setminus \mathbb{Q}\)) yields partially hyperbolic diffeomorphisms without periodic points.
    \item \textbf{Skew products on \(\mathbb{T}^3\):} Fiberwise irrational rotations over a linear Anosov base produce systems where accessibility classes are 2-tori.
\end{itemize}

\begin{prop}\label{nocompact}
    If \( \pi_1(M) \) is not virtually solvable, then \( M \) contains no compact \( su \)-leaves.
\end{prop}

\begin{proof}
    Suppose for contradiction there exists a compact \( su \)-leaf \( T \subset M \). By the Poincaré-Hopf theorem, \( T \) is a 2-torus. As shown in \cite{2011TORI}, \( T \) must be an Anosov torus. Theorem \ref{maptori} then implies \( M \) is either \( \mathbb{T}^3 \), the mapping torus of \( -\mathrm{id}: \mathbb{T}^2 \to \mathbb{T}^2 \), or the mapping torus of a hyperbolic automorphism—all cases where \( \pi_1(M) \) is solvable, contradicting the hypothesis.
\end{proof}

\subsection{Completeness of the center-stable foliation}\label{section_complete}

The concept of completeness, introduced in \cite{BW05}, is central to understanding the global structure of foliations. In this subsection, we prove completeness for center-stable and center-unstable foliations, leading to topological constraints on leaves and the ambient manifold.

\begin{defn}
    The \emph{center-stable foliation} \( \mathcal{F}^{cs} \) is \textbf{complete} if \( \mathcal{F}^{cs}(x) = W^s(W^c(x)) \) for all \( x \in M \). Analogously, the \emph{center-unstable foliation} \( \mathcal{F}^{cu} \) is complete if \( \mathcal{F}^{cu}(x) = W^u(W^c(x)) \).
\end{defn}

A leaf \( F \in \mathcal{F}^{cs} \) is a \textbf{product} of center and stable leaves if \( F = W^s(W^c(x)) = W^c(W^s(x)) \) for all \( x \in F \). Completeness implies this product structure globally:

\begin{lem}\label{bi-foliated}
    The following are equivalent:
    \begin{enumerate}
        \item \( \mathcal{F}^{cs} \) is complete;
        \item Every leaf of \( \mathcal{F}^{cs} \) is a product of center and stable leaves.
    \end{enumerate}
\end{lem}

\begin{proof}
    (2) \(\Rightarrow\) (1) is immediate. For (1) \(\Rightarrow\) (2), assume \( \mathcal{F}^{cs} \) is complete. Suppose for contradiction there exists \( F \in \mathcal{F}^{cs} \) and \( x, y \in F \) such that \( W^c(y) \cap W^s(x) = \emptyset \). By completeness, \( F = W^s(W^c(y)) \), so \( W^s(x) \) must intersect \( W^c(y) \), a contradiction.
\end{proof}

\begin{prop}\label{1complete}\cite[Lemma 6.3]{2020Seifert}
    Let \( f \) be non-accessible with \( NW(f) = M \) and no periodic points. Then \( \mathcal{F}^{cs} \) and \( \mathcal{F}^{cu} \) are complete.
\end{prop}

Let \( \widetilde{M} \) be the universal cover of \( M \), and let \( \widetilde{\mathcal{F}}^{\sigma} \) (\( \sigma = cs, cu \)) denote the lifted foliations. Proposition \ref{1complete} implies:

\begin{cor}\label{liftcomplete}
    \( \widetilde{\mathcal{F}}^{cs} \) and \( \widetilde{\mathcal{F}}^{cu} \) are complete.
\end{cor}

\begin{lem}\label{cylinder}\cite[Lemma 3.4]{Zhang21}
    A center-stable leaf containing a compact center leaf is a cylinder.
\end{lem}

\begin{proof}
    Denote by $\gamma$ a compact center leaf and by $\mathcal{F}^{cs}(x)$ the center stable leaf containing $\gamma$. If there is a stable leaf intersecting $\gamma$ twice, then we have two lifts of $\gamma$ in $\widetilde{\mathcal{F}}^{cs}(\gamma)$. The completeness of $\widetilde{\mathcal{F}}^{cs}$ implies that every stable leaf in $\mathcal{F}^{cs}(x)$ intersects $\gamma$ at least twice. Thus, the center stable leaf of $\gamma$ is a compact surface. This contradicts Theorem \ref{nocompactleaf}. If every stable leaf intersects $\gamma$ once, we can show that the fundamental group of $\mathcal{F}^{cs}(x)$ acts freely on $\mathbb{R}$. Then the H\"older Theorem \cite{Holder1901} asserts that this group is abelian. Hence $\mathcal{F}^{cs}(x)$ is a cylinder.
\end{proof}

\begin{lem}\label{cylin-plane}
    Every center-stable leaf is either a cylinder or a plane.
\end{lem}

\begin{proof}
    By Lemma \ref{cylinder}, center-stable leaves with compact center leaves are cylinders. For leaves without compact centers, let \( \widetilde{\mathcal{F}}^{cs}(\tilde{x}) \) be a lift. Deck transformations act freely on center/stable leaves in \( \widetilde{\mathcal{F}}^{cs}(\tilde{x}) \), inducing abelian actions. Thus, \( \mathcal{F}^{cs}(x) \) is a plane or cylinder.
\end{proof}

\begin{thm}\label{rosenberg}\cite{Rosenberg}
    If a closed 3-manifold admits a foliation by planes, it is the 3-torus. This holds for \( C^0 \)-foliations by \cite{Gabai90,Imanishi74}.
\end{thm}

\subsection{Intersection between weak foliations}

This subsection establishes key results about intersections between lifted foliations in \( \widetilde{M} \), culminating in Proposition \ref{>1point} and Corollary \ref{=1point}.

\begin{lem}\label{at-most-once}
    For any \( x, y, z \in \widetilde{M} \), the intersections \( \widetilde{W}^c(x) \cap \widetilde{W}^s(y) \), \( \widetilde{W}^c(x) \cap \widetilde{W}^u(y) \), and \( \widetilde{W}^s(x) \cap \widetilde{W}^u(y) \) contain at most one point.
\end{lem}

\begin{proof}
    Suppose \( \widetilde{W}^c(x) \cap \widetilde{W}^s(y) \) has two distinct points. Then \( \widetilde{W}^s(y) \) intersects \( \widetilde{\mathcal{F}}^{cu} \) at two points, violating Theorem \ref{taut} (no compact leaves). Analogous reasoning applies to other intersections.
    
\end{proof}

\begin{lem}\label{embeddedplane}
    Every leaf of \( \widetilde{\mathcal{F}}^\sigma \) (\( \sigma = cs, cu, su \)) is a properly embedded plane dividing \( \widetilde{M} \) into two connected half-spaces.
\end{lem}

\begin{proof}
    Immediate from Theorem \ref{taut}, Corollary \ref{1nocompact}, and Proposition \ref{nocompact}.  
\end{proof}

\begin{lem}\label{<1center}
    For \( x, y \in \widetilde{M} \):  
    \begin{itemize}
        \item \( \widetilde{\mathcal{F}}^{cs}(x) \cap \widetilde{\mathcal{F}}^{cu}(y) \) contains at most one center leaf;
        \item \( \widetilde{\mathcal{F}}^{cs}(x) \cap \widetilde{\mathcal{F}}^{su}(y) \) contains at most one stable leaf;
        \item \( \widetilde{\mathcal{F}}^{cu}(x) \cap \widetilde{\mathcal{F}}^{su}(y) \) contains at most one unstable leaf.
    \end{itemize}
\end{lem}

\begin{proof}
    Suppose \( \widetilde{\mathcal{F}}^{cs}(x) \cap \widetilde{\mathcal{F}}^{cu}(y) \) contains two center leaves \( L_1, L_2 \). By completeness, \( L_2 \subset W^s(L_1) \), so a stable curve \( \alpha \subset \widetilde{\mathcal{F}}^{cs}(x) \) connects \( L_1 \) and \( L_2 \). Since \( L_1,L_2 \subset \widetilde{\mathcal{F}}^{cu}(y) \), \(\alpha\) intersects \( \widetilde{\mathcal{F}}^{cu}(y) \) at two points, contradicting Lemma \ref{at-most-once}. The other cases follow similarly.
\end{proof}

A flow is \textit{regulating} for a foliation $\mathcal{F}$ if the flow is transverse to $\mathcal{F}$ and any lifted orbit of the flow intersects every leaf of $\mathcal{\widetilde{F}}$. Let \( \phi_t: \mathbb{R} \times M \to M \) be the unit-speed center flow tangent to the oriented bundle \( E^c \).

\begin{prop}\label{>1point}
    The center flow \(\phi\) is regulating for \(\mathcal{F}^{su}\).
\end{prop}

\begin{proof}
    To show \(\phi\) regulates \(\mathcal{F}^{su}\), we prove every center leaf \(\gamma \subset \widetilde{\mathcal{F}}^c\) intersects every \(\mathcal{\widetilde{F}}^{su}\)-leaf.

    By completeness of \(\widetilde{\mathcal{F}}^{cs}\), the center-stable leaf \(\widetilde{\mathcal{F}}^{cs}(\gamma)\) equals the \(s\)-saturation of \(\gamma\). If an \(su\)-leaf \(L \in \widetilde{\mathcal{F}}^{su}\) intersects \(\widetilde{\mathcal{F}}^{cs}(\gamma)\) at \(p \in \widetilde{M}\), then \(L\) contains \(W^s(p)\). By completeness, \(W^s(p)\) intersects \(\gamma\), forcing \(L \cap \gamma \neq \emptyset\). Similarly, intersections with \(\widetilde{\mathcal{F}}^{cu}(\gamma)\) imply intersections with \(\gamma\).

    Define \(\gamma^{su} := \bigcup_{x \in \gamma} \widetilde{\mathcal{F}}^{su}(x)\). For \(y \in \gamma^{su}\), let \(x = \widetilde{\mathcal{F}}^{su}(y) \cap \gamma\). Construct an \(su\)-path \(x = x_0, x_1, \dots, x_n = y\) with stable/unstable segments. Let \(\gamma_i = \widetilde{\mathcal{F}}^c(x_i)\). By the previous argument, \(\widetilde{\mathcal{F}}^{su}(y)\) intersects \(\gamma_1\), so \(\gamma^{su} = \gamma_1^{su}\). Iterating, \(\gamma^{su} = \gamma_n^{su}\). Thus, \(\gamma^{su}\) is \(c\)-saturated.

    Let \(\hat{\gamma}\) be another center leaf. If \(z \in \gamma^{su} \cap \hat{\gamma}^{su}\), then \(\widetilde{\mathcal{F}}^{su}(z) \subset \gamma^{su} \cap \hat{\gamma}^{su}\). Since \(\hat{\gamma}^{su}\) is \(c\)-saturated, it contains \(\gamma\) (and vice versa). Hence, \(\gamma^{su} = \hat{\gamma}^{su}\) when non-disjoint.

    The set \(\gamma^{su}\) is open and \(c\)-saturated. By connectivity of \(\widetilde{M}\), \(\gamma^{su} = \widetilde{M}\). Thus, every \(\gamma\) intersects every \(\mathcal{\widetilde{F}}^{su}\)-leaf, proving \(\phi\) regulates \(\mathcal{F}^{su}\).
\end{proof}

\begin{cor}\label{=1point}
    Each center leaf in \( \widetilde{\mathcal{F}}^c \) intersects every \( \mathcal{\widetilde{F}}^{su} \)-leaf at exactly one point. Consequently, \( \mathcal{F}^{su} \) is \( \mathbb{R} \)-covered.
\end{cor}

\begin{proof}
    By Proposition \ref{>1point} and Lemma \ref{at-most-once}, intersections are unique. Then the leaf space of \( \widetilde{\mathcal{F}}^{su} \) is homeomorphic to \( \mathbb{R} \).
\end{proof}


\subsection{Gromov hyperbolicity of leaves}

A foliation is \textit{minimal} if every leaf is dense in \( M \). A set \( \Lambda \) is a \textit{minimal set} if it is a sublamination where every leaf is dense in \( \Lambda \).

\begin{thm}{\cite[Theorem 4.1.3]{HectorHirsch}}\label{hectorhirsch}
    Let \( \mathcal{F} \) be a codimension-one foliation without compact leaves on a compact manifold \( M \). Then \( \mathcal{F} \) has finitely many minimal sets.
\end{thm}

\begin{lem}\label{1minimal}
    If \( \pi_1(M) \) is not (virtually) solvable, the foliations \( \mathcal{F}^{cs} \), \( \mathcal{F}^{cu} \), and \( \mathcal{F}^{su} \) are minimal.
\end{lem}

\begin{proof}
    By Corollary \ref{1nocompact}, \( \mathcal{F}^{su} \) has no compact leaves. Theorem \ref{hectorhirsch} implies \( \mathcal{F}^{su} \) has finitely many minimal sets \( S_1, \dots, S_k \). Their union \( S = \bigcup_{i=1}^k S_i \) is an \( f \)-invariant sublamination. Each \( S_i \) is \( f \)-periodic (as minimal sets cannot intersect). Let \( N \) be their common period; after iterating \( f^N \), assume each \( S_i \) is \( f \)-invariant. By Proposition \ref{1sublamination}, boundary leaves of \( S_i \) contain dense periodic points—contradicting aperiodicity. Thus, \( S_i = M \), proving \( \mathcal{F}^{su} \) is minimal.

    For \( \mathcal{F}^{cs} \) and \( \mathcal{F}^{cu} \), the union of their minimal sets is a compact invariant repeller or attractor. Since \( NW(f) = M \), these foliations must be minimal.
\end{proof}

A codimension-\( k \) foliation \( \mathcal{F} \) has a \emph{(holonomy) invariant transverse measure} \( \mu \) if \( \mu \) is a non-trivial measure on \( k \)-dimensional transversals, invariant under holonomies, and finite on compacts. The \emph{support} of \( \mu \) is the closure of points where transversals through them have \( \mu > 0 \). See \cite[Chapters 11--12]{CC00I} for details.

\begin{thm}{\cite[Chapter X, Theorem 2.3.3]{HectorHirsch}}\label{HH}
    Let \( \mathcal{F} \) be a minimal codimension-one foliation without compact leaves. If \( \mathcal{F} \) has an invariant transverse measure, this measure is unique up to scaling.
\end{thm}

\begin{prop}\label{transversemeasure}
    If \( \mathcal{F}^{su} \) has an invariant transverse measure, \( \pi_1(M) \) is (virtually) solvable.

\end{prop}

\begin{proof}
    Let \( \mu \) be an invariant transverse measure. Suppose \( \pi_1(M) \) is not virtually solvable. By Lemma \ref{1minimal}, \( \mathcal{F}^{su} \) is minimal, so \( \mu \) has full support. The pullback \( \nu = \mu \circ f \) is also invariant. By Theorem \ref{HH}, \( \nu = \lambda \mu \) for some \( \lambda \in \mathbb{R} \).

    If \( |\lambda| \neq 1 \), \( f \) contracts/expands \( \mu \), implying \( E^c \) is contracted/expanded. Thus, \( f \) is conjugate to Anosov and \(M=\mathbb{T}^3\). If \( |\lambda| = 1 \), \( f \) has a topologically neutral center \cite{BZ20}. By transitivity (Theorem \ref{1integrable}), \cite[Theorem C]{BZ20} implies \( f \) is conjugate to a skew product over Anosov (on nilmanifolds) or a non-accessible discretized Anosov flow. The latter case only may occur when \( \pi_1(M) \) is virtually solvable \cite{FP_hyperbolic}.
\end{proof}

\begin{thm}{\cite{Sullivan76, Gromov87, Candel93}}\label{SulliGrom}
    Let \( \mathcal{F} \) be a codimension-one foliation without compact leaves on a closed 3-manifold. Then either:
    \begin{itemize}
        \item \( \mathcal{F} \) has an invariant transverse measure; or
        \item All leaves of \( \mathcal{F} \) are Gromov hyperbolic.
    \end{itemize}
\end{thm}

\begin{cor}\label{hyperbolicleaf}
    If \( f \) has no periodic points, \( NW(f) = M \), and \( \pi_1(M) \) is not virtually solvable, then \( \mathcal{F}^{su} \)-leaves are Gromov hyperbolic.
\end{cor}

\begin{cor}
    Under the hypotheses of Corollary \ref{hyperbolicleaf}, the leaves of \( \mathcal{F}^{cs} \), \( \mathcal{F}^{cu} \), and \( \mathcal{F}^{su} \) are Gromov hyperbolic.
\end{cor}

\begin{proof}
    Immediate from Lemma \ref{1minimal}, Theorem \ref{SulliGrom}, and \cite[Theorem 5.1]{FP_hyperbolic}.
\end{proof}

\section{Non-dense Limit Set}\label{nondenselimit}

In subsequent sections, we assume \( f: M \to M \) is a partially hyperbolic diffeomorphism of a compact 3-manifold with no periodic points, \( f \) is not accessible, \( NW(f) = M \), and \( \pi_1(M) \) is not virtually solvable. By Corollary \ref{1nocompact} and Proposition \ref{1complete}, the foliations \( \mathcal{F}^{cs} \) (tangent to \( E^c \oplus E^s \)) and \( \mathcal{F}^{cu} \) (tangent to \( E^c \oplus E^u \)) are complete. Each leaf of \( \mathcal{F}^{cs} \) and \( \mathcal{F}^{cu} \) is either a cylinder or a plane. By Lemma \ref{1minimal} and Corollary \ref{hyperbolicleaf}, the foliation \( \mathcal{F}^{su} \) (tangent to \( E^s \oplus E^u \)) is minimal, and its leaves are Gromov hyperbolic.

On the universal cover \( \widetilde{M} \), each leaf \( L \in \widetilde{\mathcal{F}}^{su} \) is a Poincaré disk with an ideal circle at infinity denoted by \( \partial_\infty L \). By \cite{Candel93}, there exists a metric on \( M \) restricting to a hyperbolic metric on every \( \mathcal{F}^{su} \)-leaf. The fundamental group of a hyperbolic leaf is generated by homotopy classes of simple closed curves lifting to geodesics in \( \widetilde{M} \). By Theorem \ref{rosenberg}, since \( M \) is not \( \mathbb{T}^3 \), there exists at least one leaf with non-trivial fundamental group. The corresponding deck transformations act on lifted leaves as Möbius transformations preserving the hyperbolic metric and geodesics. Each non-trivial Möbius transformation is hyperbolic, fixing two endpoints on \( \partial_\infty L \) and preserving an invariant geodesic (its \emph{axis}).

\subsection{Single ideal point}

By Corollary \ref{=1point}, every center leaf in \(\widetilde{\mathcal{F}}^c\) intersects any leaf \(L \in \widetilde{\mathcal{F}}^{su}\) at exactly one point. The completeness of \(\widetilde{\mathcal{F}}^{cs}\) (Corollary \ref{liftcomplete}) ensures that the intersection of each center-stable leaf \(\widetilde{\mathcal{F}}^{cs}\) with \(L\) is a stable leaf. Thus, \(\widetilde{\mathcal{F}}^{cs} \cap L\) coincides with the stable foliation restricted to \(L\). We denote by \(\widetilde{\Lambda}^s_L\) and \(\widetilde{\Lambda}^u_L\) the intersections of \(L\) with \(\widetilde{\mathcal{F}}^{cs}\) and \(\widetilde{\mathcal{F}}^{cu}\), respectively, and by \(\partial_\infty L\) the ideal boundary of \(L \in \widetilde{\mathcal{F}}^{su}\). We analyze the asymptotic behavior of these foliations.

\begin{defn}
    The \emph{cylinder at infinity} \(\mathcal{A}\) is defined as \(\mathcal{A} := \bigcup_{F \in \widetilde{\mathcal{F}}^{su}} \partial_\infty F\), the union of ideal boundaries of all leaves in \(\widetilde{\mathcal{F}}^{su}\).
\end{defn}

Following \cite{Calegari00, Fenley02}, we topologize \(\mathcal{A}\): For a leaf \(F \in \widetilde{\mathcal{F}}^{su}\) and \(x \in F\), geodesic rays from \(x\) define a homeomorphism between the unit tangent circle at \(x\) and \(\partial_\infty F\). Let \(\tau \in \widetilde{\mathcal{F}}^c\) contain \(x\); by Corollary \ref{=1point}, \(\tau\) intersects each leaf \(L \in \widetilde{\mathcal{F}}^{su}\) at a unique point \(\tau(L)\). Geodesic rays from \(\tau(L)\) in \(L\) correspond to points in \(\partial_\infty L\), inducing a homeomorphism between the unit tangent circle at \(\tau(L)\) and \(\partial_\infty L\). The union of these structures over \(\tau\) gives \(\mathcal{A}\) the topology of \(S^1 \times \mathbb{R}\). This extends to \(\widetilde{M} \cup \mathcal{A} \approx D^2 \times \mathbb{R}\), where \(D^2 \times \{t\}\) corresponds to \(L \cup \partial_\infty L\) for \(t \in \mathbb{R}\).

For a leaf \(L \in \widetilde{\mathcal{F}}^{su}\), basepoint \(p \in L\), and ideal interval \(I \subset \partial_\infty L\), the \emph{wedge} \(W_p(I)\) consists of geodesic rays from \(p\) to \(I\). The \emph{complementary wedge} is \(W_p(\partial_\infty L \setminus I)\).

We prove that every ray in \(\widetilde{\Lambda}^s_F\) converges to a single ideal point in \(\partial_\infty F\) for \(F \in \widetilde{\mathcal{F}}^{su}\). Levitt \cite{Levitt83} established this for Reebless foliations on hyperbolic surfaces, and Fenley \cite{Fenley09} generalized it to foliations with transverse pseudo-Anosov flows. We adapt these ideas to partially hyperbolic diffeomorphisms using the following lemma.

\begin{lem}\label{distinguished}
    Let \( L \in \widetilde{\mathcal{F}}^{su} \) and let \( l \subset \widetilde{\Lambda}^s_L \) be a ray converging to a non-trivial ideal interval \( I \subset \partial_\infty L \). Then for any \( L_0 \in \widetilde{\mathcal{F}}^{su} \), there exists \( a \in \partial_\infty L_0 \) such that every closed segment \( I_0 \subset \partial_\infty L_0 \setminus \{a\} \) is accumulated by subsegments of a ray \( l_0 \subset \widetilde{\Lambda}^s_{L_0} \).
\end{lem}

\begin{proof}
    Let \( l_i \subset l \) be arcs converging to \( I \). Choose \( p_i \in l_i \) such that \( p_i \to \xi \in \text{int}(I) \). By passing to a subsequence, deck transformations \( g_i \in \pi_1(M) \) satisfy \( g_i(p_i) \to p_0 \in L_0 \), where \( L_0 \in \widetilde{\mathcal{F}}^{su} \). For any \( p_0^1 \in L_0 \), let \( d_1 = d_{L_0}(p_0, p_0^1) \). The circles \( c_i^1 \subset L \) of radius \( d_1 \) centered at \( p_i \) satisfy \( g_i(c_i^1) \to c_0^1 \subset L_0 \), where \( c_0^1 \) is the circle centered at \( p_0^1 \) of radius \(d_1\). Choosing \( p_i^1 \in c_i^1 \) with \( g_i(p_i^1) \to p_0^1 \), we have \( p_i^1 \to \xi \) (see Figure~\ref{fig:single1}). Similarly, for \( p_0^j \in L_0 \) with \( p_0^j \to  p_0^\infty  \in \partial_\infty L_0 \), select \( p_i^j \in L \) such that \( g_i(p_i^j) \to p_0^j \) and \( p_i^j \to \xi \). By the diagonal argument, \( p_i^i \to \xi \) and \( g_i(p_i^i) \to p_0^\infty \in \partial_\infty L_0 \).

    Define \( S((g_i), \xi) = \{\eta \in \widetilde{M} \cup \mathcal{A} \mid \exists \eta_i \in L, \eta_i \to \xi,\ g_i(\eta_i) \to \eta\} \). Then \( S((g_i), \xi) = L_0 \cup \partial_\infty L_0 \). By minimality (Lemma \ref{1minimal}), \( \pi(S((g_i), \xi)) \) is dense in \( M \), implying \( S(\xi) := \bigcup_{(g_i)} S((g_i), \xi) = \widetilde{M} \cup \mathcal{A} \).

    \begin{figure}[htb]
        \centering
        \includegraphics[width=0.8\textwidth]{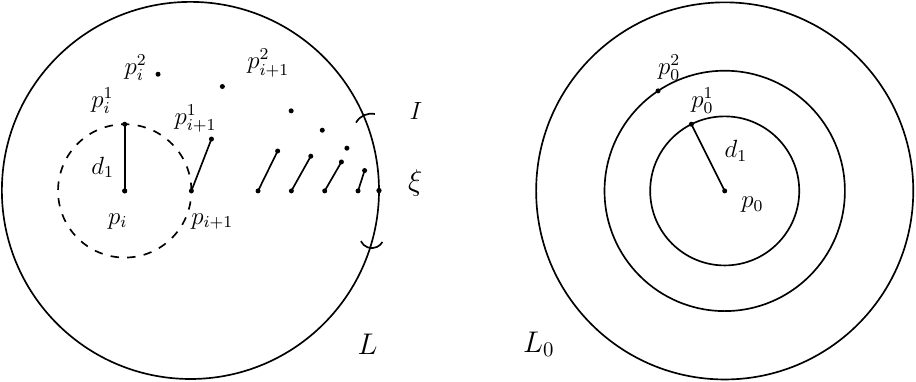}
        \caption{Deck transformations \( g_i \) map sequences in \( L \) converging to \( \xi \) onto \( L_0 \cup \partial_\infty L_0 \).}
        \label{fig:single1}
    \end{figure}

    Consider the angle \( \angle_{p_i}I \) at \( p_i \in L \). As \( p_i \to \xi \), \( \angle_{p_i}I \to 2\pi \). Deck transformations preserve angles, so \( \angle_{g_i(p_i)}g_i(I) \to 2\pi \), forcing \( \angle_{g_i(p_i)}(\partial_\infty g_i(L) \setminus g_i(I)) \to 0 \). Thus, \( W_{g_i(p_i)}(\partial_\infty g_i(L) \setminus g_i(I)) \) collapses to a geodesic ray in \( L_0 \) from \( p_0 \) to \( a \in \partial_\infty L_0 \), and \( g_i(I) \to \partial_\infty L_0 \setminus \{a\} \) (see Figure~\ref{fig:single2}).

    \begin{figure}[htb]
        \centering
        \includegraphics[width=0.8\textwidth]{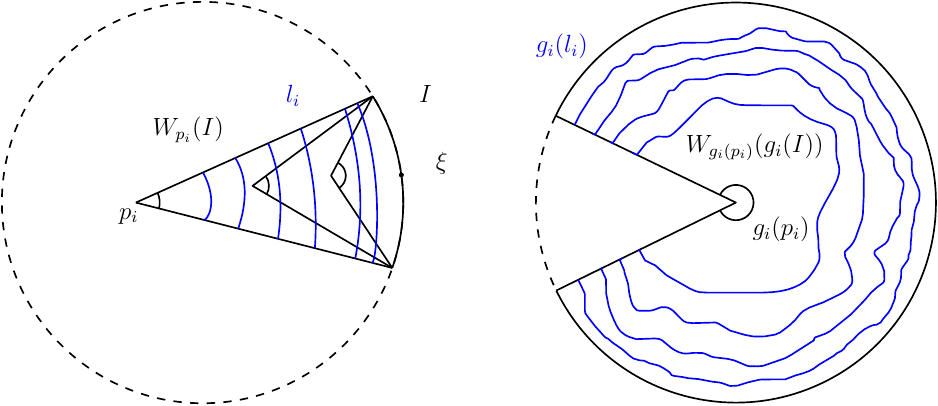}
        \caption{Angles and wedges under deck transformations. Complementary wedges collapse to a single ideal point \( a \).}
        \label{fig:single2}
    \end{figure}

    For any closed segment \( I_0 \subset \partial_\infty L_0 \setminus \{a\} \), there exist stable segments \( s_{n_i}^j \subset \widetilde{\Lambda}^s_{g_i(L)} \) converging to \(I_0 \). Projecting along center leaves, we obtain stable segments in \( \widetilde{\Lambda}^s_{L_0} \) approximating \( I_0 \).
\end{proof}

\begin{prop}\label{singlelimit}
	
	For any leaf $F$ of $\widetilde{\mathcal{F}}^{su}$, each ray $l$ in a leaf of $\widetilde{\Lambda}^{s}_{F}$ ($\widetilde{\Lambda}^{u}_F$) accumulates at a single point of $\partial_{\infty }F$.
\end{prop}

\begin{proof}
    By Lemma~\ref{embeddedplane}, all leaves of $\widetilde{\mathcal{F}}^{su}$ and $\widetilde{\mathcal{F}}^{cs}$ are properly embedded in $\widetilde{M}$. For any leaf $F \in \widetilde{\mathcal{F}}^{su}$, each ray $l$ of $\widetilde{\Lambda}^s_F$ is likewise properly embedded in $F$. Consequently, $l$ can only converge to the boundary $\partial_{\infty}F$.

    Suppose for contradiction that there exists a leaf $F \in \widetilde{\mathcal{F}}^{su}$ and a ray $l$ in a leaf of $\widetilde{\Lambda}^{s}_{F}$ such that $l$ does not converge to a single point. Then $l$ must limit to a connected interval $I$ in $\partial_{\infty}F$. As established in Lemma~\ref{distinguished}, every leaf $L \in \widetilde{\mathcal{F}}^{su}$ possesses an ideal point $a$ on its boundary $\partial_{\infty}L$ with the following property: for any closed segment $I_0$ in $\partial_{\infty}L \setminus \{a\}$, there exists a ray $l_0$ in a leaf of $\widetilde{\Lambda}^{s}_{L}$ with subsegments converging to $I_0$.

    Since the manifold is not the 3-torus, there exists at least one leaf of $\mathcal{F}^{su}$ with non-trivial fundamental group. Assume $\pi_1(\pi(L))$ is non-trivial.

We now analyze the asymptotic behavior of leaves in $\widetilde{\Lambda}_L^u$. We demonstrate that either:
\begin{enumerate}
    \item Each leaf of $\widetilde{\Lambda}_L^u$ converges to every closed interval in $\partial_{\infty}L \setminus \{a\}$ (analogous to $\widetilde{\Lambda}_L^s$), or
    \item All leaves of $\widetilde{\Lambda}_L^u$ possess only one ideal point $a$.
\end{enumerate}
By examining these cases separately, we derive a contradiction to complete the proof. Let $l^s$ and $l^u$ denote leaves of $\widetilde{\Lambda}_L^s$ and $\widetilde{\Lambda}_L^u$ respectively.

\medskip\noindent\textbf{Case I: Existence of a leaf in $\widetilde{\Lambda}^u_L$ accumulating at an ideal point $b \neq a$.}

Suppose there exists a ray $l^u_+$ in a leaf of $\widetilde{\Lambda}_L^u$ accumulating at an ideal point $b \neq a$ in $\partial_{\infty}L$. Select an interval $I_b \subset \partial_{\infty}L$ containing $b$ but excluding $a$. Since $l^s$ converges to any interval in $\partial_{\infty}L \setminus \{a\}$, there exists a sequence of segments $\{l^s_i\}_{i\in\mathbb{N}}$ in $l^s$ converging to $I_b$. By Lemma~\ref{at-most-once}, no leaf of $\widetilde{\mathcal{F}}^u$ intersects $l^s$ multiple times. Therefore, $l^u_+$ must align parallel to $l^s$ to avoid multiple intersections with the $l^s_i$ sequence. As $l^u_+$ accumulates at $b$, it converges to $I_b$ along with $\{l^s_i\}$ (see Figure~\ref{fig:2cases}). Expanding $I_b$ arbitrarily within $\partial_{\infty}L \setminus \{a\}$, we conclude $l^u_+$ converges to every interval in $\partial_{\infty}L \setminus \{a\}$.

Let $\gamma$ be a geodesic in $L$ corresponding to a non-trivial element of $\pi_1(\pi(L))$. This curve $\gamma$ possesses at least one ideal point $z \in \partial_{\infty}L$ distinct from $a$. For any interval $I_z$ containing $z$ but not $a$, the ray $l^u_+$ converges to $I_z$, resulting in infinitely many intersections with $\gamma$. Crucially, these intersection points project to distinct points in $M$ under the covering map $\pi: \widetilde{M} \rightarrow M$. If not, a non-trivial deck transformation would preserve $l^u_+$, implying a closed unstable leaf in $\mathcal{F}^u$ - an impossibility.

Let $\gamma_0 = \pi(\gamma)$. Observe that the intersection $\pi(l^u_+) \cap \gamma_0$ contains infinitely many distinct points. Since $l^u_+$ converges to every interval in $\partial_{\infty}L \setminus \{a\}$, we may select infinitely many points in $\pi(l^u_+) \cap \gamma_0$ lying in distinct local unstable manifolds. There exist $x_1, x_2 \in \pi(l^u_+) \cap \gamma_0$ within the same local product neighborhood but in different local unstable manifolds.

Within this neighborhood, we smoothly connect the local unstable manifolds of $x_1$ and $x_2$ to construct a closed curve $\alpha$ transverse to $\mathcal{F}^s$. By Novikov's theorem, such a curve in an su-leaf transverse to $\mathcal{F}^s$ must be non-contractible. If $\alpha$ were null-homotopic, it would simultaneously be transverse to $\mathcal{F}^{cs}$, contradicting Theorem~\ref{taut}. 

Lifting $\alpha$ to $\widetilde{\alpha}$ in $\widetilde{M}$, its non-contractibility implies invariance under a non-trivial deck transformation. This transformation may differ from that fixing $\gamma$, as $\pi_1(\pi(L))$ might have multiple generators. The curve $\widetilde{\alpha}$ has two ideal points in $\partial_{\infty}L$, with at least one distinct from $a$. Let $w \in \partial_{\infty}L \setminus \{a\}$ be this point, and consider an interval $I_w \subset \partial_{\infty}L \setminus \{a\}$ containing $w$. Since $l^s$ contains a ray converging to $I_w$, it must intersect $\widetilde{\alpha}$ infinitely often. However, Theorem~\ref{taut} establishes that $\widetilde{\alpha}$ - being a transversal to $\widetilde{\mathcal{F}}^s$ - intersects each leaf at most once, yielding a contradiction.

\begin{figure}[htb]
    \centering
    \subcaptionbox{Leaf $l^u$ with ideal point $b \neq a$}{\includegraphics{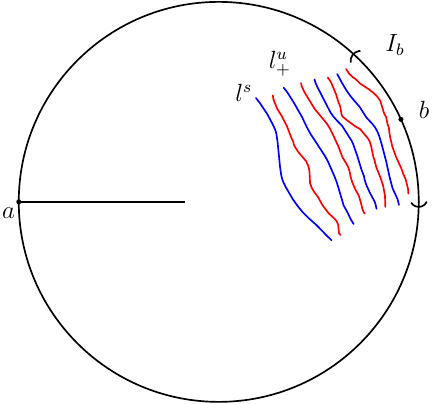}}
    \subcaptionbox{All leaves of $\widetilde{\Lambda}^u_L$ share ideal point $a$}{\includegraphics{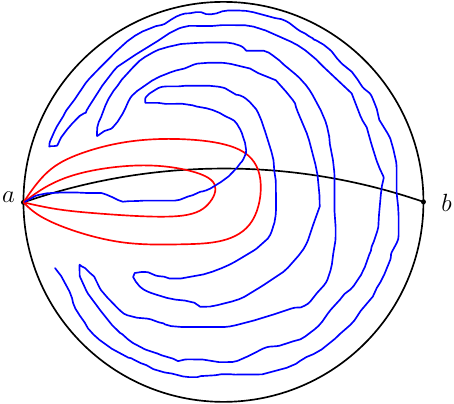}}
    \caption{Configurations of unstable leaves}
    \label{fig:2cases}
\end{figure}

\noindent\textbf{Case II: All rays of $\widetilde{\Lambda}^u_L$ share ideal point $a$.}

From the preceding argument, all leaves of $\widetilde{\Lambda}_L^u$ converge to $a$ from both sides. Deck transformations on $L$ preserve $a$ since they map leaves of $\widetilde{\Lambda}_L^u$ to each other. In the Poincaré disk, geodesics sharing an ideal point must be asymptotic. However, non-homotopic closed curves cannot be asymptotic, implying only one closed geodesic converges to $a$. Therefore, $\pi_1(\pi(L))$ is either trivial or cyclic, making $\pi(L)$ a plane or cylinder. By assumption $\pi(L)\neq\mathbb{R}^2$, it must be a cylinder.

Let $\gamma$ be the axis of $\pi_1(\pi(L))$'s generator, with limit points $a$ and $b\in\partial_{\infty}L$. As all leaves of $\widetilde{\Lambda}_L^u$ limit to $a$, any intersection between $l^s$ and $l^u\in\widetilde{\Lambda}_L^u$ traps one ray of $l^s$ in the region bounded by $l^u$ and $a$, since $l^s$ intersects each unstable leaf at most once.

For $v\in l^s$, let $l^s_-(v)$ denote the ray from $v$ to $a$, and $l^s_+(v)$ the complementary ray. We orient $l^s$ with $a$ as the negative direction. The projected ray $\pi(l^s_+(v))$ intersects $\gamma_0:=\pi(\gamma)$ at infinitely many distinct points. For any $\gamma_0'$ homotopic to $\gamma_0$ in the cylindrical leaf $\pi(L)\subset M$, $\pi(l^s_+(v))$ intersects $\gamma_0'$ infinitely often.

\begin{figure}[htb]
    \centering
    \includegraphics{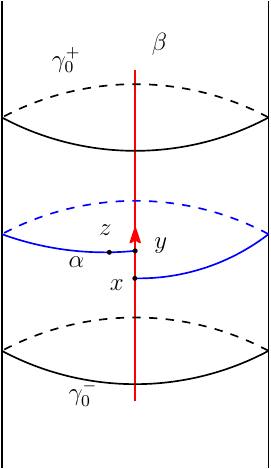}
    \caption{Stable ray $l^s_+(x)$ intersects only one of $\gamma_0^+$ or $\gamma_0^-$ on the cylinder}
    \label{case2}
\end{figure}

The infinite intersection set $l^s_+(v)\cap\gamma_0$ accumulates along $\gamma_0$. For any $\epsilon>0$, there exist $x,z\in l^s_+(v)\cap\gamma_0$ within an $\epsilon$-neighborhood on $\pi(L)$. Following Case I's methodology, connect $x$ and $z$ via a stable curve and a local unstable segment through $x$ to form a non-contractible closed curve $\alpha\subset\pi(L)$. Let $y\in\alpha$ be the intersection of $W^u_{loc}(x)$ and $W^s_{loc}(z)$ (see Figure~\ref{case2}).

The cylindrical topology forces $\alpha$ to be homotopic to $\gamma_0$. Let $\beta$ be the unstable curve through $x$, oriented such that $y$ follows $x$. Consider closed curves $\gamma_0^+,\gamma_0^-\subset\pi(L)$ homotopic to $\alpha$ on opposite sides. The stable ray $l^s_+(v)$ would need to intersect both $\gamma_0^+$ and $\gamma_0^-$ - an impossibility by orientability. This contradiction completes the proof.
\end{proof}

We emphasize that the absence of periodic points remains unused in the preceding proof. Furthermore, neither the completeness of $\mathcal{F}^{cs}$ or $\mathcal{F}^{cu}$, nor dynamical coherence constitutes a necessary condition. Specifically, dynamical coherence is only employed to preclude compact center stable leaves - a configuration that Theorem~\ref{Novikov} explicitly forbids by preventing the existence of transversals intersecting a center stable leaf multiple times.

\subsection{Non-uniform foliation}

A taut foliation $\mathcal{F}$ is called \emph{uniform} if in the lifted foliation $\widetilde{\mathcal{F}}$ on the universal cover $\widetilde{M}$, every pair of leaves remains at bounded Hausdorff distance. Formally, for any two leaves $L, F \in \widetilde{\mathcal{F}}$, there exists $k = k(L,F) < \infty$ such that $L$ is contained in the $k$-neighborhood of $F$ (and vice versa) with respect to Hausdorff distance. For any $\mathbb{R}$-covered codimension-one foliation - whether uniform or non-uniform - with hyperbolic leaves, there exists a vertical foliation in $\mathcal{A}$. The leaf space of this vertical foliation forms a circle, known as the \textit{universal circle} $\mathcal{U}$ in Thurston's framework, encoding all ideal circles of leaves. Through harmonic measure theory in closed manifolds of arbitrary dimension, Thurston demonstrated the prevalence of contracting directions in codimension-one foliations. This result was later topologically generalized by Calegari and Dunfield \cite{CalegariDunfield03}.

\begin{defn}[Contracting Direction]\label{def_contracting}
    Let $F$ be a leaf of $\widetilde{\mathcal{F}}$ containing a point $x$, and $\gamma$ a geodesic ray in $F$ starting at $x$ with initial tangent vector $v$. We say $\gamma$ (or $v$) is a \emph{contracting direction} if:
    \begin{itemize}
        \item There exists a transversal $\tau$ to $\widetilde{\mathcal{F}}$ through $x$;
        \item For every leaf $L \in \widetilde{\mathcal{F}}$ intersecting $\tau$, the distance $d(\gamma(t), L) \to 0$ as $t \to \infty$.
    \end{itemize}
    
    Similarly, $\gamma$ (or $v$) is called an \emph{$\epsilon$-non-expanding direction} if:
    \begin{itemize}
        \item There exists $\epsilon > 0$ and a transversal $\tau_\epsilon$ to $\widetilde{\mathcal{F}}$ through $x$;
        \item For every leaf $L \in \widetilde{\mathcal{F}}$ intersecting $\tau_\epsilon$, the distance $d(\gamma(t), L) \leq \epsilon$ for all $t \geq 0$.
    \end{itemize}
\end{defn}

We will utilize the following theorem in the non-uniform setting:
\begin{thm}\cite{Calegari00,Fenley02} \label{Thurston}
    Let $\mathcal{F}$ be a minimal, non-uniform $\mathbb{R}$-covered foliation with hyperbolic leaves. Given any point $x$ in a leaf $F$ of the lifted foliation $\widetilde{\mathcal{F}}$, there exists a dense set of contracting directions from $x$ to every other leaf $L \in \widetilde{\mathcal{F}}$.
\end{thm}

This theorem features stronger hypotheses and conclusions compared to Thurston's original result, specifically guaranteeing density of contracting directions between arbitrary leaf pairs. While Thurston's result holds in arbitrary dimensions, we specialize to 3-manifolds for simplicity.

\begin{thm}[Thurston]\label{Thurston_rmk}
    Let $\mathcal{F}$ be a codimension-one foliation with hyperbolic leaves in a closed 3-manifold $M^3$. For any $\epsilon > 0$ and leaf $F \in \widetilde{\mathcal{F}}$, the $\epsilon$-non-expanding directions are dense at every $x \in F$. If no holonomy-invariant transverse measure supported on $\pi(F)$ exists for any $F \in \widetilde{\mathcal{F}}$, then contracting directions become dense in $F$. Furthermore, when $\pi(F)$ is non-compact, the transversal $\tau$ from Definition~\ref{def_contracting} may be chosen with $x$ in its interior.
\end{thm}

We will use this theorem in the no holonomy-invariant transverse measure case. A proof can be found in \cite{CalegariDunfield03}.

Let $\tau_t: \widetilde{M} \to \widetilde{M}$ be a regulating flow for $\widetilde{\mathcal{F}}$. Through the topology of the cylindrical boundary $\mathcal{A}$, each contracting direction along $\tau_t$ generates a curve in $\mathcal{A}$.

\begin{defn}[Markers]\label{def_markers}
    Given a geodesic ray $\gamma$ associated with a contracting direction in a leaf $F \in \widetilde{\mathcal{F}}$ along $\tau_t$, define for any $L \in \widetilde{\mathcal{F}}$:
    \begin{itemize}
        \item A flow-projected curve $\gamma_L \subset L$ with unique ideal endpoint $a_L \in \partial_{\infty}L$;
        \item The \emph{marker} $m := \bigcup_{L \in \widetilde{\mathcal{F}}} a_L$.
    \end{itemize}
\end{defn}

\begin{rmk}\label{marker_properties}
    While $\gamma_L$ need not be geodesic in $L$, its geodesic curvature vanishes asymptotically near $\partial_{\infty}L$, ensuring unique ideal endpoints. Markers form embedded curves in $\mathcal{A}$ with pairwise disjointness or coincidence \cite{Fenley02,Calegari00}.
\end{rmk}

Markers exhibit continuity \cite{Fenley02}. To elucidate: Given $F \in \widetilde{\mathcal{F}}$ and $\xi \in \partial_{\infty}F$, density of contracting directions yields sequences $\{\alpha_i\}, \{\beta_i\} \subset \partial_{\infty}F$ bounding shrinking intervals $(\alpha_i, \beta_i) \ni \xi$. Associated markers $m_{\alpha_i}, m_{\beta_i}$ connect to leaves $L_i \in \widetilde{\mathcal{F}}$ accumulating at $F$, forming rectangles $R_i \subset \mathcal{A}$ with:
\begin{itemize}
    \item Horizontal edges: $(\alpha_i, \beta_i) \subset \partial_{\infty}F$ and $(\partial_{\infty}L_i \cap m_{\alpha_i}, \partial_{\infty}L_i \cap m_{\beta_i})$;
    \item Vertical edges: Marker segments between $\partial_{\infty}F$ and $\partial_{\infty}L_i$.
\end{itemize}
Continuity manifests as $\lim_{i\to\infty} R_i = \xi$ in $\mathcal{A}$'s topology.

Beyond the continuity of geodesic ideal points, we establish continuity for ideal points of rays in $\widetilde{\Lambda}^{s}_{L_i}$ (or $\widetilde{\Lambda}^{u}_{L_i}$) along sequences $L_i \in \widetilde{\mathcal{F}}^{su}$. This continuity holds transversely to $\widetilde{\mathcal{F}}^{su}$, though intrinsic continuity within individual leaves $F \in \widetilde{\mathcal{F}}^{su}$ fails in general (Lemma~\ref{continuity}). The following result, originally proved in \cite{Fenley09} for foliations transverse to almost pseudo-Anosov flows, remains valid in our context:

\begin{lem}\label{transverse_continuity}
    Ideal points of $\widetilde{\Lambda}^s_F$ vary continuously in $\mathcal{A}$ as $F$ ranges through $\widetilde{\mathcal{F}}^{su}$.
\end{lem}

\begin{proof}
    Let $F \in \widetilde{\mathcal{F}}^{su}$ contain a ray $l \in \widetilde{\Lambda}^s_F$ initiating at $p \in F$. By Proposition~\ref{singlelimit}, $l$ terminates at a unique $u \in \partial_{\infty}F$. Theorem~\ref{Thurston} provides contracting directions $r_1, r_2$ from $p$ with ideal points arbitrarily close to $u$. For any $L \in \widetilde{\mathcal{F}}^{su}$, the center projections $\widetilde{W^c}(r_1)$, $\widetilde{W^c}(r_2)$, and $\widetilde{W^c}(l)$ intersect $L$ as curves $s_1$, $s_2$, and $l'$ respectively (non-empty by Proposition~\ref{>1point}). Here $l' \in \widetilde{\Lambda}^s_L$ terminates at $v \in \partial_{\infty}L$.

    The critical observation is $v$'s containment in the interval $J \subset \partial_{\infty}L$ bounded by $v_1 := \partial_{\infty}s_1$ and $v_2 := \partial_{\infty}s_2$. To verify this, consider:
    \begin{itemize}
        \item A bounded domain $D \subset F$ where $l \setminus D$ lies in a wedge $W$ between $r_1$ and $r_2$;
        \item The closure $V := \overline{W \setminus D}$;
        \item A corresponding region $U \subset L$ bounded by center projections of $\partial V$.
    \end{itemize}

    Let $l_0$ be $l$'s final intersection with $D$, and $l_0' := \widetilde{\mathcal{F}}^c(l_0) \cap L$. Should $l'$ escape $U$ beyond $l_0'$, it would recross $\partial U$ - forcing $l$ to similarly recross $\partial V$ through center holonomy, contradicting $l_0$'s maximality. Thus $l' \subset U$ beyond $l_0'$, proving continuous variation of ideal points in $\mathcal{A}$.
\end{proof}

This result does not hold for general foliations. It depends crucially on the existence of a transverse flow preserving the stable foliation structure.

The following proposition is proved in \cite{Calegari00, Fenley09} for general $\mathbb{R}$-covered foliations in an atoroidal 3-manifold. However, the manifolds we consider in this paper are not necessarily atoroidal.

\begin{prop}\label{nonuniform}
    Let $\mathcal{F}^{su}$ be a non-uniform foliation containing a leaf $L \in \widetilde{\mathcal{F}}^{su}$ where the ideal points of $\widetilde{\Lambda}^s_L$ or $\widetilde{\Lambda}^u_L$ are non-dense in $\partial_{\infty}L$. Then $\mathcal{F}^{su}$ is the stable foliation of a flow topologically conjugate to a suspension Anosov flow, and $M$ is a torus bundle over $\mathbb{S}^1$.
\end{prop}

\begin{proof}
    We address the case of non-dense $\widetilde{\Lambda}^s_L$; the unstable case is analogous. Let $I \subset \partial_{\infty}L$ be an open interval disjoint from the ideal limit set of $\widetilde{\Lambda}^s_L$ rays. Fix $a \in I$ and choose $\{p_i\}_{i\in\mathbb{N}} \subset L$ converging to $a$, with associated wedges $W_{p_i}(I) \subset L$ whose complementary angles vanish as $i\to\infty$. 

    Up to a subsequence, select deck transformations $\{g_i\}$ such that $g_i(p_i) \to p_0$ in a leaf $F \in \widetilde{\mathcal{F}}^{su}$ with non-trival $\pi_1(F)$ (via Lemma~\ref{distinguished}). Let $W_i \subset L$ be the $p_i$-centered wedges bounded by $\partial_{\infty}L \setminus I$. The transformed wedges $g_i(W_i)$ converge to a geodesic ray in $F$ initiating at $p_0$ and terminating at $a_F \in \partial_{\infty}F$.

    Suppose $\exists x \neq a_F \in \partial_{\infty}F$ as a limit of some $r \in \widetilde{\Lambda}^s_F$. Center projection $\widetilde{\mathcal{F}}^c(r)$ intersects $g_i(L)$ in rays $r' \in \widetilde{\Lambda}^s_{g_i(L)}$ with endpoints $x' \in \partial_{\infty}g_i(L)$. By Lemma~\ref{transverse_continuity}, $x'$ approximates $x$ but must lie in $g_i(\partial_{\infty}L \setminus I)$ - a contradiction. Thus $a_F$ is fixed under $\pi_1(F)$.

    Each $F \in \widetilde{\mathcal{F}}^{su}$ acquires a distinguished ideal point $a_F$, forming a $\pi_1(M)$-invariant curve $\zeta \subset \mathcal{A}$ via continuity (Lemma~\ref{transverse_continuity}).

\medskip
\textbf{Claim 1:} The curve $\zeta$ is disjoint from all markers in $\mathcal{A}$.
\begin{proof}
    Proposition~\ref{transversemeasure} establishes that $\mathcal{F}^{su}$ cannot admit a transverse invariant measure unless $M$ contains periodic points or has (virtually) solvable fundamental group. Theorem~\ref{Thurston} guarantees dense contracting directions in every leaf of $\widetilde{\mathcal{F}}^{su}$. Each point on $\zeta$ arises as a limit of $g_i(\partial_{\infty}L \setminus I)$ under deck transformations $(g_i)_{i\in\mathbb{N}}$. The density of contracting directions ensures markers intersect each $g_i(\partial_{\infty}L \setminus I)$, yielding marker curves $m_t$ converging pointwise to $\zeta$ in $\mathcal{A}$.

    Suppose $\zeta$ intersects a marker $m$ at a leaf boundary ideal point. Marker continuity forces $m_t \to m$ locally, making $\zeta$ coincide with $m$ through minimality and $\pi_1(M)$-invariance. For $F \in \widetilde{\mathcal{F}}^{su}$ with nontrivial deck transformation $g$ (axis $\gamma$), consider $h(F)$ for $h \in \pi_1(M)$. The $\pi_1(M)$-invariance fixes $h(a_F)$ under deck transformations on $h(F)$, making it an endpoint of $h(\gamma)$. If $\zeta=m$, then $\gamma$ and $h(\gamma)$ would be asymptotic in $\widetilde{M}$. However, distinct lifts of the closed curve $\pi(\gamma) \subset M$ maintain positive minimal separation - a contradiction.
\end{proof}

Thus, $\zeta$ remains disjoint from all contracting directions.

\medskip
\textbf{Claim 2:} Every direction not aligned with $\zeta$ contracts between arbitrary leaves.
\begin{proof}
    As demonstrated in \cite[Proposition 3.21]{Fenley02} using $\mathcal{F}^{su}$'s non-uniformity: For any geodesic $\gamma \subset E \in \widetilde{\mathcal{F}}^{su}$ with endpoints $a_E \in \zeta$ and $b \notin \zeta$, if $b$-direction isn't contracting, then $a_E$-direction must contract. This fact joint with Claim 1 implies our claim here.
\end{proof}

Consequently, $\mathcal{A}$ decomposes into a 1-dimensional foliation comprising $\zeta$ and all markers.

\medskip
\textbf{Claim 3:} There exists a topological Anosov flow with $\mathcal{F}^{su}$ as stable foliation.
\begin{proof}
    We begin by constructing a topological Anosov flow. For each leaf \( F \in \widetilde{\mathcal{F}}^{su} \), observe that all geodesics converging to the distinguished ideal point \( a_F \) constitute a one-dimensional foliation within \( F \). Define \( \widetilde{\phi} \) as the unit-speed flow along these \( a_F \)-directed geodesics. This construction ensures smooth flow lines within individual leaves while maintaining continuous variation both within leaves and across \( \widetilde{\mathcal{F}}^{su} \). The continuity within leaves follows immediately from geodesic continuity. For inter-leaf continuity, consider convergent sequences \( p_i \in F_i \to p \in F \): geodesics from \( p_i\) to \(a_{F_i} \) accumulate to the one from \( p\) to \(a_F \) through \( \mathcal{A} \)-topology continuity of ideal points. Thus \( \widetilde{\phi} \) constitutes a well-defined flow on \( \widetilde{M} \).

    The flow respects deck transformations: For any \( g \in \pi_1(M) \) and \( F \in \widetilde{\mathcal{F}}^{su} \), \( g \)-images of \( a_F \)-directed geodesics in \( F \) become \( a_{g(F)} \)-directed geodesics in \( g(F) \), preserved by \( \zeta \)'s \( \pi_1(M) \)-invariance. Since \( g \) acts isometrically, \( \widetilde{\phi}_t(g(x)) = g(\widetilde{\phi}_t(x)) \) holds for all \( x \in F \), \( t \in \mathbb{R} \), making \( \widetilde{\phi} \) equivariant. This induces a flow \( \phi \) on \( M \) tangent to \( \mathcal{F}^{su} \), with smooth transversely continuous flow lines.

    The stable foliation emerges naturally: Each \( F \in \widetilde{\mathcal{F}}^{su} \) serves as a stable leaf for \( \widetilde{\phi} \), as all flow lines accumulate forward at \( a_F \). Consequently, \( \widetilde{\mathcal{F}}^{su} \) becomes the stable foliation of \( \widetilde{\phi} \), making \( \mathcal{F}^{su} \) the stable foliation of \( \phi \).

    For the unstable foliation: Given \( p \in F \in \widetilde{\mathcal{F}}^{su} \), consider the geodesic from \( p\) to \(a_F \) with backward ideal point \( m_F \). The \( m_F \)-direction contracts by Claim 2. Using center leaves \( \widetilde{\mathcal{F}}^c \) as transversals, Corollary \ref{=1point} ensures this contracting direction defines a marker intersecting each \( \partial_{\infty}E \) uniquely at \( m_E \). Flow lines from \( m_E \) to \( a_E \) under \( \widetilde{\phi} \) backward-asymptote to the orbit of \( p \) in \( F \). The union \( V_p\) of geodesics connecting $m_E$ and $a_E$ for all $E\in \widetilde{\mathcal{F}}^{su}$ forms a 2-surface in \( \widetilde{M} \). These pairwise disjoint or identical surfaces constitute \( \widetilde{\phi} \)'s unstable foliation, projecting to \( \phi \)'s unstable foliation on \( M \). Thus \( \phi \) is a topological Anosov flow.
\end{proof}

Observe that every stable leaf of $\widetilde{\phi}$ intersects each unstable leaf in $\widetilde{M}$ along a geodesic terminating at a distinguished ideal point. Under this geometric condition, Fenley \cite[Proposition 3.7]{Fenley09} establishes that for any deck transformation $g$ preserving $F \in \widetilde{\mathcal{F}}^{su}$, $F$ must be the unique $g$-invariant leaf. Let $\alpha \subset F$ be the $g$-invariant geodesic. Any other flow line in the $\widetilde{\phi}$-unstable leaf through $\alpha$ becomes backward-asymptotic to $\alpha$ and thus non-$g$-invariant. 

Should a distinct $g$-invariant leaf $E \in \widetilde{\mathcal{F}}^{su}$ exist with $g$-axis $\beta \subset E$, the unstable leaf through $\beta$ would intersect $F$, creating another $g$-invariant geodesic asymptotic to $\alpha$ - a contradiction. This structure forces $\pi_1(M)$ to be solvable per \cite{Barbot95etds}. Following Plante's argument \cite{Plante83solvable}, $\phi$ becomes topologically conjugate to a suspension Anosov flow, making $M$ a torus bundle over $\mathbb{S}^1$. This completes the proposition's proof.
\end{proof}

We formulate a general principle from the preceding arguments, independent of periodic point considerations.

\begin{cor}\label{onedense_alldense}
    Let $\mathcal{F}$ be a codimension-one minimal foliation with hyperbolic leaves in a closed 3-manifold, where each leaf $L \in \widetilde{\mathcal{F}}$ on the universal cover $\widetilde{M}$ is assigned a subfoliation $\widetilde{\Lambda}_L$. Suppose every ray in any $\widetilde{\Lambda}_L$-leaf converges to a unique ideal point in $\partial_{\infty}L$. If the limit set of $\widetilde{\Lambda}_F$ becomes dense in $\partial_{\infty}F$ for some leaf $F \in \widetilde{\mathcal{F}}$, then this density property propagates to all leaves: the limit set of $\widetilde{\Lambda}_L$ remains dense throughout $\partial_{\infty}L$ for every $L \in \widetilde{\mathcal{F}}$.
\end{cor}

\subsection{Uniform foliation}

We focus on the uniform case, demonstrating that for any leaf $L \in \widetilde{\mathcal{F}}^{su}$, the ideal limit set of $\widetilde{\Lambda}^s_L$ becomes dense, as shown in Proposition~\ref{uniform}.

Consider first a fundamental property of $\mathbb{R}$-covered foliations:

\begin{lem}\cite{Fenley92}\label{QI}
    Let $\mathcal{F}$ be an $\mathbb{R}$-covered foliation. For any $a_0 > 0$, there exists $a_1 = \phi(a_0) > 0$ through a function $\phi: \mathbb{R} \to \mathbb{R}$, such that whenever $x, y \in \widetilde{M}$ lie in the same leaf $L \in \widetilde{\mathcal{F}}$ with $d(x, y) \leq a_0$, their leafwise distance satisfies $d_L(x, y) \leq a_1$.
\end{lem}

Crucially, $a_1$ depends solely on $a_0$ rather than specific points or leaves. This characteristic equivalently defines $\mathbb{R}$-covered foliations lacking compact leaves.

\begin{lem}\label{quasi-isometry}
    Assume $\mathcal{F}^{su}$ is a uniform foliation. Then for any two leaves $E, F \in \widetilde{\mathcal{F}}^{su}$ in the universal cover $\widetilde{M}$, the lengths of center curves connecting $E$ and $F$ are uniformly bounded. Furthermore, the center holonomy mapping between $E$ and $F$ constitutes a quasi-isometry.
\end{lem}

\begin{proof}
    For any pair of leaves $E$ and $F$ in the foliation $\widetilde{\mathcal{F}}^{su}$, their Hausdorff distance is uniformly bounded. Let $a := d_H(E, F) > 0$ denote this Hausdorff distance. By compactness of the manifold $M$, there exists a connected compact subset $K \subset \widetilde{M}$ on the universal cover satisfying $\pi(K) = M$, where $\pi: \widetilde{M} \to M$ represents the covering projection.

    By Corollary~\ref{=1point}, every center leaf in $\widetilde{\mathcal{F}}^c$ intersects each leaf of $\widetilde{\mathcal{F}}^{su}$ exactly once, making each center leaf homeomorphic to both the leaf space of $\widetilde{\mathcal{F}}^{su}$ and $\mathbb{R}$.

For a center leaf $\alpha \in \widetilde{\mathcal{F}}^c$, a total order $<$ can be defined on $\alpha$ through the transverse orientation of $\widetilde{\mathcal{F}}^{su}$, which induces an ordering on the leaf space of $\widetilde{\mathcal{F}}^{su}$.

Projecting along $\widetilde{\mathcal{F}}^{su}$-leaves onto $\alpha$ identifies each leaf with a unique point on $\alpha$. Under this projection, the compact set $K$ maps to a compact interval in $\alpha$. This projection preserves monotonicity: if $A < B$ on $\alpha$, then $B$'s corresponding leaf succeeds $A$'s in the order. In the Hausdorff topology, let $B_H(K,a)$ denote the closed ball containing all points within distance $\leq a$. The projection of $B_H(K,a)$ forms a compact interval $I \subset \alpha$, with $I^+$ and $I^-$ representing its maximal and minimal endpoints under the defined order.

For any point on $\alpha$ exceeding $I^+$, the corresponding leaf lies outside $B_H(K, a)$, with its Hausdorff distance from $K$'s leaves strictly exceeding $a$. Conversely, if a leaf intersects $B_H(K, a)$, it projects within $I$. Analogously, every point below $I^-$ corresponds to leaves strictly separated from $K$ by distance $a$. Thus, all leaves within Hausdorff distance $a$ from $K$ project into the interval $I$.

For any arbitrary point $p \in E$, there exists an element $g \in \pi_1(M)$ such that $g(p) \in K$. Since $E$ and $F$ have Hausdorff distance $a$ and $g$ acts isometrically, $g(F)$ must intersect $B_H(K, a)$ and project to a point within $I$. Consequently, $g(F)$ necessarily resides between $g(E)$ and either $\widetilde{\mathcal{F}}^{su}(I^+)$ or $\widetilde{\mathcal{F}}^{su}(I^-)$. Without loss of generality, assume $g(E) < g(F) < \widetilde{\mathcal{F}}^{su}(I^+)$.  

By Corollary~\ref{=1point}, the center leaf through $g(p)$ intersects both $g(F)$ and $\widetilde{\mathcal{F}}^{su}(I^+)$ at unique points, denoted respectively as $h(g(p))$ and $l(g(p))$. The center segment from $g(p)$ to $h(g(p))$ is shorter than that from $g(p)$ to $l(g(p))$. The mapping from $K$ to $\widetilde{\mathcal{F}}^{su}(I^+)$ via center curves defines a continuous length function, which attains a finite maximum $C$ due to $K$'s compactness.  

Thus, the center curve length from $g(p)$ to $g(F)$—and equivalently from $p$ to $F$—is bounded by $C$. As $p$ is arbitrary and $g$ exists for every $p \in E$, the uniform bound $C$ (independent of $p$ and $g$) ensures all center curves between $E$ and $F$ have uniformly bounded lengths.

Let $h$ denote the center holonomy from $E$ to $F$ - a map $h: E \to F$ defined by $h(x) = \widetilde{\mathcal{F}}^c(x) \cap F$, which is well-defined and bijective by Corollary~\ref{=1point}. For any $x, y \in E$, construct a sequence $\{x_i\}_{i=0}^{k+1}$ along the geodesic segment in $E$ between $x$ and $y$, where $x_0 = x$, $x_{k+1} = y$, $d_E(x_i, x_{i+1}) = 1$ for $0 \leq i \leq k-1$, and $d_E(x_k, y) \leq 1$. Here $k \in \mathbb{N}$ satisfies $d_E(x, y) \in [k, k+1)$.

The bound $d(h(x_i), h(x_{i+1})) \leq 2C + 1$ holds for $0 \leq i \leq k$ since center curves between $E$ and $F$ have lengths bounded by $C$. By Lemma~\ref{QI}, $d_F(h(x_i), h(x_{i+1})) \leq \phi(2C+1) =: a_1$ for each $i$. Summing these estimates gives:
$$
d_F(h(x), h(y)) \leq a_1(k + 1) \leq a_1(d_E(x, y) + 1) = a_1d_E(x, y) + a_1.
$$

Define the inverse holonomy $h': F \to E$ by $h'(p) = \widetilde{\mathcal{F}}^c(p) \cap E$. The bijectivity of $h$ and $h'$ implies $h' \circ h(x) = x$ for all $x \in E$. Applying parallel reasoning to $h'$ yields:
$$
d_E(x, y) \leq a_1d_F(h(x), h(y)) + a_1.
$$

Combining these inequalities produces the quasi-isometric bounds:
$$
\frac{1}{a_1}d_E(x, y) - 1 \leq d_F(h(x), h(y)) \leq a_1d_E(x, y) + a_1.
$$
Therefore, the center holonomy $h$ constitutes a quasi-isometry.
\end{proof}

\begin{rmk}
    If the foliation $\mathcal{F}^{su}$ possesses a holonomy-invariant transverse measure, it must be uniform. The subsequent proposition demonstrates that when a leaf $F \in \widetilde{\mathcal{F}}^{su}$ exists where the ideal limit set of $\widetilde{\Lambda}^s_F$ fails to be dense in $\partial_{\infty}F$, this configuration excludes uniformity for $\mathcal{F}^{su}$. Consequently, $\mathcal{F}^{su}$ must be non-uniform and lacks any holonomy-invariant transverse measure.
\end{rmk}

\begin{prop}\label{uniform}
    If there exists a leaf \( L \in \widetilde{\mathcal{F}}^{su} \) where the ideal points of \( \widetilde{\Lambda}^s_L \) or \( \widetilde{\Lambda}^u_L \) fail to be dense in \( \partial_{\infty}L \), then the foliation \( \mathcal{F}^{su} \) is non-uniform.	
\end{prop}

\begin{proof}

    As established by Proposition~\ref{>1point} and Corollary~\ref{=1point}, each leaf in $\widetilde{\mathcal{F}}^c$ intersects every leaf of $\widetilde{\mathcal{F}}^{su}$ transversely at precisely one point. Assigning an orientation to $\widetilde{\mathcal{F}}^c$, we define a center flow $\psi: \mathbb{R} \times \widetilde{M} \to \widetilde{M}$ where for any $x \in \widetilde{M}$, $\psi_t(x)$ denotes the unique point on $\widetilde{\mathcal{F}}^c(x)$ at distance $|t|$ from $x$ along the leaf.

    For any pair of leaves $E, F \in \widetilde{\mathcal{F}}^{su}$, define the center holonomy $h(E,F): E \to F$ by $h(x) = \widetilde{\mathcal{F}}^c(x) \cap F$. Lemma~\ref{quasi-isometry} establishes $h$ as a quasi-isometry. This map extends to a boundary homeomorphism $h: \partial_{\infty}E \to \partial_{\infty}F$ (as shown in \cite{Gromov87}, \cite{Thurston88}), which we retain as $h(E,F)$ for simplicity.

    We now construct a vertical foliation on $\mathcal{A}$ to facilitate building a universal circle. 

For a fixed leaf $F \in \widetilde{\mathcal{F}}^{su}$ and ideal point $x \in \partial_{\infty}F$, define 
$$
\sigma_x := \bigcup_{L \in \widetilde{\mathcal{F}}^{su}} h(F,L)(x)
$$ 
as a curve in the cylinder at infinity $\mathcal{A}$. Observe that for any $L \in \widetilde{\mathcal{F}}^{su}$, the image $h(F,L)(x)$ constitutes a single point in $\partial_{\infty}L$, implying $\sigma_x$ intersects every circle-at-infinity uniquely. 

Since $h(F,L): \partial_{\infty}F \to \partial_{\infty}L$ is a homeomorphism for any $F,L \in \widetilde{\mathcal{F}}^{su}$, distinct curves $\sigma_{x_1}$ and $\sigma_{x_2}$ ($x_1 \neq x_2$ in $\partial_{\infty}F$) remain disjoint. Furthermore, every $y \in \partial_{\infty}L$ corresponds to a unique $x \in \partial_{\infty}F$ satisfying $h(F,L)(x) = y$. This establishes the disjoint union:
$$
\bigcup_{x \in \partial_{\infty}F} \sigma_x = \mathcal{A}.$$

We now demonstrate the continuity of each curve $\sigma_x$ on $\mathcal{A}$, which establishes $\bigcup_{x \in \partial_{\infty}F} \sigma_x$ as a vertical foliation of $\mathcal{A}$.  

Let $\alpha$ be a center leaf of $\widetilde{\mathcal{F}}^c$, and consider a sequence $p_i \in L_i \in \widetilde{\mathcal{F}}^{su}$ converging along $\alpha$ to $p_0 \in L_0 \in \widetilde{\mathcal{F}}^{su}$. To verify continuity, we show $y_i = h(L, L_i)(y)$ converges to $y_0 = h(L, L_0)(y)$ in $\mathcal{A}$ for every $y \in \partial_{\infty}L$.  

Assume for contradiction that some subsequence $y_i \to z_0 \in \partial_{\infty}L_0$ with $z_0 \neq y_0$. By Lemma~\ref{quasi-isometry}, the maps $h(L_i, L_0)$ are uniform quasi-isometries because all $L_i$ and $L_0$ lie between fixed leaves $E, F \in \widetilde{\mathcal{F}}^{su}$. Let $l_i$ denote the geodesic ray from $p_i$ to $y_i$ in $L_i$. Under $h(L_i, L_0)$, each $l_i$ maps to a quasi-geodesic in $L_0$ connecting $p_0$ to $y_0$, remaining within bounded Hausdorff distance of $l_0$ (see Figure~\ref{fig:uniformcase}).  

The quasi-isometric property of $h(L_i, L_0)$ ensures uniform boundedness of $d_H(l_i, l_0)$ via the triangle inequality. Consequently, the limit geodesic of $l_i$ on $L_0$ must asymptotically align with $l_0$ - a contradiction. Thus, $\sigma_x$ is continuous.  

\begin{figure}[htb]	
		\centering
		\includegraphics{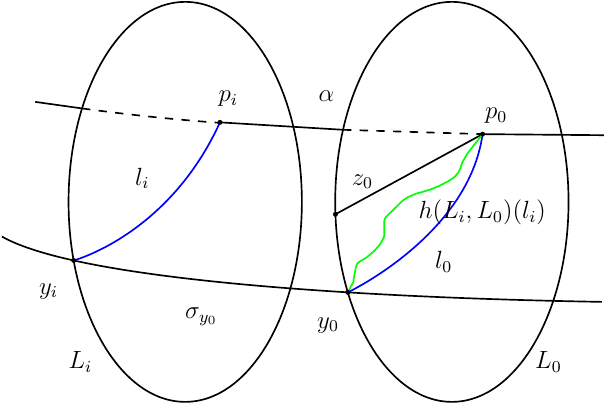}\\
		\caption{A discontinous curve}
		\label{fig:uniformcase}
	\end{figure}
    
Therefore, the union $\bigcup_{x \in \partial_{\infty}F} \sigma_x$ forms a $\pi_1(M)$-invariant vertical foliation of $\mathcal{A}$, transverse to the foliation $\bigcup_{F \in \widetilde{\mathcal{F}}^{su}} \partial_{\infty}F$ on $\mathcal{A}$. Collapsing the leaves of this vertical foliation defines a universal circle $\mathcal{U} = \mathcal{A}/{\sim}$, where $x \in \partial_{\infty}F$ is equivalent to $y \in \partial_{\infty}L$ if $y = h(F,L)(x)$. This circle $\mathcal{U}$ inherits $\pi_1(M)$-invariance, as the group action preserves inter-leaf distances and quasi-geodesic distances. Furthermore, $\mathcal{U}$ can be canonically identified with any circle-at-infinity $\partial_{\infty}L$ for $L \in \widetilde{\mathcal{F}}^{su}$.

    As demonstrated in Theorem 1.2 of \cite{FP20minimal}, the fundamental group $\pi_1(M)$ acts minimally on the universal circle $\mathcal{U}$, meaning every orbit lies dense in $\mathcal{U}$. For each element $g \in \pi_1(M)$, its action on $\mathcal{U}$ corresponds precisely to the action $\rho_L(g) := h(g(L), L) \circ g$ on $\partial_{\infty}L$ for all $L \in \widetilde{\mathcal{F}}^{su}$. Consequently, $\rho_L$ induces a minimal action on $\partial_{\infty}L$ for every leaf $L \in \widetilde{\mathcal{F}}^{su}$.

    Take a leaf \( L \in \widetilde{\mathcal{F}}^{su} \) where the limit points of \( \widetilde{\Lambda}^s_L \) fail to be dense in \( \partial_{\infty}L \). For every leaf \( l_s \in \widetilde{\Lambda}^s_L \) and each \( g \in \pi_1(M) \), the image \( g(l_s) \) becomes a leaf of \( \widetilde{\Lambda}^s_{g(L)} \). The map \( h(g(L), L) \) functions as a projection along the center foliation onto \( L \). Due to the completeness of \( \widetilde{\mathcal{F}}^{cs} \) and Corollary~\ref{=1point}, each leaf of \( \widetilde{\mathcal{F}}^{cs} \) intersects every leaf of \( \widetilde{\mathcal{F}}^{su} \) in precisely one stable leaf. Consequently, every center-stable leaf of \( \widetilde{\mathcal{F}}^{cs} \) remains invariant under \( h(g(L), L) \). Specifically,
\[
h(g(L), L) \circ g(l_s) = \widetilde{\mathcal{F}}^c(g(l_s)) \cap L
\]
also constitutes a leaf of \( \widetilde{\Lambda}^s_L \). Thus, \( \rho_L(g) \) maps leaves of \( \widetilde{\Lambda}^s_L \) to leaves of \( \widetilde{\Lambda}^s_L \), implying the set of limit points of \( \widetilde{\Lambda}^s_L \) is invariant under \( \rho_L(g) \) for all \( g \in \pi_1(M) \). The minimality of the action leads directly to a contradiction.

This contradiction demonstrates that \( \widetilde{\mathcal{F}}^{su} \) cannot be uniform except when the ideal limit points of \( \widetilde{\Lambda}^s_L \) are dense in \( \partial_{\infty}L \) for every \( L \in \widetilde{\mathcal{F}}^{su} \).
\end{proof}


\section{Dense Limit Set}\label{denselimit}

Throughout this section, we work under the following hypothesis: The limit set of $\widetilde{\Lambda}^s_F$ is dense in $\partial_{\infty}F$ for all $F \in \widetilde{\mathcal{F}}^{su}$.  

In the previous section, we established the continuity of ideal points for stable rays across distinct leaves of $\widetilde{\mathcal{F}}^{su}$. However, this continuity fails to extend to ideal points of $\widetilde{\Lambda}^s_F$ within individual leaves $F \in \widetilde{\mathcal{F}}^{su}$. The following lemma asserts that continuous variation of ideal points for stable rays in a leaf $F \in \widetilde{\mathcal{F}}^{su}$ holds provided the leaf space of $\widetilde{\Lambda}^s_F$ satisfies the Hausdorff condition (see \cite[Lemma 6.5]{Fenley09}).

Given a transverse orientation of $\mathcal{\widetilde{F}}^{cu}$, we canonically orient each leaf of $\widetilde{\Lambda}^s_F$ for all $F \in \widetilde{\mathcal{F}}^{su}$. For any $p \in F$, let $l^s(p)$ denote the $\widetilde{\Lambda}^s_F$-leaf through $p$, with rays $l^s_+(p)$ and $l^s_-(p)$ originating from $p$. The ideal points $a_+(p), a_-(p) \in \partial_{\infty}F$ represent the endpoints of $l^s_+(p)$ and $l^s_-(p)$ respectively.

\begin{lem}\label{continuity}
    If the leaf space of $\widetilde{\Lambda}^s_F$ satisfies the Hausdorff condition for some $F \in \widetilde{\mathcal{F}}^{su}$, then the ideal points $a_+(p)$ and $a_-(p)$ vary continuously as $p$ ranges over $F$.
\end{lem}

\begin{figure}[htb]	
	\centering
	\includegraphics[scale=0.8]{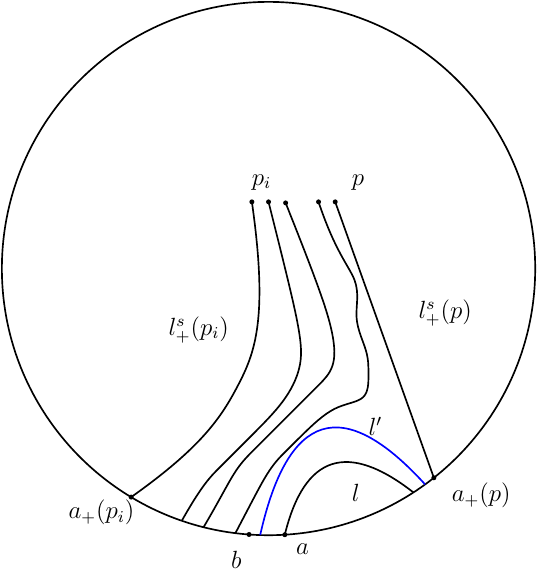}\\
	\caption{The non-continuity of ideal points}
	\label{fig:continuity}
\end{figure}

\begin{proof}
    Let $(p_i)_i$ be a sequence of points in $F \in \widetilde{\mathcal{F}}^{su}$ converging to $p \in F$. Suppose $a_+(p_i)$ fails to converge to $a_+(p)$. By passing to a subsequence, assume $a_+(p_i) \to b \in \partial_{\infty}F$ with $b \neq a_+(p)$. Define closed ideal intervals $I_i \subset \partial_{\infty}F$ bounded by $a_+(p_i)$ and $a_+(p)$, weakly nested and shrinking as $i \to \infty$ (allowing coincident endpoints for some $a_+(p_i)$). This yields a limiting closed ideal interval $I \subset \partial_{\infty}F$ with endpoints $a_+(p)$ and $b$.  

Density of $\widetilde{\Lambda}^s_F$'s limit set in $\partial_{\infty}F$ guarantees leaves of $\widetilde{\Lambda}^s_F$ with ideal points in $I$. We assert there exists a leaf $l \in \widetilde{\Lambda}^s_F$, distinct from $l^s(p)$, accumulated by the rays $l^s_+(p_i)$ and containing an ideal point in $I$. To verify, select any leaf $l$ with ideal point $a \in I$.

If the ray \( l^s_+(p_i) \) fails to accumulate at \( l \), stable foliation continuity ensures another leaf \( l' \in \widetilde{\Lambda}^s_F \) with an ideal point between \( a \) and \( b \) (see Figure~\ref{fig:continuity}). This \( l' \) separates \( l^s_+(p_i) \) from \( l \) for all \( i \in \mathbb{N} \), while not separating \( l^s_+(p_i) \) from \( l^s(p) \). If the leaf \(l'\) is not accumulated by \(l^s_+(p_i)\), repeat the process and finally we can find a limit leaf \(l_0\) being accumulated by \( l^s_+(p_i) \) as \( i \to \infty \). Observe that \( l^s_+(p_i) \) and \( l \) reside in the same half-plane of \( F \) bisected by \( l^s(p) \). Consequently, \( l^s(p) \) and \( l_0 \) become distinct leaves in \( F \), simultaneously accumulated by \( l^s(p_i) \), contradicting the Hausdorff property of \( \widetilde{\Lambda}^s_F \)'s leaf space.
\end{proof}

Observe that the lemma's proof does not utilize the absence of periodic points. The argument solely requires that every ray in a leaf of $\widetilde{\Lambda}^s_L$ converges to a unique ideal point in $\partial_{\infty}L$, combined with the density of $\widetilde{\Lambda}^s_L$'s limit set.

\subsection{Uniform distance from geodesics to curves}

We examine interactions between leaves of $\widetilde{\Lambda}^s_F$ and geodesics within $F$ for every leaf $F \in \widetilde{\mathcal{F}}^{su}$, conditional on $\widetilde{\Lambda}^s_F$ possessing a dense ideal limit set.

\begin{lem}\label{uniformbound}
    For any $F \in \widetilde{\mathcal{F}}^{su}$, let $l$ be a stable segment with endpoints $a$ and $b$, and let $l^*$ denote the geodesic segment connecting $a$ and $b$. There exists a constant $K_1 > 0$, independent of both $F$ and $l$, such that $l^*$ lies within the $K_1$-neighborhood of $l$. Furthermore, the same conclusion holds if $l$ is a ray or a leaf of $\widetilde{\Lambda}^s_F$, with $l^*$ being the corresponding geodesic ray or geodesic.
\end{lem}

\begin{proof}

    Assume the statement is false. Then for every $i \in \mathbb{N}$, there exist a leaf $F_i \in \widetilde{\mathcal{F}}^{su}$, a segment $l_i$ in a leaf of $\widetilde{\Lambda}^s_{F_i}$, and a corresponding geodesic segment $l^*_i$ connecting the endpoints of $l_i$, such that $l^*_i$ fails to be contained in the $i$-neighborhood of $l_i$. This implies the existence of a point $p_i \in l^*_i$ where the hyperbolic metric disk $B(p_i, i)$ - centered at $p_i$ with radius $i$ - remains disjoint from $l_i$. Let $a_i$ and $b_i$ denote the endpoints of $l_i$.

\begin{figure}[htb]
    \centering
    \includegraphics{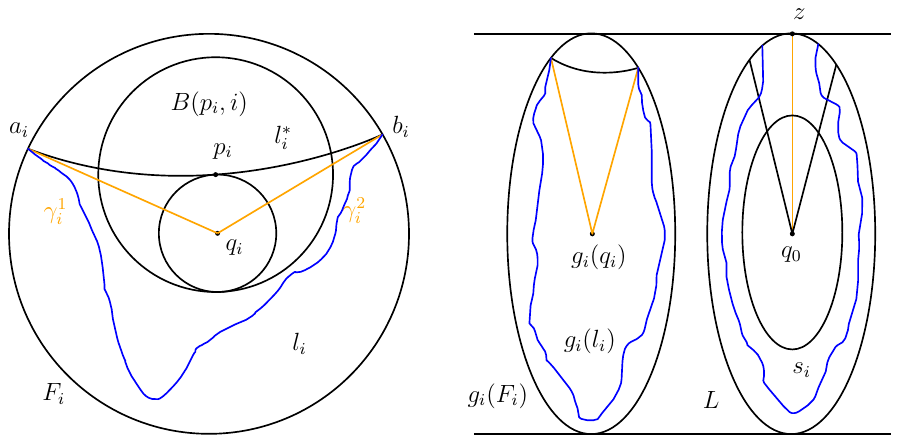}
    \caption{Geodesic segment $l^*_i$ escapes uniform neighborhoods of stable leaf $l_i$}
    \label{K_1}
\end{figure}

Select a sequence of points $q_i$ between $l_i$ and $l^*_i$ such that each $q_i$ lies in the interior of $B(p_i,i)$ and the disk $B(q_i, \frac{i}{2})$ is tangent to $B(p_i,i)$; see Figure~\ref{K_1}. By construction, $B(q_i, \frac{i}{2})$ remains disjoint from $l_i$. Choose deck transformations $g_i$ such that, after passing to a subsequence, $g_i(q_i)$ converges to a point $q_0$ in a leaf $L \in \widetilde{\mathcal{F}}^{su}$. 

For sufficiently large $i$, the leaf $g_i(l_i)$ projects via the center foliation to a curve $s_i$ in $L$, where $s_i$ is itself a leaf of $\widetilde{\Lambda}^s_L$. Since each $g_i(l_i)$ lies outside the disk $B(g_i(q_i), \frac{i}{2})$, continuity ensures that $s_i$ remains outside $B(q_0, \frac{i}{3})$ for large $i$. The choices of $p_i$ and $q_i$, combined with the isometric nature of deck transformations, ensure that the lengths of geodesic segments from $g_i(q_i)$ to $g_i(p_i)$ tend to infinity as $i \to \infty$.

Let $\gamma^1_i$ and $\gamma^2_i$ denote the geodesic rays connecting $q_i$ to $a_i$ and $b_i$ respectively. Define $W_i$ as the wedge bounded by $\gamma^1_i$ and $\gamma^2_i$, containing the geodesic $l^*_i$. As $i \to \infty$, the geodesic segment from $g_i(q_i)$ to $g_i(p_i)$ in $g_i(F_i)$ converges to a geodesic ray $\gamma$ in $L$ initiating at $q_0$. Let $\angle_{g_i(q_i)}(g_i(a_i), g_i(b_i))$ represent the angle at $g_i(q_i)$ formed by the wedge $g_i(W_i)$. 

Since the distance between $g_i(p_i)$ and $g_i(q_i)$ tends to infinity as $i \to \infty$, the angle $\angle_{g_i(q_i)}(g_i(a_i), g_i(b_i))$ approaches zero. Consequently, the wedge $g_i(W_i)$ collapses to the geodesic ray $\gamma$ in $L$. Let $z \in \partial_{\infty}L$ denote the ideal endpoint of $\gamma$. Observing that the boundary geodesics $\gamma^1_i$ and $\gamma^2_i$ of $W_i$ connect the endpoints $a_i$ and $b_i$ of $l_i$ respectively, we note that as the radius of $B(g_i(q_i), \frac{i}{2})$ tends to infinity, the endpoints of $g_i(l_i)$ converge to $z \in \partial_{\infty}L$.

Construct a narrow wedge $W$ at $q_0$ in $L$ whose ideal interval $I \subset \partial_{\infty}L$ contains $z$ as an interior point. By continuity, there exists $N \in \mathbb{N}$ such that for all $i \geq N$, the endpoints of $s_i$ lie in the interval $I$. Given that the radius of $B(q_0, \frac{i}{3})$ increases without bound, we may always select a leaf of $\widetilde{\Lambda}^s_L$ that:
1. Resides entirely outside the $R$-radius disk about $q_0$ for arbitrarily large $R \geq 0$;
2. Has its ideal endpoints contained in the interval $I$.

This implies that the limit set of $\widetilde{\Lambda}^s_L$ leaves is confined to the narrow interval $I \subset \partial_{\infty}L$, contradicting the density of the ideal points of leaves in $\partial_{\infty}L$.

An analogous argument applies when either endpoint of $l$ lies at an ideal point in $\partial_{\infty}F$. Therefore, $l$ may be taken as either a stable ray or a full stable leaf of $\widetilde{\Lambda}^s_F$, with $l^*$ representing the associated geodesic ray or complete geodesic respectively.
\end{proof}

We emphasize that the converse containment remains undetermined - it is not yet established whether $l$ lies within any uniform neighborhood of $l^*$.

\subsection{Distinct ideal points}

In this subsection, we establish that each stable curve possesses distinct ideal points, as shown in Proposition~\ref{twodistinctiidealpoint}. First, we characterize a necessary condition for the existence of stable curves with coinciding ideal points.

\begin{lem}\label{nonHausdorffness}
    If there exists a leaf of $\widetilde{\Lambda}^s_F$ possessing a single ideal point for some $F \in \widetilde{\mathcal{F}}^{su}$, then the leaf space of $\widetilde{\Lambda}^s_F$ within $F$ becomes non-Hausdorff. Furthermore, the leaf space of $\widetilde{\mathcal{F}}^{cs}$ consequently fails to be Hausdorff.
\end{lem}

\begin{proof}

    Given a leaf $F \in \widetilde{\mathcal{F}}^{su}$, let $l^+$ and $l^-$ denote the ideal points of a leaf $l \in \widetilde{\Lambda}^s_F$. Assume $l^+ = l^- =: z \in \partial_{\infty}F$. Choose a basepoint $x_0 \in l$, which divides the leaf into two rays $l_1$ and $l_2$ emanating from $x_0$ and both asymptotic to $z$. Let $l^*_0$ be the geodesic ray from $x_0$ to $z$. By Lemma~\ref{uniformbound}, there exists a uniform constant $K_1 > 0$ such that $l^*_0$ lies within the $K_1$-neighborhood of both $l_1$ and $l_2$.

For any sequence $\{x_i\}$ in $l^*_0$ approaching $z$ as $i \to \infty$, there exist corresponding sequences $\{p_i\}$ in $l_1$ and $\{q_i\}$ in $l_2$ converging to $z$ with:
\[
d_F(p_i, x_i) \leq K_1 \quad \text{and} \quad d_F(q_i, x_i) \leq K_1
\]
for all $i \in \mathbb{N}$. This implies $d_F(p_i, q_i) \leq 2K_1$ uniformly. Through careful selection of deck transformations $\{g_i\}$, we may assume (after passing to a subsequence) that $g_i(p_i) \to p_0$ and $g_i(q_i) \to q_0$ within a common leaf $L \in \widetilde{\mathcal{F}}^{su}$, while maintaining $d_L(p_0, q_0) \leq 2K_1$.

\begin{figure}[htb]
		\centering
		\subcaptionbox{Single ideal point}
		{\includegraphics{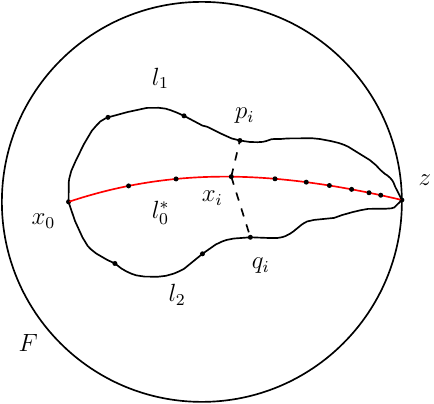}}
		\subcaptionbox{Non-Hausdorffness of $\widetilde{\Lambda}^s_L$}
		{\includegraphics{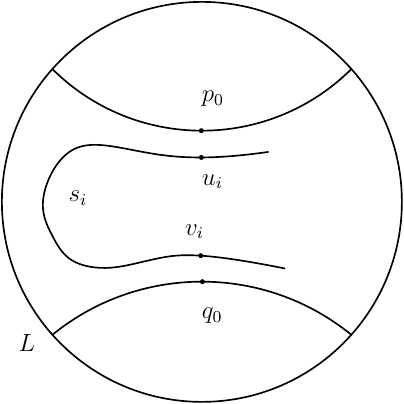}}
        \caption{Proof of Lemma \ref{nonHausdorffness}}
        \label{fig:singlepoint}
\end{figure}

    For sufficiently large $i$, the points $g_i(p_i)$ and $g_i(q_i)$ project via center leaves to points $u_i$ and $v_i$ in $L$. The segment of $g_i(l)$ between $g_i(p_i)$ and $g_i(q_i)$ projects via central leaves to a segment $s_i$ of a leaf in $\widetilde{\Lambda}^s_L$ connecting $u_i$ and $v_i$. Analysis reveals that $u_i \to p_0$ and $v_i \to q_0$ as $i \to \infty$; see Figure~\ref{fig:singlepoint}. 

This convergence precludes $p_0$ and $q_0$ from lying on the same leaf of $\widetilde{\Lambda}^s_L$: While their separation $d_L(p_0, q_0)$ remains bounded, any hypothetical shared leaf segment $s_0$ between them would inherit this bounded length. By convergence, $s_i$ would then also exhibit bounded length for large $i$. However, direct computation contradicts this - the length of $s_i$ diverges because the lengths of the $g_i(l)$ segments between $g_i(p_i)$ and $g_i(q_i)$ go to infinity. This contradiction forces $p_0$ and $q_0$ to reside on distinct leaves of $\widetilde{\Lambda}^s_L$.

Within leaf $L$, we observe a sequence of leaves in $\widetilde{\Lambda}^s_L$ converging to distinct leaves of $\widetilde{\Lambda}^s_L$, demonstrating that the leaf space of $\widetilde{\Lambda}^s_L$ in $L$ fails to be Hausdorff. Furthermore, projecting the leaves containing $p_0$ and $q_0$ along center foliation leaves onto $g_i(F)$ reveals two distinct leaves in $\widetilde{\Lambda}^s_{g_i(F)}$ that are simultaneously accumulated by a sequence of $\widetilde{\Lambda}^s_{g_i(F)}$ leaves. This demonstrates that both the leaf space of $\widetilde{\Lambda}^s_{g_i(F)}$ and consequently the leaf space of $\widetilde{\Lambda}^s_F$ are non-Hausdorff. The non-Hausdorff property for $\mathcal{F}^{cs}$ follows from the fundamental property that the center foliation $\widetilde{\mathcal{F}}^c$ regulates $\widetilde{\mathcal{F}}^{su}$.
\end{proof}

While the leaves containing $p_0$ and $q_0$ represent potential limits, the sequence $\{g_i(l)\}$ might additionally accumulate at other leaves of $\widetilde{\Lambda}^s_L$. We now establish a technical lemma essential for proving Proposition~\ref{twodistinctiidealpoint}.

\begin{lem}\label{unbounded}
    Let $F \in \widetilde{\mathcal{F}}^{cs}$ be a leaf where the stable subfoliation $\widetilde{\mathcal{F}}^s_F$ possesses a Hausdorff leaf space. Suppose there exists a non-trivial deck transformation $h$ preserving $F$. Given sequences $(x_i)_{i\in\mathbb{N}}$ in $F$, distinct integers $(n_i)_{i\in\mathbb{N}}$, and stable leaves $l_i \in \widetilde{\mathcal{F}}^s_F$ through $x_i$, if $h^{n_i}(x_0)$ lies in $l_i$ for all $i \in \mathbb{N}$, then the sequence $(x_i)_{i\in\mathbb{N}}$ must be unbounded.
\end{lem}

\begin{proof}
    By assumption, the leaf space $\mathcal{H}$ of $\widetilde{\mathcal{F}}^s_F$ is homeomorphic to $\mathbb{R}$. Endowing $\mathcal{H}$ with an appropriate order $<$, the nonexistence of $h$-invariant leaves in $\widetilde{\mathcal{F}}^s_F$ forces $h$ to act as a strictly increasing transformation under this ordering. Formally, this requires $\tau < h(\tau)$ for every $\tau \in \mathcal{H}$.

    Assume contrary that $(x_i)_{i\in\mathbb{N}}$ is bounded, which possesses a convergent subsequence. Without loss of generality, assume $(x_i)_{i\in\mathbb{N}}$ converges to some $x \in F$. Let $l \in \widetilde{\mathcal{F}}^s_F$ be the stable leaf through $x$. Leaf space continuity ensures $l_i \to l$ in $\mathcal{H}$. Passing to a subsequence, $h^{n_i}(l_0)$ converges monotonically to $l$ in $\mathcal{H}$ as $i \to \infty$.

For monotone increasing $n_i$, the relation $h^{n_i+1}(l_0) \leq h^{n_{i+1}}(l_0)$ holds universally. Thus,
$$h(l) = \lim_{i\to\infty} h^{n_i+1}(l_0) \leq \lim_{i\to\infty} h^{n_{i+1}}(l_0) = l.$$
This inequality necessitates either $l = h(l)$ or non-separation in $\mathcal{H}$ - both impossible under the Hausdorff condition. 

For decreasing $n_i$, we similarly obtain $h^{n_i+1}(l_0) \leq h^{n_{i-1}}(l_0)$, leading to:
$$h(l) = \lim_{i\to\infty} h^{n_i+1}(l_0) \leq \lim_{i\to\infty} h^{n_{i-1}}(l_0) = l,$$
again contradicting order preservation.
\end{proof}

The following proposition establishes that no leaf in $\widetilde{\Lambda}^s_F$ for any $F \in \widetilde{\mathcal{F}}^{su}$ possesses a single ideal point. Our argument adapts techniques from \cite{Fenley09}, where stable spike regions and Reeb annuli in leaves were studied, though Fenley's original arguments relied on pseudo-Anosov flow existence.

\begin{prop}\label{twodistinctiidealpoint}
    For any leaf $F \in \widetilde{\mathcal{F}}^{su}$ and $l \in \widetilde{\Lambda}^s_F$, the stable leaf $l$ possesses exactly two distinct ideal points.
\end{prop}

\begin{proof}
    Let $l^+$ and $l^-$ denote the ideal points of a leaf $l \in \widetilde{\Lambda}^s_F$ for some $F \in \widetilde{\mathcal{F}}^{su}$. Suppose there exists a leaf $l \in \widetilde{\Lambda}^s_F$ with identical ideal points $z := l^+ = l^- \in \partial_{\infty}F$. Following Lemma~\ref{nonHausdorffness}, fix $x_0 \in l$ splitting it into rays $l_1$ and $l_2$. By Lemma~\ref{uniformbound}, there exist sequences $p_i \in l_1$ and $q_i \in l_2$ satisfying $d_F(p_i, q_i) \leq 2K_1$ for all $i \in \mathbb{N}$, where $K_1$ is the uniform constant from Lemma~\ref{uniformbound}.

Select deck transformations $g_i$ (after subsequence extraction) such that $g_i(p_i) \to p_0$ and $g_i(q_i) \to q_0$ within a common leaf $L \in \widetilde{\mathcal{F}}^{su}$. Let $s, r \in \widetilde{\Lambda}^s_L$ denote distinct leaves containing $p_0$ and $q_0$ respectively. Define $\alpha_i$ as leaves in $L$ obtained by projecting $g_i(l)$ along $\widetilde{\mathcal{F}}^c$, with $u_i, v_i \in \alpha_i$ located on center leaves through $g_i(p_i)$ and $g_i(q_i)$ respectively - well-defined due to $\widetilde{\mathcal{F}}^c$ regulating $\widetilde{\mathcal{F}}^{su}$ (see Figure~\ref{fig:nonHaus1}). 

Observe $u_i \to p_0$ and $v_i \to q_0$ as $i \to \infty$. For large $i$, each $\alpha_i$ satisfies $\alpha_i^+ = \alpha_i^- \in \partial_{\infty}L$ by continuous variation of $g_i(l)$'s ideal points. The path distance between $u_i$ and $v_i$ along $\alpha_i$ tends to infinity while all $\alpha_i$ only intersect at $\partial_{\infty}L$, yielding a nested subsequence of $(\alpha_i)_{i=1}^\infty$ sharing a common ideal point $\xi$. After further subsequence selection, assume $\{\alpha_i\}$ forms a nested collection. Let $\beta_i \subset \alpha_i$ and $\gamma_i \subset \alpha_i$ denote rays from $u_i$ and $v_i$ respectively, both asymptotic to $\xi$.

	\begin{figure}[htb]	
		\centering
		\includegraphics{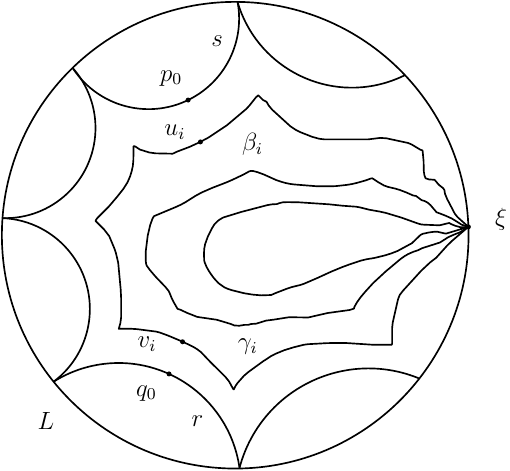}\\
		\caption{Non-Hausdorffness in the limit leaf $L$}
		\label{fig:nonHaus1}
	\end{figure}

    Let $C_1$ and $C_2$ denote collections of leaves in $\widetilde{\Lambda}^s_L$ accumulated by $\beta_i$ and $\gamma_i$ respectively. The inclusion $s \in C_1$ and $r \in C_2$ follows from the convergences $u_i \to p_0 \in s$ and $v_i \to q_0 \in r$. Define $C := C_1 \cup C_2$. We exclude leaves $c \in C$ satisfying $c^+ = c^-$, as these do not impact our conclusion. Consequently, we may assume each leaf in $C$ possesses two distinct ideal points, with any pair sharing at most one ideal point.

The cardinality of $C_1$ and $C_2$ could be finite or infinite. In either case, $C_1$ admits an ordering - finite or isomorphic to $\mathbb{N}$ - and similarly for $C_2$.

Indeed, this order arises naturally through transversality. Each leaf in $C_1$ accumulates points from $(\beta_i)$ as $i \to \infty$. By transverse intersection properties, the unstable manifold through any such accumulation point must intersect $\beta_i$ for sufficiently large $i$. As every $\beta_i$ constitutes a 1-dimensional ray asymptotic to $\xi \in \partial_{\infty}L$, we fix an orientation where the direction from $u_i$ to $\xi$ represents the positive orientation. This induces a natural order on $\beta_i$: For $x_i, y_i \in \beta_i$, declare $x_i < y_i$ if $y_i$ lies positively downstream from $x_i$. For leaves $A, B \in C_1$, define $A < B$ when there exist points $a_0 \in A$, $b_0 \in B$, and an index $i \gg 0$ where their local unstable manifolds intersect $\beta_i$ at:
\[
a_i := \widetilde{W}^u_{\text{loc}}(a_0) \cap \beta_i, \quad b_i := \widetilde{W}^u_{\text{loc}}(b_0) \cap \beta_i
\]  
satisfying $a_i < b_i$ in the $\beta_i$-order. Lemma~\ref{at-most-once} ensures unique intersections since unstable leaves intersect stable rays at most once. Crucially, this order remains consistent across indices $i$ due to coherent $\beta_i$ orientations and it is independent of representative points $a_0$, $b_0$ by leaf uniformity. Thus $C_1$ admits a total order - either finite or order-isomorphic to $\mathbb{N}$ - through this transverse intersection hierarchy.

Observe that points on $\gamma_i$ converge to $\xi$ with reverse orientation compared to $\beta_i$, as both rays belong to the same leaf $\alpha_i$. Consequently, $C_2$ requires an inverse ordering. Designate the direction from $v_i$ to $\xi$ as the negative orientation of $\gamma_i$ for all $i \in \mathbb{N}$. For $x_i, y_i \in \gamma_i$, declare $x_i < y_i$ if $x_i$ precedes $y_i$ in the negative-oriented sequence. This dual structure induces:  
- $C_1$ ordered as $\mathbb{N}$ or finite positives  
- $C_2$ ordered as $-\mathbb{N}$ or finite negatives.
The combined order on $C = C_1 \cup C_2$ satisfies $c_1 > c_2$ for all $c_1 \in C_1$, $c_2 \in C_2$, completing the total ordering.

For finite $C_1$, let $s'$ denote the maximal leaf in $C_1$ and $\xi_1$ its ideal point where no other leaves in $C_1$ share $\xi_1$. When $C_1$ is infinite, define $\xi_1$ as the limit of ideal points from $C_1$'s ordered leaves. Analogously, construct $\xi_2$ for $C_2$ through this dual procedure.

We will establish $\xi_1 = \xi_2$ by their mutual convergence to $\xi$ - the common ideal point of all $\alpha_i$ from Proposition~\ref{twodistinctiidealpoint}'s construction. Thus $\xi_1 = \xi_2 = \xi$ follows inherently from the asymptotic coherence of nested stable leaves.

Assume $\xi_1$ and $\xi_2$ are distinct ideal points in $\partial_{\infty}L$. Select disjoint ideal intervals $I_1, I_2 \subset \partial_{\infty}L$ such that:
- $\xi_1 \in \text{int}(I_1)$, 
- $\xi_2 \in \text{int}(I_2)$,  
- $I_1 \cap I_2 = \emptyset$.
Fix a basepoint $x_0 \in L$. Construct a wedge $W_1$ by connecting $x_0$ to $I_1$'s boundary points via geodesics. Similarly, define $W_2$ as the geodesic wedge from $x_0$ to $I_2$'s endpoints.

When $C_1$ is finite, $\xi_1$ serves as an ideal point of the maximal leaf $s' \in C_1$. Since $s'$ accumulates $\beta_i$ sequences, every neighborhood of $s'$ contains terminal segments of $\beta_i$ for sufficiently large $i \in \mathbb{N}$. As $\xi_1$ represents $C_1$'s extremal ideal point, beyond some index $i_0$, the truncated rays $\beta_i|_{[u_i,\xi)}$ reside entirely within any prescribed neighborhood of $s'$. To formalize this, fix a disk $D = B(x_0, R) \subset L$ with radius $R > 0$ such that $s'$'s $\xi_1$-asymptotic ray lies in $W_1 \setminus D$ after crossing $\partial D$ (see Figure~\ref{fig:nonHaus2}). For $i \gg 0$, the containment $\beta_i \cap (W_1 \setminus D) \subset W'_1$ holds, where $W'_1 \subset W_1$ is a narrower wedge guaranteed by $\xi_1 \in \text{int}(I_1)$. Under center foliation projection, the ideal endpoint of $g_i(l_1)$ necessarily resides in $J_1$ - the boundary interval of $W'_1$'s associated asymptotic sector.

\begin{figure}[htb]	
    \centering
    \includegraphics[scale=1.2]{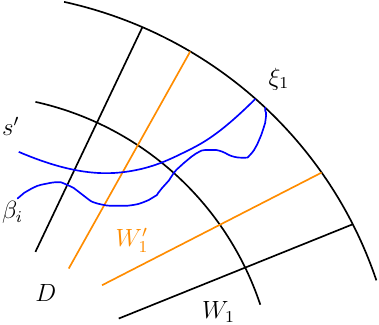}
    \caption{The curve $\beta_i$ remains confined in $W'_1 \setminus D$ after crossing $\partial D$}
    \label{fig:nonHaus2}
\end{figure}

For infinite $C_1$, let $\{s_j\}_{j\in\mathbb{N}}$ denote the ordered leaves whose ideal points converge to $\xi_1$. All leaves $s_j$ are distinct and ordered isomorphic to $\mathbb{N}$. There exists $j_0 \in \mathbb{N}$ such that for all $j \geq j_0$, $s_j \subset W_1$. As $\xi_1$ is the limit of these ideal points, terminal segments of $\beta_i$ for $i \gg 0$ lie entirely within $W_1 \setminus D$ and further concentrate in a nested wedge $W'_1 \subset W_1$. Through center foliation projection, the ideal endpoint of $g_i(l_1)$ necessarily resides in $J_1$ - the boundary interval of $W'_1$'s asymptotic sector.

An analogous argument applies to $\xi_2$, demonstrating the ideal point of $g_i(l_2)$ lies within a distinct small interval $J_2 \subset I_2$. The disjointness of $I_1$ and $I_2$ forces $J_1 \cap J_2 = \emptyset$. However, this contradicts the fundamental property that $g_i(l_1)$ and $g_i(l_2)$ share identical ideal points as lifts of the original leaf $l$ with $l^+=l^-=z$. This contradiction establishes $\xi_1 = \xi_2 = \xi$.

The collection $C$ manifests in two scenarios: finite or infinite cardinality. We examine these cases distinctly. 

First consider the infinite case, where at least one subset $C_1$ or $C_2$ contains infinitely many leaves. Consider sequences $\{s_j\}_{j\in\mathbb{N}} \subset C_1$ and $\{r_j\}_{j\in\mathbb{N}} \subset C_2$ whose ideal points converge to $\xi \in \partial_{\infty}L$. For finite $C_1$, assume $s_j = s_N$ for all $j \geq N$ (analogously for finite $C_2$). Crucially, the index $j \in \mathbb{N}$ does not reflect the total order in $C$, where $C_1$ elements dominate $C_2$. The sequence $s_j$ progresses upward in $C$'s order as $j$ increases, while $r_j$ descends due to their opposing convergence directions to $\xi$. This dichotomy arises from the nested leaves $\{\alpha_i\}_{i\in\mathbb{N}}$ sharing $\xi$ as their common ideal point, where:
- $\beta_i \subset \alpha_i$ converges to $C_1$ while ascending toward $\xi$;
- $\gamma_i \subset \alpha_i$ converges to $C_2$ while descending toward $\xi$.

Let $\tau$ be a geodesic with ideal endpoint $\xi$ that eventually separates all but finitely many $\beta_i$ and $\gamma_i$. By Lemma~\ref{uniformbound}, any $\xi$-asymptotic ray of $\tau$ starting at $\tau \cap \alpha_i$ remains within a uniform $K_1$-neighborhood of $\alpha_i$ for all $i \in \mathbb{N}$. Select a sequence $\{z_k\}_{k\in\mathbb{N}}$ along $\tau$ converging to $\xi$ that satisfies:
\begin{itemize}
    \item There are points $b_i(z_k) \in \beta_i$ and $c_i(z_k) \in \gamma_i$ with $d_L(z_k, b_i(z_k)) \leq K_1$ and $d_L(z_k, c_i(z_k)) \leq K_1$ for all $i, k \in \mathbb{N}$;
    \item Convergence $b_i(z_k) \to b(z_k) \in s_{j_k} \subset C_1$ and $c_i(z_k) \to c(z_k) \in r_{j_k} \subset C_2$ as $i \to \infty$;
    \item Monotonicity $j_k \nearrow \infty$ with $k$.
\end{itemize}
See Figure~\ref{fig:nonHaus3}. This configuration yields uniform bounds:
\[
d_L(b_i(z_k), c_i(z_k)) \leq 2K_1 \quad \forall i,k \in \mathbb{N}
\]
and consequently:
\[
d_L(b(z_k), c(z_k)) \leq 2K_1 \quad \forall k \in \mathbb{N}
\]
where $K_1$ is the universal constant from Lemma~\ref{uniformbound}. 

\begin{figure}[htb]
    \centering
    \includegraphics{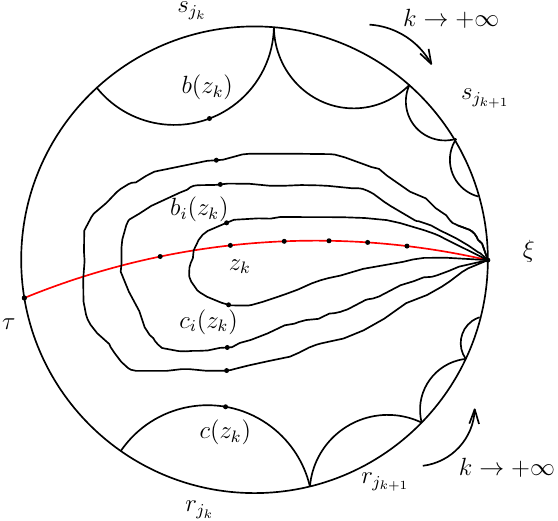}
    \caption{Convergence mechanism for infinite configuration $C$}
    \label{fig:nonHaus3}
\end{figure}

Project each point $z_k$ via the covering map $\pi \colon \widetilde{M} \to M$. After passing to a subsequence, assume $\pi(z_k)$ converges in $M$. Choose deck transformations $\{g_k\}$ such that $g_k(z_k) \to z_0$, and let $L_0 \in \widetilde{\mathcal{F}}^{su}$ denote the leaf containing $z_0$. As established in Lemma~\ref{distinguished}, $L_0$ serves as the $k \to \infty$ limit of $g_k(L)$. Given $d_L(z_k, b(z_k)) \leq K_1$ for all $k \in \mathbb{N}$, refine the subsequence further to ensure $g_k(b(z_k)) \to b(z_0) \in L_0$. Analogously, select a subsequence where $g_k(c(z_k)) \to c(z_0) \in L_0$. Both $b(z_0)$ and $c(z_0)$ lie within the $K_1$-radius disk centered at $z_0$ in $L_0$. For any $\epsilon > 0$, there exists $N_\epsilon \in \mathbb{N}$ such that for all $n \geq N_\epsilon$:
\[
d_{\widetilde{M}}(g_n(b(z_n)), b(z_0)) < \epsilon \quad \text{and} \quad d_{\widetilde{M}}(g_n(c(z_n)), c(z_0)) < \epsilon.
\]

Let $s_0 \in \widetilde{\Lambda}^s_{L_0}$ denote the stable leaf through $b(z_0)$, and $S_0 \in \widetilde{\mathcal{F}}^{cs}$ its associated center-stable leaf. Similarly, define $r_0 \in \widetilde{\Lambda}^s_{L_0}$ as the stable leaf through $c(z_0)$ with corresponding center-stable leaf $R_0 \in \widetilde{\mathcal{F}}^{cs}$. We claim that for sufficiently small $\epsilon > 0$, the inclusions $g_n(s_{j_n}) \subset S_0$ and $g_n(r_{j_n}) \subset R_0$ hold for all $n \geq N_\epsilon$.

Assume contrarily that for some \( n \geq N_{\epsilon} \), the intersection \( S_0 \cap g_n(L) \) contains a leaf \( s_0^n \in \widetilde{\Lambda}^s_{g_n(L)} \) distinct from \( g_n(s_{j_n}) \). Equivalently, \( S_n \cap L_0 \) contains a leaf \( s^0_{j_n} \in \widetilde{\Lambda}^s_{L_0} \) distinct from \( s_0 \), where \( S_n \in \widetilde{\mathcal{F}}^{cs} \) denotes the leaf containing \( g_n(s_{j_n}) \). We observe that $g_k(b(z_k)) \to b(z_0)$ as $k \to \infty$ and $g_k(b_i(z_k)) \to g_k(b(z_k))$ as $i \to \infty$.
 
The two half-planes in \( g_n(L) \), divided by the leaf \( g_n(s_{j_n}) \), must separately contain \( s_0^n \) and the collection \( \{g_n(\alpha_i)\}_{i\in\mathbb{N}} \). Indeed, if \( s_0^n \) resides in the same half-plane as some \( g_n(\alpha_i) \), it becomes confined within the band bounded by consecutive leaves \( g_n(\alpha_{i_0}) \) and \( g_n(\alpha_{i_0+1}) \) for some \( i_0 \in \mathbb{N} \). This forces \( s_0^n \) to possess only one ideal point \( g_n(\xi) \). Recall that \( C_1 \) comprises stable leaves with distinct ideal points. By Lemma~\ref{transverse_continuity}, the continuity of ideal points ensures \( s_0 \) (and consequently \( s_0^n \)) must have distinct ideal points. This is a contradiction.

By the diagonal principle, \( g_k(b_k(z_k)) \in g_k(\beta_k) \) converges to \( b(z_0) \in s_0 \). Projecting \( g_k(\alpha_k) \) along \( \widetilde{\mathcal{F}}^c \) yields leaves \( \alpha_k^0 \in \widetilde{\Lambda}^s_{L_0} \) accumulating on \( s_0 \). Suppose \( s_{j_n}^0 \) and \( \alpha_k^0 \) lie within the same \( s_0 \)-divided half-plane of \( L_0 \). The leaf \( s_{j_n}^0 \) must then reside in the band bounded by \( \alpha_{k_0}^0 \) and \( \alpha_{k_0+1}^0 \) for some \( k_0 \in \mathbb{N} \). The continuity in Lemma~\ref{transverse_continuity} yields a contradiction: While \( \alpha_{k_0}^0 \) and \( \alpha_{k_0+1}^0 \) share a common ideal point, \( s_{j_n}^0 \) maintains distinct ideal points. Consequently, \( s_0 \) separates \( s_{j_n}^0 \) from the sequence \( \alpha_k^0 \).

Observe that \( s_0, s_0^n \subset S_0 \) and \( g_n(s_{j_n}), s_{j_n}^0 \subset S_n \). This leads to a contradiction: Leaves of \( \widetilde{\mathcal{F}}^{cs} \) are properly embedded, and any leaf must reside entirely within one of the two half-spaces determined by another leaf. We therefore conclude \( s_0^n = g_n(s_{j_n}) \subset S_0 \) for all \( n \geq N_{\epsilon} \). Analogous argument applies to \( g_n(r_{j_n}) \subset R_0 \).

Define transformations \( H_i := g^{-1}_{N_\epsilon} \circ g_{N_\epsilon + i} \) for all \( i \in \mathbb{N} \). Both \( S_0 \) and \( R_0 \) remain \( H_i \)-invariant for every \( i \). By Lemma~\ref{cylin-plane}, every center-stable leaf constitutes either a cylinder or plane. Consequently, \( H_i = h^{n_i} \) for some non-trivial deck transformation $h$. Assume without loss that \( n_i > 0 \) for all \( i \in \mathbb{N} \).

Given the infinitude of \( C_1 \) or \( C_2 \), either \( s_{j_k} \) or \( r_{j_k} \) must be pairwise distinct. Suppose \( s_{j_k} \) are all distinct and $h$ is chosen to preserve $S_0$. Then for \( n > m \geq N_\epsilon \), the leaves \( g_m(s_{j_n}) \) and \( g_m(s_{j_m}) \) differ, as do their center-stable leaves. The transformation \( g_n \circ g_m^{-1} \) sends \( g_m(s_{j_n}) \) to \( g_n(s_{j_n}) \). Since \( g_n(s_{j_n}) \) and \( g_m(s_{j_m}) \) lie in \( S_0 \), \( g_n \circ g_m^{-1} \) is non-trivial. This forces \( H_i = h^{n_i} \) to be non-trivial deck transformations with distinct positive iterations \( n_i \), ensuring all \( H_i \) are distinct. 

Define \( x_i := g_{N_\epsilon + i}(b(z_{N_\epsilon + i})) \) and \( l_i := g_{N_\epsilon + i}(s_{j_{N_\epsilon + i}}) \) for each \( i \in \mathbb{N} \). Let \( x := b(z_0) \) and \( l := s_0 \), observing \( h^{n_i}(l_0) = l_i \). These constructions satisfy:
- \( x_i \) and \( l_i \) lie within the center-stable leaf \( S_0 \in \widetilde{\mathcal{F}}^{cs} \);
- \( S_0 \) remains invariant under the nontrivial deck transformation \( h \);
- The Hausdorff property of the \( \widetilde{\mathcal{F}}^{su} \) leaf space.

By Lemma~\ref{unbounded}, the sequence \( \{x_i\} \) cannot be bounded. This directly contradicts the established convergence \( g_i(b(z_i)) \to x \).

Consequently, both \( C_1 \) and \( C_2 \) must be finite. This finiteness guarantees the existence of boundary leaves \( s' \in C_1 \) and \( r' \in C_2 \) sharing \( \xi \) as a common ideal point in \( \partial_{\infty}L \).

\begin{figure}[htb]	
    \centering
    \includegraphics{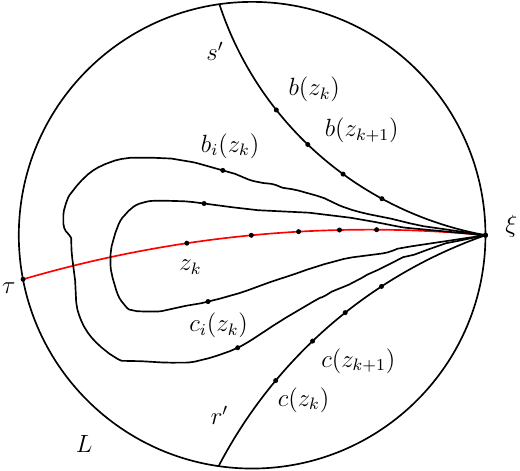}
    \caption{Convergence mechanism for two elements in $C$}
    \label{fig:nonHaus4}
\end{figure}

Consider a geodesic $\tau$ with ideal endpoint $\xi$ and a sequence $\{z_k\}_{k\in\mathbb{N}} \subset \tau$ where:
\begin{itemize}
    \item $z_k \to \xi$ as $k \to \infty$;
    \item For each $k$, there exist $b_i(z_k) \in \beta_i$ and $c_i(z_k) \in \gamma_i$ satisfying $d_L(z_k, b_i(z_k)) \leq K_1$ and $d_L(z_k, c_i(z_k)) \leq K_1$ for all $i \in \mathbb{N}$;
    \item The sequences $b_i(z_k) \to b(z_k) \in s'$ and $c_i(z_k) \to c(z_k) \in r'$ as $i \to \infty$. 
\end{itemize}
See Figure~\ref{fig:nonHaus4}. Furthermore, we deduce $d_L(b(z_k), c(z_k)) \leq 2K_1$ for all $k \in \mathbb{N}$, with both $b(z_k)$ and $c(z_k)$ converging to $\xi$ as $k \to \infty$.

Choose deck transformations $\{g_k\}$ such that (after subsequence extraction):
- $g_k(z_k) \to z^* \in L^*$;
- $g_k(L) \to L^* \in \widetilde{\mathcal{F}}^{su}$; 
- $g_k(b(z_k)) \to b(z^*) \in L^*$;
- $g_k(c(z_k)) \to c(z^*) \in L^*$.

Let $s^*, r^* \in \widetilde{\Lambda}^s_{L^*}$ denote the stable leaves through $b(z^*)$ and $c(z^*)$ respectively - these form a non-separated pair. Similarly to the infinite case, their associated center-stable leaves $S^*, R^* \in \widetilde{\mathcal{F}}^{cs}$ satisfy: For sufficiently small $\epsilon > 0$, there exists $N_\epsilon \in \mathbb{N}$ such that for all $n \geq N_\epsilon$:
\[
g_n(s') \subset S^* \quad \text{and} \quad g_n(r') \subset R^*.
\] The transformations $H_i := g^{-1}_{N_\epsilon} \circ g_{N_\epsilon + i}$ preserve both $S^*$ and $R^*$, maintaining their invariance under deck transformations.

There exists \( i \in \mathbb{N} \) for which \( H_i \) is non-trivial. Suppose otherwise: if \( H_i = \text{id} \) for all \( i \in \mathbb{N} \), then \( g_{N_\epsilon + i} = g_{N_\epsilon} \). This implies:
\[
L^* = \lim_{i \to \infty} g_i(L) = \lim_{i \to \infty} g_{N_\epsilon + i}(L) = g_{N_\epsilon}(L)
\]
and
\[
b(z^*) = \lim_{i \to \infty} g_i(b(z_i)) = \lim_{i \to \infty} g_{N_\epsilon}(b(z_{N_\epsilon + i})).
\]
Since \( b(z^*) \in L^* = g_{N_\epsilon}(L) \), there exists \( y \in L \) such that \( b(z^*) = g_{N_\epsilon}(y) \). The isometric property of \( g_{N_\epsilon} \) yields:
\[
y = \lim_{i \to \infty} b(z_{N_\epsilon + i}) = \xi \in \partial_\infty L,
\]
which is impossible as \( y \in L \) cannot equal an ideal point. This contradiction proves \( H_i \neq \text{id} \) for some \( i \).

Given that the fundamental group of each \( \mathcal{F}^{cs} \)-leaf has at most one generator, \( H_i = h^{n_i} \) for some non-trivial deck transformation \( h \) preserving \( S^* \).

We now establish the existence of a subsequence \((n_{i_m})_{m\in\mathbb{N}}\) of \((n_i)\) with distinct terms such that \(\lim_{m \to \infty} h^{n_{i_m}}(g_{N_\epsilon}(L)) = L^*\). Assume contrarily that some \(N_1 \in \mathbb{N}\) satisfies \(n_i = n_{N_1}\) for all \(i \geq N_1\). This implies \(H_i = h^{n_i} = h^{n_{N_1}} = H_{N_1}\) and consequently \(g_{N_\epsilon + i} = g_{N_\epsilon + N_1}\) for all \(i \geq N_1\). We arrive to a contradiction as above with $g_{N_{\epsilon}+N_1}$ playing the role of $g_{N_{\epsilon}}$.
This contradiction invalidates the initial assumption, proving the required subsequence exists. 

Furthermore, since \(\pi_1(\mathcal{F}^{cs})\)-leaves have cyclic fundamental groups, the non-trivial deck transformation \(h\) preserves \(S^*\), with \(h^{n_{i_m}}\) generating distinct transformations for the subsequence.

To establish the required contradiction via Lemma~\ref{unbounded}, we align notation as follows: Let \( x_{i_m} := g_{N_\epsilon + i_m}(b(z_{N_\epsilon + i_m})) \) and \( l_{i_m} := g_{N_\epsilon + i_m}(s') \), both residing entirely within the center-stable leaf \( S^* \in \widetilde{\mathcal{F}}^{cs} \). The nontrivial deck transformation \( h \) preserves \( S^* \), with distinct exponents \( \{n_{i_m}\}_{m\in\mathbb{N}} \). 

Critically, the Hausdorff leaf space of \( \widetilde{\mathcal{F}}^{cs} \) conflicts with the coexistence of:
1. The relation \( h^{n_{i_m}}(l_0) = l_{i_m} \);
2. The convergent sequence \( g_i(b(z_i)) \to b(z^*) \). Lemma~\ref{unbounded} explicitly prohibits such configurations, finalizing the contradiction. This inconsistency resolves the proposition's proof.
\end{proof}

\subsection{Quasi-geodesic stable curves}

A curve $l$ in a leaf $F \in \widetilde{\mathcal{F}}^{su}$ is called a \textit{quasi-geodesic} if there exist constants $\lambda > 1$ and $c > 0$, along with a parametrization $\chi \colon \mathbb{R} \cup \{\pm \infty\} \to F \cup \partial_{\infty}F$, satisfying:
\[
\lambda^{-1}|t - s| - c \leq d_l(\chi(t), \chi(s)) \leq \lambda|t - s| + c
\]
for all $t,s \in \mathbb{R}$, where $d_l(\chi(t), \chi(s))$ denotes the path length between $\chi(t)$ and $\chi(s)$. A family of curves is \textit{uniformly quasi-geodesic} if $\lambda$ and $c$ can be chosen independently for all curves in the family.

The Morse Lemma establishes that in hyperbolic spaces, every quasi-geodesic remains within a bounded distance from a unique geodesic with matching endpoints (see \cite{Gromov87} or \cite[Chapter III.H]{BH13metric}). This bound depends quantitatively on the quasi-geodesic parameters through a function $K(\lambda,c)$, ensuring every $(\lambda,c)$-quasi-geodesic lies within the $K(\lambda,c)$-neighborhood of its corresponding geodesic. 

An equivalent characterization requires a curve $l$ to be contained in the $K$-neighborhood of $l^*$ for the geodesic $l^*$ sharing $l$'s endpoints, where $K$ depends only on the quasi-geodesic constants. A family of curves is \textit{uniformly quasi-geodesic} if a single $K > 0$ satisfies this containment for all members simultaneously.

\begin{lem}\label{quasi-geodesic}
    For every leaf $F \in \widetilde{\mathcal{F}}^{su}$, the leaves of $\widetilde{\Lambda}^s_F$ constitute uniform quasi-geodesics. Furthermore, there exists a universal constant $K_0 > 0$, independent of $F$, such that every leaf of $\widetilde{\Lambda}^s_F$ and its corresponding geodesic sharing both ideal endpoints are mutually contained within $K_0$-neighborhoods of one another.
\end{lem}

\begin{proof}
    Lemma~\ref{uniformbound} establishes that for every leaf $F \in \widetilde{\mathcal{F}}^{su}$, there exists a universal constant $K_1 > 0$ - independent of both $F$ and $l \in \widetilde{\Lambda}^s_F$ - such that the associated geodesic $l^*$ (sharing $l$'s ideal points) satisfies $l^* \subset N_{K_1}(l)$, where $N_{K_1}(l)$ denotes the $K_1$-neighborhood of $l$. To complete the quasi-geodesic characterization, we now demonstrate the symmetric containment: There exists $K_2 > 0$, also independent of $F$ and $l$, for which $l \subset N_{K_2}(l^*)$. By the Morse Lemma, this mutual neighborhood equivalence will confirm the uniform quasi-geodesic property with $K_0 := \max\{K_1, K_2\}$.

    Assume, for contradiction, that no uniform constant \( K_2 > 0 \) exists. Then for every \( i \in \mathbb{N} \), there exist:
- A leaf \( F_i \in \widetilde{\mathcal{F}}^{su} \),
- A stable leaf \( l_i \in \widetilde{\Lambda}^s_{F_i} \),
- A geodesic \( l_i^* \subset F_i \) sharing \( l_i \)'s ideal points \( a_i, b_i \in \partial_\infty F_i \),
such that \( l_i \nsubseteq N_i(l_i^*) \). This guarantees a point \( p_i \in l_i \) where the hyperbolic metric ball \( D_i:= B(p_i, i) \subset F_i \) satisfies \( D_i \cap l_i^* = \emptyset \); see Figure~\ref{fig:K0}. Choose deck transformations \( \{g_i\} \) ensuring \( g_i(p_i) \to p_0 \) in some leaf \( L \in \widetilde{\mathcal{F}}^{su} \) as \( i \to \infty \). 

\begin{figure}[htb]	
    \centering
    \subcaptionbox{Stable leaf $l_i$ escapes uniform neighborhoods of its geodesic $l_i^*$}
    {\includegraphics{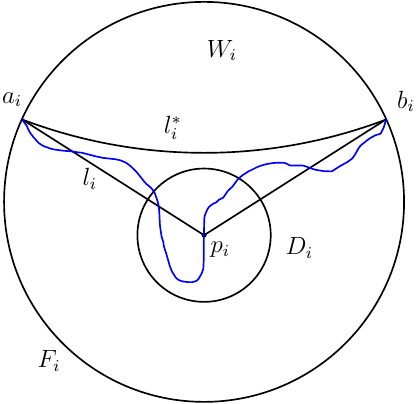}}
    \subcaptionbox{Nested ideal intervals under $V_i \cap V_j \neq \emptyset$}
    {\includegraphics{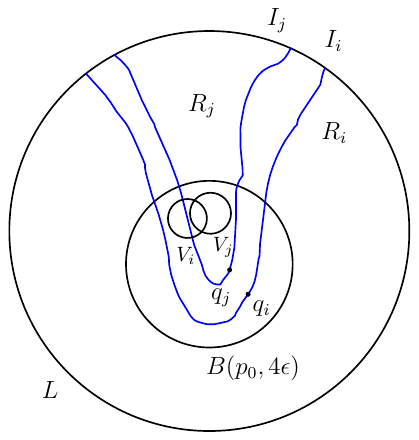}}
    \caption{Limit leaf $L$ contains stable leaf with single ideal point}
    \label{fig:K0}
\end{figure}

Denote by $l_0$ the leaf of $\widetilde{\Lambda}^s_L$ passing through $p_0$. Since deck transformations act as isometries, for every $i \in \mathbb{N}$, the geodesic $g_i(l^*_i)$ lies outside the disk $g_i(D_i)$ centered at $g_i(p_i)$ with radius $i$. Consider two geodesic rays $r_i^1$ and $r_i^2$ starting from $p_i$ and converging to the ideal points $a_i$ and $b_i$, respectively. Let $W_i$ be the wedge at $p_i$ bounded by $r_i^1$ and $r_i^2$ such that $W_i$ contains the geodesic $l^*_i$. Note that the distance between $g_i(p_i)$ and $g_i(l^*_i)$ approaches infinity as $i \to \infty$. Since $g_i(l_i^*)$ lies within the region $g_i(W_i) \setminus g_i(D_i)$ and connects the two boundary points of $g_i(W_i)$, the wedge $g_i(W_i)$ converges to a geodesic in $L$ as $i \to \infty$. 

For each $i \in \mathbb{N}$, the wedge $g_i(W_i)$ projects along the center foliation to the region $W'_i$ in $L$. Denote by $I_i \subset \partial_{\infty}L$ the ideal interval corresponding to the region $W'_i$. For each $i \in \mathbb{N}$, let $q_i$ be the intersection point of $\mathcal{\widetilde{F}}^c(g_i(p_i))$ and $L$. Denote by $s_i$ the curve in $L$ obtained by projecting $g_i(l_i)$ along the center foliation. Due to the completeness of $\mathcal{\widetilde{F}}^{cs}$, the curve $s_i$ is the leaf of $\widetilde{\Lambda}^s_L$ passing through $q_i$. As $i \to \infty$, the interval $I_i$ shrinks to an ideal point $\xi \in \partial_{\infty}L$. Without loss of generality, we assume that the convergence is monotonic.

Let $R_i \subset L \cup \partial_{\infty}L$ be the region bounded by $s_i$ and $I_i$. By the local product structure, there exists $\epsilon > 0$ such that if $s_i$ has two connected components within an $\epsilon$-ball for some $i$, then $s_i$ intersects an unstable curve at two distinct points. From Lemma~\ref{at-most-once}, we deduce that each $s_i$ contains at most one connected component within any $\epsilon$-ball.

Define $U_i := R_i \cap B(q_i, 3\epsilon)$, where $B(q_i, 3\epsilon)$ denotes the hyperbolic ball in $L$ centered at $q_i$ with radius $3\epsilon$. Since $U_i$ is connected, it contains at least one $\epsilon$-ball, denoted $V_i$. Observe that the sequence $(q_i)_{i\in\mathbb{N}}$ converges to $p_0$ as $i \to \infty$. There exists $N \in \mathbb{N}$ such that for all $i \geq N$, the point $q_i$ lies within $B(p_0, \epsilon)$. Consequently, for each $i \geq N$, both $B(q_i, 3\epsilon)$ and thus $V_i$ are contained in $B(p_0, 4\epsilon)$.

Given the compactness of $\overline{B(p_0, 4\epsilon)}$, there exist only finitely many pairwise disjoint $\epsilon$-balls within it. This implies the existence of $k \in \mathbb{N}$ such that any $(k+1)$ $\epsilon$-balls must contain at least two with non-empty intersection. We can therefore select a subsequence $N \leq i_1 < i_2 < \cdots$ where $V_{i_n} \cap V_{i_{n+1}} \neq \emptyset$ holds for every $n \in \mathbb{N}$.

For any indices \(i < j\) satisfying \(V_i \cap V_j \neq \emptyset\), it necessarily follows that \(R_i \cap R_j \neq \emptyset\). The regions \(R_i\) and \(R_j\) are respectively bounded by distinct leaves \(s_i\) and \(s_j\). Consequently, \(R_i\) and \(R_j\) must nest within one another, implying nested ideal boundaries \(I_i\) and \(I_j\). Given the monotonic decrease of the sequence \((I_i)_i\), we establish \(I_j \subset I_i\). By induction, the subsequence \((I_{i_n})_n\) consists of nested ideal intervals. The ideal point \(\xi\), defined as \(\lim_{n \to \infty} I_{i_n}\), lies within every \(I_{i_n}\). 

Observe that for each \(n \in \mathbb{N}\), the leaf \(s_{i_n}\) has two distinct ideal points forming the boundary of \(I_{i_n}\). This forces the limit leaf \(l_0\) of \((s_{i_n})_n\) to possess only one ideal point \(\xi\), directly contradicting Proposition~\ref{twodistinctiidealpoint}.

The resulting contradiction confirms the existence of the uniform constant \(K_2\). Defining \(K_0 := \max\{K_1, K_2\}\), we conclude that for every leaf \(F \in \widetilde{\mathcal{F}}^{su}\), each leaf \(l \in \widetilde{\Lambda}^s_F\) and its associated geodesic \(l^*\) are mutually contained within \(K_0\)-neighborhoods of one another.
\end{proof}

\subsection{Equivalence of Hausdorffness}

Next, we demonstrate that a one-dimensional foliation consisting of uniform quasi-geodesics on a hyperbolic plane possesses a leaf space homeomorphic to $\mathbb{R}$.

\begin{lem}\label{Hausdorff}
    The leaf space of $\widetilde{\Lambda}^s_F$ satisfies the Hausdorff separation axiom.
\end{lem}

\begin{proof}
    If the leaf space of $\widetilde{\Lambda}^s_F$ is not Hausdorff, there exist distinct leaves $l$ and $l'$ in $\widetilde{\Lambda}^s_F$ that are non-separated. Consider a sequence of leaves $l_n \in \widetilde{\Lambda}^s_F$ converging simultaneously to both $l$ and $l'$ as $n \to \infty$. Specifically, select points $x \in l$ and $y \in l'$ accumulated by sequences $\{x_n\}_{n\in\mathbb{N}} \subset l_n$ and $\{y_n\}_{n\in\mathbb{N}} \subset l_n$ respectively. Note that $x$ and $y$ are distinct points in $F$. Consequently, $d_F(x_n, y_n)$ remains bounded for sufficiently large $n \in \mathbb{N}$.

    Assume there exists a constant $K > 0$ such that the distance between $x_n$ and $y_n$ along $l_n$ is uniformly bounded by $K$ for all $n \in \mathbb{N}$. As $x_n \to x$, every segment of $l_n$ containing $x_n$ with uniformly bounded length converges to a segment of $l$ containing $x$. Analogously, segments containing $y_n$ converge to segments of $l'$ containing $y$. 

Consequently, the segment of $l_n$ connecting $x_n$ to $y_n$ must converge to either a segment of $l$ containing $x$ or a segment of $l'$ containing $y$. The uniform boundedness of these segments implies $x$ and $y$ lie on the same leaf of $\widetilde{\Lambda}^s_F$, with their separation along the leaf bounded by $K$. This contradicts the initial assumption that $l$ and $l'$ are distinct leaves of $\widetilde{\Lambda}^s_F$, establishing the required contradiction.

In conclusion, the distance along leaves $d_{l_n}(x_n, y_n)$ is unbounded and therefore tends to infinity. The segments of $l_n$ between $x_n$ and $y_n$ must escape every uniform neighborhood of the geodesic segments containing $x_n$ and $y_n$ as $n \to \infty$. Consequently, the leaves $l_n$ fail to be uniform quasi-geodesics. This contradiction with Lemma~\ref{quasi-geodesic} completes the proof.
\end{proof}

From the preceding analysis, we derive a rigid equivalence relation under the condition that $\widetilde{\Lambda}^s_F$ possesses densely distributed limit ideal points within $\partial_{\infty}F$ for every $F \in \widetilde{\mathcal{F}}^{su}$.

\begin{prop}\label{equivalent}
    The following statements are pairwise equivalent:
    \begin{enumerate}[(1)]
        \item The center-stable foliation $\mathcal{F}^{cs}$ is $\mathbb{R}$-covered;
        \item For every leaf $F \in \widetilde{\mathcal{F}}^{su}$, the leaf space of $\widetilde{\Lambda}^s_F$ within $F$ satisfies the Hausdorff property;
        \item For all $F \in \widetilde{\mathcal{F}}^{su}$, each leaf of $\widetilde{\Lambda}^s_F$ has two distinct ideal points in $\partial_{\infty}F$;
        \item For any $F \in \widetilde{\mathcal{F}}^{su}$, all leaves of $\widetilde{\Lambda}^s_F$ form uniform quasi-geodesics in the hyperbolic metric.
    \end{enumerate}
\end{prop}

\begin{proof}
    Statement (2) is equivalent to the condition that for every $F \in \mathcal{\widetilde{F}}^{su}$, the foliation $\widetilde{\Lambda}^s_F$ is $\mathbb{R}$-covered within $F$. Given that the center foliation $\mathcal{\widetilde{F}}^c$ regulates $\mathcal{\widetilde{F}}^{su}$, we conclude the equivalence of (1) and (2). Lemma~\ref{nonHausdorffness} establishes the implication (2) $\Rightarrow$ (3). Lemma~\ref{quasi-geodesic} proves (3) $\Rightarrow$ (4). By Lemma~\ref{Hausdorff}, (4) $\Rightarrow$ (2). This completes the equivalence relations among these statements.
\end{proof}

\subsection{Leafwise configuration}
We now analyze the action of deck transformations on leaves $F$ of $\mathcal{\widetilde{F}}^{su}$. For foundational insights regarding uniform quasi-geodesic foliations, we direct readers to \cite[Section 6]{BFP20collapsed}, whose framework informs our subsequent analysis. Recent developments in \cite{FP23intersection} provide general results concerning Hausdorff leaf spaces that readers may consult to bypass this subsection.

\begin{lem}\label{1/2fixedpoints}
    Every non-trivial deck transformation on $F$ preserves two distinct ideal points in $\partial_{\infty}F$.
\end{lem}

\begin{proof}
    Let \( g \) be a non-trivial deck transformation acting on \( F \). The transformation \( g \) cannot fix any leaf of \( \widetilde{\Lambda}^s_F \), as closed stable leaves are absent in \( M \). By Lemma~\ref{Hausdorff}, the leaf space of \( \widetilde{\Lambda}^s_F \) is homeomorphic to \( \mathbb{R} \), and consequently, \( g \) acts via translation on this leaf space.

    Let \( l \) be a leaf of \( \widetilde{\Lambda}^s_F \), and let \( a_i \), \( b_i \) denote the ideal points of \( g^i(l) \). By Proposition~\ref{equivalent}(3), \( a_i \) and \( b_i \) cannot coincide. However, \( a_i \) or \( b_i \) may coincide with another ideal point \( a_j \) or \( b_j \).

If \( a_i \) and \( b_i \) are not fixed by \( g \) for any \( i \in \mathbb{N} \), they converge to a common ideal point in \( \partial_{\infty}F \) as \( i \to +\infty \), which lies on the axis of \( g \). Similarly, these points converge to another ideal point on \( g \)-axis as \( i \to -\infty \). Let \( \xi \in \partial_{\infty}F \) denote the forward limit of \( a_i \) and \( b_i \), and \( \eta \in \partial_{\infty}F \) the backward limit—both fixed by \( g \). Notably, the points $\xi$ and $\eta$ could coincide with either $a_0$ or $b_0$ if \( a_i = a_j \) or \( b_i = b_j \) for some \( i \neq j \).

For any deck transformation \( g' \neq g \), the points \( \xi \) and \( \eta \) are likewise fixed by \( g' \). As \( g' \) preserves no leaves of \( \widetilde{\Lambda}^s_F \), it induces a translational action on the leaf space, which is ordered via its homeomorphism to \( \mathbb{R} \) (Lemma~\ref{Hausdorff}). For every leaf \( s \in \widetilde{\Lambda}^s_F \), each iterate \( g'^i(s) \) lies between consecutive leaves \( g^j(l) \) and \( g^{j+1}(l) \) for some \( j \in \mathbb{Z} \). Consequently, the ideal points of \( g'^i(s) \) approach \( \xi \) as \( i \to +\infty \) and \( \eta \) as \( i \to -\infty \), confirming \( \xi \) and \( \eta \) as \( g' \)-fixed points.

It is worth noting that \( \xi \neq \eta \) since any deck transformation is hyperbolic. This completes the lemma's proof.

\end{proof}

\begin{cor}\label{su-cylinder}
    For every leaf $F \in \widetilde{\mathcal{F}}^{su}$, its projection onto $M$ is either a cylinder or a plane.
\end{cor}

\begin{proof}
    Every deck transformation acting on a hyperbolic leaf $F$ constitutes a hyperbolic M\"obius transformation with two fixed ideal boundary points. Each non-trivial deck transformation on $F$ determines a unique connecting geodesic between these fixed points. By Lemma~\ref{1/2fixedpoints}, if $F$ is not a plane, there exists a unique geodesic preserved by all non-trivial deck transformations on $F$. This forces commutativity among deck transformations, endowing $F$'s projection with a cyclic fundamental group. Then such projections must be cylinders or planes.
\end{proof}

A leaf \( F \in \widetilde{\mathcal{F}}^{su} \) is called a \textit{weak quasi-geodesic fan} for \( \widetilde{\Lambda}^s_F \) if there exists an ideal point serving as a limit point for every leaf of \( \widetilde{\Lambda}^s_F \). The term "weak" allows for pairs of leaves in \( \widetilde{\Lambda}^s_F \) to share both ideal points. This weak quasi-geodesic fan structure sufficiently supports our argument; detailed discussions appear in \cite{BFP20collapsed}.

\begin{prop}\label{weakfan}
    Every leaf \( L \in \widetilde{\mathcal{F}}^{su} \) forms a weak quasi-geodesic fan.
\end{prop}

\begin{proof}
    Given the density of $\widetilde{\Lambda}^s_F$'s limit set in $\partial_{\infty}F$, there exist two leaves $l_1, l_2 \in \widetilde{\Lambda}^s_F$ sharing no common pair of ideal points. Let $a_1, b_1$ denote the ideal points of $l_1$, and $a_2, b_2$ those of $l_2$. By quasi-geodesic property, $a_i \neq b_i$ for $i=1,2$. Assuming $a_1 \neq a_2$, they bound an ideal interval $I \subset \partial_{\infty}F$ containing neither $b_1$ nor $b_2$. The interval $I$ contains dense ideal points from $\widetilde{\Lambda}^s_F$. Notice that $b_1$ and $b_2$ may or may not coincide.

\begin{figure}[htb]	
    \centering
    \includegraphics{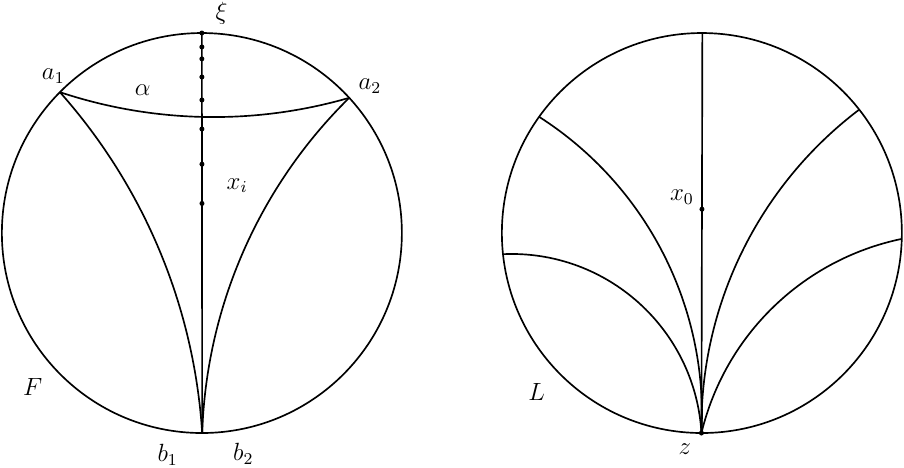}
    \caption{Construction of a weak quasi-geodesic fan}
    \label{fig:fan}
\end{figure}

    Select an ideal point $\xi \in I$. Let $\gamma$ denote the geodesic connecting $\xi$ and $b_1$ in $\partial_{\infty}F$. Consider points $x_i \in \gamma$ converging to $\xi$ as $i \to +\infty$; see Figure~\ref{fig:fan}. There exist deck transformations $\{g_i\}$ such that, after passing to a subsequence, $g_i(x_i) \to x_0$ in some leaf $L \in \widetilde{\mathcal{F}}^{su}$. 

Given $q \in L$, we choose points $q_i \in F$ satisfying $d_F(x_i, q_i) \leq C$ for all $i \in \mathbb{N}$, with $g_i(q_i) \to q$. These $q_i$ also satisfy $q_i \to \xi$ in $\partial_{\infty}F$. As $q$ approaches $\partial_{\infty}L$, the corresponding $q_i$ approach $\partial_{\infty}F$. 

By the diagonal principle, the set of limit points $x_0$ from sequences $x_i \to \xi$ in $F$ contains a minimal $\mathcal{F}^{su}$-sublamination (see Lemma~\ref{distinguished}). Since $\mathcal{F}^{su}$ is minimal (Lemma~\ref{1minimal}), every $L \in \widetilde{\mathcal{F}}^{su}$ contains such a limit point $x_0$ accumulated by $\{g_i(x_i)\}$.

Let $\alpha$ denote the geodesic connecting ideal points $a_1$ and $a_2$. Define $\beta_1$ and $\beta_2$ as the geodesics from $g_i(x_i)$ to $g_i(a_1)$ and $g_i(a_2)$ respectively. Up to neglecting a finite subsequence, all points $g_i(x_i)$ lie within the half-plane bounded by $g_i(\alpha)$. Let $W_i$ be the wedge bounded by $\beta_1$ and $\beta_2$ with boundary interval $I$, and $W^c_i$ its complementary wedge. Observe that $g_i(\alpha)$ resides within $W^c_i$. 

As $i \to +\infty$, the angle of $W_i$ at $g_i(x_i)$ increases monotonically to $2\pi$. Consequently, the wedge $W^c_i$ collapses to a geodesic ray from $x_0$, forcing $g_i(\alpha)$ to converge to an ideal point $z \in \partial_{\infty}L$.

Given the density of $\widetilde{\Lambda}^s_F$'s limit set in $I$, leaves of $\widetilde{\Lambda}^s_F$ with ideal points in $I$ must exist. Proposition~\ref{equivalent} establishes that $\widetilde{\Lambda}^s_F$ constitutes an $\mathbb{R}$-covered foliation within $F$. Consequently, any leaf of $\widetilde{\Lambda}^s_F$ containing an ideal point in $I$ necessarily intersects $\alpha$ at least once. Furthermore, these leaves converge to ideal points within the $\partial_{\infty}F$ interval bounded by $b_1$ and $b_2$—which may collapse to a single point when $b_1 = b_2$. 

As $g_i(\alpha)$ converges to $z \in \partial_{\infty}L$, each $\widetilde{\Lambda}^s_F$ leaf intersecting $\alpha$ converges to a $\widetilde{\Lambda}^s_L$ leaf terminating at $z$. Thus, every leaf $L \in \mathcal{\widetilde{F}}^{su}$ forms a weak quasi-geodesic fan.
\end{proof}

We emphasize that Proposition~\ref{weakfan}'s proof remains valid independently of periodic point existence or dynamical coherence requirements, provided the following hold:  
1. The final three equivalences in Proposition~\ref{equivalent};  
2. A leaf of $\widetilde{\mathcal{F}}^{su}$ exhibits a dense limit set in the ideal boundary.


\section{Proof of Theorem \ref{1accessible}}\label{proof}

Let \( f \) be a partially hyperbolic diffeomorphism of a compact 3-manifold \( M \) satisfying \( NW(f) = M \) with no periodic points. After possibly taking a finite lift and iterate, we assume \( M \), along with the bundles \( E^\sigma \) (\(\sigma = c, s, u\)), are orientable, and \( f \) preserves these orientations. For any lift \( g: \widehat{M} \to \widehat{M} \) of a finite iterate in a finite cover \( \widehat{M} \), we observe \( NW(g) = \widehat{M} \) by \( NW(f) = M \) (see \cite{FP_hyperbolic}). Accessibility of \( f \) follows from that of \( g \), ensuring our orientation assumptions do not affect results. 

Assuming \( \pi_1(M) \) is not virtually solvable, Theorem~\ref{maptori} precludes the existence of 2-dimensional tori tangent to \( E^s \oplus E^u \), \( E^c \oplus E^s \), or \( E^c \oplus E^u \). If \( f \) is non-accessible with \( NW(f) = M \), Corollary~\ref{1nocompact} guarantees three invariant foliations \( \mathcal{F}^{cs} \), \( \mathcal{F}^{cu} \), and \( \mathcal{F}^{su} \) devoid of compact leaves. By Lemma~\ref{1minimal}, \( \mathcal{F}^{su} \) is minimal.

By Corollary~\ref{hyperbolicleaf}, the foliation $\mathcal{F}^{su}$ comprises hyperbolic leaves. Each lifted leaf of $\widetilde{\mathcal{F}}^{su}$ is modeled as a Poincaré disk. For any leaf $F \in \widetilde{\mathcal{F}}^{su}$, let $\partial_{\infty}F$ designate its ideal boundary. As established in Proposition~\ref{>1point} and Corollary~\ref{=1point}, every center leaf of $\widetilde{\mathcal{F}}^c$ determines a flow intersecting each leaf of $\widetilde{\mathcal{F}}^{su}$ at precisely one point. Furthermore, the foliations $\widetilde{\mathcal{F}}^{cs}$ and $\widetilde{\mathcal{F}}^{cu}$ generate subfoliations $\widetilde{\Lambda}^s_F$ and $\widetilde{\Lambda}^u_F$ within every leaf $F$.

If the limit set of $\widetilde{\Lambda}^s_F$ fails to be dense in $\partial_{\infty}F$ for some $F \in \mathcal{\widetilde{F}}^{su}$, Proposition~\ref{uniform} implies $\mathcal{F}^{su}$ cannot exhibit uniformity. Under non-uniformity, Proposition~\ref{nonuniform} demonstrates that $\mathcal{F}^{su}$ constitutes the stable foliation of a flow conjugate to the suspension of an Anosov automorphism, with the fundamental group $\pi_1(M)$ being solvable.

If the limit set of $\widetilde{\Lambda}^s_F$ is dense in $\partial_{\infty}F$ for every $F \in \mathcal{\widetilde{F}}^{su}$, Proposition~\ref{twodistinctiidealpoint} ensures each leaf of $\widetilde{\Lambda}^s_F$ possesses two distinct ideal points in $\partial_{\infty}F$. By Proposition~\ref{equivalent}, this is equivalent to all leaves of $\widetilde{\Lambda}^s_F$ being uniform quasi-geodesics. Combining Proposition~\ref{quasi-geodesic} with the minimality of $\mathcal{F}^{su}$, every leaf $F \in \mathcal{\widetilde{F}}^{su}$ constitutes a weak quasi-geodesic fan. Corollary~\ref{su-cylinder} and Theorem~\ref{rosenberg} necessitate the existence of a cylindrical leaf in $\mathcal{F}^{su}$. Consider $F \in \mathcal{\widetilde{F}}^{su}$ projecting to such a cylinder. Let $h$ denote the deck transformation associated with the generator of the cylinder's fundamental group. Define $z \in \partial_{\infty}F$ as the ideal point determined by all leaves of $\widetilde{\Lambda}^s_F$. This $z$ must be fixed by $h$. Observe that any deck transformation preserving a leaf of $\mathcal{\widetilde{F}}^{su}$ acts as a hyperbolic Möbius transformation. Consequently, $h$ fixes another distinct ideal point $w \in \partial_{\infty}F$.

By Proposition~\ref{equivalent}, the leaf space of $\widetilde{\Lambda}^s_F$ in $F$ is Hausdorff. Lemma~\ref{continuity} then ensures continuous variation of $\widetilde{\Lambda}^s_F$'s ideal points. Given the dense limit points of $\widetilde{\Lambda}^s_F$ in $\partial_{\infty}F$, there exists a leaf $l \in \widetilde{\Lambda}^s_F$ containing the ideal point $w$. Consequently, $l$ possesses two distinct ideal points $w$ and $z$. Let $\gamma$ denote the $h$-invariant geodesic connecting $w$ and $z$. The leaf $l$ lies entirely within the $K_0$-neighborhood of $\gamma$, where $K_0$ is defined in Lemma~\ref{quasi-geodesic}. As $\gamma$ is $h$-invariant, the $h$-orbit of $l$ remains confined within $\gamma$'s $K_0$-neighborhood. This configuration implies the existence of an $h$-invariant leaf in the closure of this neighborhood. Such invariance would necessitate a closed stable leaf in $M$—an impossibility under our assumptions.

Hence, we finish the proof of the Theorem \ref{1accessible}.


\section{Ergodicity}\label{pf_ergodic}

In this section, we establish the proof of Theorem~\ref{ergodic}. Our work directly enhances the following foundational result:

\begin{thm}\cite{2020Seifert}\label{compactperiodiccenter}
    Let $f$ be a $C^2$ conservative partially hyperbolic diffeomorphism without periodic points on a compact 3-manifold. If $f$ admits a compact periodic center leaf, then $f$ is ergodic.
\end{thm}

We employ several established theorems in our argumentation framework:

\begin{thm}\cite{BI08}\label{phlinearpart}
    For a partially hyperbolic diffeomorphism $f: M \rightarrow M$ on a compact 3-manifold $M$ with abelian fundamental group, the induced linear map $g$ on the first homology group $H_1(M)$ remains partially hyperbolic.
\end{thm}

The concept of leaf conjugacy was first introduced in \cite{HPS77}. Given two diffeomorphisms $f, g: M \to M$ with respective invariant foliations $\mathcal{F}$ and $\mathcal{G}$, we say $(f, \mathcal{F})$ and $(g, \mathcal{G})$ are leaf-conjugate if there exists a homeomorphism $h: M \to M$ satisfying: $h(\mathcal{F}(x)) = \mathcal{G}(h(x))$ and $h \circ f(\mathcal{F}(x)) = g \circ h(\mathcal{F}(x))$ for all $x \in M$.

\begin{thm}\cite{Hammerlindl13}\label{on3torus}
    Every partially hyperbolic diffeomorphism on the 3-torus is leaf-conjugate to its linear part.
\end{thm}

The following theorem identifies the first class of closed 3-manifolds where all conservative partially hyperbolic diffeomorphisms are ergodic.
\begin{thm}\cite{2008nil}\label{nil}
    If a conservative $C^2$ partially hyperbolic diffeomorphism acts on a compact orientable 3-manifold with (virtually) nilpotent fundamental group, then it is ergodic unless the manifold is $\mathbb{T}^3$.
\end{thm}

We utilize the following classification modulo leaf conjugacy, which comprehensively characterizes partially hyperbolic diffeomorphisms on manifolds with virtually solvable fundamental groups.
\begin{thm}\cite{HP15}\label{sol-nil}
    Let $f: M \rightarrow M$ be a partially hyperbolic diffeomorphism. If $f$ is dynamically coherent and $\pi_1(M)$ is virtually solvable but not virtually nilpotent, then $f$ — up to a finite iterate — is leaf conjugate to the time-one map of a suspension Anosov flow.
\end{thm}

We now proceed to establish the proof of Theorem~\ref{ergodic}.

\begin{prop}\label{ergodic1}
    Let $f$ be a $C^2$ conservative partially hyperbolic diffeomorphism lacking periodic points on a compact 3-manifold. If $f$ is non-accessible, then $f$ must be ergodic.
\end{prop}

\begin{proof}
    Our proof strategy utilizes Theorem~\ref{compactperiodiccenter} by verifying the existence of a compact periodic center leaf.  As given by Theorem~\ref{1integrable}, the center bundle \( E^c \) of \( f \) is uniquely integrable, therefore admitting a center foliation. If \( f \) is non-accessible, Theorem~\ref{1accessible} establishes that \( \pi_1(M) \) must be virtually solvable. We now categorize our analysis according to three manifold types.

    If the manifold \( M \) is a 3-torus, its fundamental group \( \pi_1(M) = \mathbb{Z}^3 \) is abelian. By Theorem~\ref{phlinearpart}, the linear part \( g \) of \( f \) remains partially hyperbolic. This implies \( g \) possesses eigenvalues \( \lambda_1 \), \( \lambda_2 \), and \( \lambda_3 \) satisfying \( |\lambda_1| < |\lambda_2| < |\lambda_3| \) with \( |\lambda_1| < 1 < |\lambda_3| \). If \( |\lambda_2| \neq 1 \), then \( g \) constitutes a linear Anosov diffeomorphism, rendering \( f \) a DA map (diffeomorphism isotopic to Anosov). This contradicts \( f \)'s possession of periodic points, forcing \( |\lambda_2| = 1 \). The map \( g: \mathbb{T}^2 \times \mathbb{S}^1 \to \mathbb{T}^2 \times \mathbb{S}^1 \) now emerges as a skew product over a 2-toral Anosov automorphism with isometric center. Theorem~\ref{on3torus} ensures both \( g \) and \( f \) exhibit compact invariant center leaves. Thus, \( f \)'s ergodicity follows from Theorem~\ref{compactperiodiccenter}.

    If $\pi_1(M)$ is virtually solvable but not virtually nilpotent, Theorem~\ref{sol-nil} implies that an iterate of $f$ becomes leaf conjugate to the time-one map of a suspension Anosov flow. Such flows possess countably many invariant closed orbits, ensuring — through leaf conjugacy — the existence of a compact periodic center leaf for $f$. Consequently, $f$ is ergodic by Theorem~\ref{compactperiodiccenter}.
    
    The remaining case, where $M$ has a virtually nilpotent fundamental group distinct from $\mathbb{T}^3$, directly follows from Theorem~\ref{nil}.
\end{proof}

We note that the $C^2$ regularity requirement in the preceding proposition enables application of M.-R. Herman's theorem (detailed in \cite{2020Seifert}). The subsequent result however requires reduced regularity $C^{1+\alpha}$, where $\alpha$ denotes a Hölder exponent.

\begin{thm}\cite{08invent,BW10annals}\label{ergodic2}
    For any accessible $C^{1+\alpha}$ conservative partially hyperbolic diffeomorphism with one-dimensional center distribution, the K-property holds, thereby ensuring ergodicity.
\end{thm}

The proof of Theorem~\ref{ergodic} follows directly from combining Proposition~\ref{ergodic1} with Theorem~\ref{ergodic2}.


\section*{Acknowledgements}
We extend our gratitude to Sergio R. Fenley for clarifying Thurston's theorem. We further wish to express our appreciation to Federico Rodriguez Hertz for the invitation and the Department of Mathematics at Penn State University for their hospitality during our visit, where portions of this work were written.

\bibliographystyle{alpha}
\bibliography{ref}

\newcommand{\etalchar}[1]{$^{#1}$}
\begin{thebibliography}{BHH{\etalchar{+}}08}

\bibitem[ACW21]{ACW21}
A.~Avila, S.~Crovisier, and A.~Wilkinson.
\newblock {$C^1$} density of stable ergodicity.
\newblock {\em Advances in Mathematics}, 379:107496, 2021.

\bibitem[Ano67]{Anosov1967}
D.~V. Anosov.
\newblock Geodesic flows on closed riemannian manifolds of negative curvature.
\newblock {\em Trudy Matematicheskogo Instituta Imeni VA Steklova}, 90:3--210, 1967.

\bibitem[Arn65]{Arnold65}
V.~Arnold.
\newblock Small denominators, {I}: mappings of the circumference into itself.
\newblock {\em AMS Trans. Series 2}, 46:213, 1965.

\bibitem[AS67]{AnosovSinai67}
D.~V. Anosov and Ja.~G. Sinai.
\newblock Certain smooth ergodic systems.
\newblock {\em Uspehi Mat. Nauk}, 22(5(137)):107--172, 1967.

\bibitem[Bar95]{Barbot95etds}
T.~Barbot.
\newblock Caract{\'e}risation des flots d'anosov en dimension 3 par leurs feuilletages faibles.
\newblock {\em Ergodic Theory and Dynamical Systems}, 15(2):247--270, 1995.

\bibitem[BFP23]{BFP20collapsed}
T.~Barthelm\'e, S.~R. Fenley, and R.~Potrie.
\newblock Collapsed {A}nosov flows and self orbit equivalences.
\newblock {\em Commentarii Mathematici Helvetici}, 98(4):771--875, 2023.

\bibitem[BGHP20]{BGHP3}
C.~Bonatti, A.~Gogolev, A.~Hammerlindl, and R.~Potrie.
\newblock Anomalous partially hyperbolic diffeomorphisms {III}: abundance and incoherence.
\newblock {\em Geometry \& Topology}, 24(4):1751--1790, 2020.

\bibitem[BH13]{BH13metric}
M.~Bridson and A.~Haefliger.
\newblock {\em Metric spaces of non-positive curvature}, volume 319.
\newblock Springer Science \& Business Media, 2013.

\bibitem[BHH{\etalchar{+}}08]{BHHTU08}
K.~Burns, F.~R. Hertz, M.~R. Hertz, A.~Talitskaya, and R.~Ures.
\newblock Density of accessibility for partially hyperbolic diffeomorphisms with one-dimensional center.
\newblock {\em Discrete and Continuous Dynamical Systems}, 22(1\&2):75--88, 2008.

\bibitem[BI08]{BI08}
D.~Burago and S.~Ivanov.
\newblock Partially hyperbolic diffeomorphisms of 3-manifolds with abelian fundamental groups.
\newblock {\em Journal of Modern Dynamics}, 2(4):541, 2008.

\bibitem[BP74]{BrinPesin74}
M.~I. Brin and Ja.~B. Pesin.
\newblock Partially hyperbolic dynamical systems.
\newblock {\em Izv. Akad. Nauk SSSR Ser. Mat.}, 38:170--212, 1974.

\bibitem[BW05]{BW05}
C.~Bonatti and A.~Wilkinson.
\newblock Transitive partially hyperbolic diffeomorphisms on 3-manifolds.
\newblock {\em Topology}, 44(3):475--508, 2005.

\bibitem[BW08]{BurnsWilkinson08}
K.~Burns and A.~Wilkinson.
\newblock Dynamical coherence and center bunching.
\newblock {\em Discrete and Continuous Dynamical Systems}, 22(1-2):89--100, 2008.

\bibitem[BW10]{BW10annals}
K.~Burns and A.~Wilkinson.
\newblock On the ergodicity of partially hyperbolic systems.
\newblock {\em Annals of Mathematics. Second Series}, 171(1):451--489, 2010.

\bibitem[BZ20]{BZ20}
C.~Bonatti and J.~Zhang.
\newblock Transitive partially hyperbolic diffeomorphisms with one-dimensional neutral center.
\newblock {\em Science China Mathematics}, 63(9):1647--1670, 2020.

\bibitem[Cal00]{Calegari00}
D.~Calegari.
\newblock The geometry of $\mathbb{R}$-covered foliations.
\newblock {\em Geometry \& Topology}, 4(1):457--515, 2000.

\bibitem[Cal07]{Calegari07book}
D.~Calegari.
\newblock {\em Foliations and the geometry of 3-manifolds}.
\newblock Oxford Mathematical Monographs. Oxford University Press, Oxford, 2007.

\bibitem[Can93]{Candel93}
A.~Candel.
\newblock Uniformization of surface laminations.
\newblock {\em Annales scientifiques de l'Ecole normale sup{\'e}rieure}, 26(4):489--516, 1993.

\bibitem[CC00]{CC00I}
A.~Candel and L.~Conlon.
\newblock {\em Foliations. {I}}, volume~23 of {\em Graduate Studies in Mathematics}.
\newblock American Mathematical Society, Providence, RI, 2000.

\bibitem[CD03]{CalegariDunfield03}
D.~Calegari and N.~Dunfield.
\newblock Laminations and groups of homeomorphisms of the circle.
\newblock {\em Inventiones mathematicae}, 152(1):149--204, 2003.

\bibitem[CHHU18]{2018survey}
P.~Carrasco, F.~R. Hertz, J.~R. Hertz, and R.~Ures.
\newblock Partially hyperbolic dynamics in dimension three.
\newblock {\em Ergodic Theory and Dynamical Systems}, 38(8):2801--2837, 2018.

\bibitem[DW03]{DW03}
D.~Dolgopyat and A.~Wilkinson.
\newblock Stable accessibility is {$C^1$} dense.
\newblock {\em Ast{\'e}risque}, 287:33--60, 2003.

\bibitem[Fen92]{Fenley92}
S.~R. Fenley.
\newblock Quasi-isometric foliations.
\newblock {\em Topology}, 31(3):667--676, 1992.

\bibitem[Fen02]{Fenley02}
S.~R. Fenley.
\newblock Foliations, topology and geometry of 3-manifolds: R-covered foliations and transverse pseudo-anosov flows.
\newblock {\em Commentarii Mathematici Helvetici}, 77(3):415--490, 2002.

\bibitem[Fen09]{Fenley09}
S.~R. Fenley.
\newblock Geometry of foliations and flows {I}: Almost transverse pseudo-anosov flows and asymptotic behavior of foliations.
\newblock {\em Journal of Differential Geometry}, 81(1):1--89, 2009.

\bibitem[FP21]{FP20minimal}
S.~R. Fenley and R.~Potrie.
\newblock Minimality of the action on the universal circle of uniform foliations.
\newblock {\em Groups, Geometry, and Dynamics}, 15(4):1489--1521, 2021.

\bibitem[FP22]{FP_hyperbolic}
S.~R. Fenley and R.~Potrie.
\newblock Ergodicity of partially hyperbolic diffeomorphisms in hyperbolic 3-manifolds.
\newblock {\em Advances in Mathematics}, 401:108315, 2022.

\bibitem[FP23]{FP23intersection}
S.~R. Fenley and R.~Potrie.
\newblock Intersection of transverse foliations in 3-manifolds: Hausdorff leafspace implies leafwise quasi-geodesic.
\newblock {\em arXiv preprint arXiv:2310.05176}, 2023.

\bibitem[FP24a]{FP21accessible}
S.~R. Fenley and R.~Potrie.
\newblock Accessibility and ergodicity for collapsed {A}nosov flows.
\newblock {\em American Journal of Mathematics}, 146(5):1339--1359, 2024.

\bibitem[FP24b]{FP-gafa}
S.~R. Fenley and R.~Potrie.
\newblock Partial hyperbolicity and pseudo-{A}nosov dynamics.
\newblock {\em Geometric and Functional Analysis}, 34(2):409--485, 2024.

\bibitem[FU24]{FU2}
Z.~Feng and R.~Ures.
\newblock Accessibility and central integrability in the absence of periodic points.
\newblock {\em In preparation}, 2024+.

\bibitem[Gab90]{Gabai90}
D.~Gabai.
\newblock Foliations and 3-manifolds.
\newblock In {\em Proceedings of the International Congress of Mathematicians}, volume~1, pages 609--619, 1990.

\bibitem[GO89]{GabaiOertel89}
D.~Gabai and U~Oertel.
\newblock Essential laminations in 3-manifolds.
\newblock {\em Annals of Mathematics}, 130(1):41--73, 1989.

\bibitem[GPS94]{GPS94}
M.~Grayson, C.~Pugh, and M.~Shub.
\newblock Stably ergodic diffeomorphisms.
\newblock {\em Annals of Mathematics}, 140(2):295--329, 1994.

\bibitem[Gro87]{Gromov87}
M.~Gromov.
\newblock Hyperbolic groups.
\newblock In {\em Essays in group theory}, volume~8 of {\em Math. Sci. Res. Inst. Publ.}, pages 75--263. Springer, New York, 1987.

\bibitem[GS20]{GS20DA}
S.~Gan and Y.~Shi.
\newblock Rigidity of center lyapunov exponents and $ su $-integrability.
\newblock {\em Commentarii Mathematici Helvetici}, 95(3):569--592, 2020.

\bibitem[Ham13]{Hammerlindl13}
A.~Hammerlindl.
\newblock Leaf conjugacies on the torus.
\newblock {\em Ergodic theory and dynamical systems}, 33(3):896--933, 2013.

\bibitem[HH87]{HectorHirsch}
G.~Hector and U.~Hirsch.
\newblock {\em Introduction to the Geometry of Foliations, Part B}.
\newblock Springer, 1987.

\bibitem[HHU07]{2006Some}
F.~R. Hertz, M.~R. Hertz, and R.~Ures.
\newblock Some results on the integrability of the center bundle for partially hyperbolic diffeomorphisms.
\newblock In {\em Partially hyperbolic dynamics, laminations, and {T}eichm\"uller flow}, volume~51 of {\em Fields Inst. Commun.}, pages 103--109. Amer. Math. Soc., Providence, RI, 2007.

\bibitem[HHU08a]{08invent}
F.~R. Hertz, M.~R. Hertz, and R.~Ures.
\newblock Accessibility and stable ergodicity for partially hyperbolic diffeomorphisms with 1d-center bundle.
\newblock {\em Inventiones mathematicae}, 2(172):353--381, 2008.

\bibitem[HHU08b]{2008nil}
F.~R. Hertz, M.~R. Hertz, and R.~Ures.
\newblock Partial hyperbolicity and ergodicity in dimension three.
\newblock {\em Journal of Modern Dynamics}, 2(2):187--208, 2008.

\bibitem[HHU11]{2011TORI}
F.~R. Hertz, J.~R. Hertz, and R.~Ures.
\newblock Tori with hyperbolic dynamics in 3-manifolds.
\newblock {\em Journal of Modern Dynamics}, 5(1):185--202, 2011.

\bibitem[HHU16a]{2015Center}
F.~R. Hertz, J.~R. Hertz, and R.~Ures.
\newblock Center-unstable foliations do not have compact leaves.
\newblock {\em Mathematical Research Letters}, 23(6):1819--1832, 2016.

\bibitem[HHU16b]{2016example}
F.~R. Hertz, J.~R. Hertz, and R.~Ures.
\newblock A non-dynamically coherent example on t3.
\newblock {\em Annales de l'Institut Henri Poincaré C, Analyse non linéaire}, 33(4):1023--1032, 2016.

\bibitem[HHU20]{2020Seifert}
A.~Hammerlindl, J.~R. Hertz, and R.~Ures.
\newblock Ergodicity and partial hyperbolicity on {S}eifert manifolds.
\newblock {\em J. Mod. Dyn.}, 16:331--348, 2020.

\bibitem[H{\"o}l01]{Holder1901}
O.~H{\"o}lder.
\newblock Die axiome der quantit{\"a}t und die lehre vom mass.
\newblock {\em Ber. Verh. Sachs. Ges. Wiss. Leipzig, Math. Phys. C1.}, 53:1--64, 1901.

\bibitem[Hop39]{Hopf1939}
E.~Hopf.
\newblock Statistik der geod\"atischen {L}inien in {M}annigfaltigkeiten negativer {K}r\"ummung.
\newblock {\em Ber. Verh. S\"achs. Akad. Wiss. Leipzig Math.-Phys. Kl.}, 91:261--304, 1939.

\bibitem[HP15]{HP15}
A.~Hammerlindl and R.~Potrie.
\newblock Classification of partially hyperbolic diffeomorphisms in 3-manifolds with solvable fundamental group.
\newblock {\em Journal of Topology}, 8(3):842--870, 2015.

\bibitem[HPS77]{HPS77}
M.~Hirsch, C~Pugh, and M~Shub.
\newblock Invariant manifolds.
\newblock {\em Springer Lecture Notes in Mathematics, 583.}, 1977.

\bibitem[HS21]{HS21DA}
A.~Hammerlindl and Y.~Shi.
\newblock Accessibility of derived-from-anosov systems.
\newblock {\em Transactions of the American Mathematical Society}, 374(4):2949--2966, 2021.

\bibitem[HU14]{HamU14CCM}
A.~Hammerlindl and R.~Ures.
\newblock Ergodicity and partial hyperbolicity on the 3-torus.
\newblock {\em Communications in Contemporary Mathematics}, 16(04):1350038, 22pp, 2014.

\bibitem[Ima74]{Imanishi74}
H.~Imanishi.
\newblock On the theorem of denjoy-sacksteder for codimension one foliations without holonomy.
\newblock {\em Journal of Mathematics of Kyoto University}, 14(3):607--634, 1974.

\bibitem[Kol54]{KAM54}
A.~Kolmogorov.
\newblock On the conservation of conditionally periodic motions under small perturbation of the hamiltonian.
\newblock {\em Dokl. Akad. Nauk. SSR}, 98(527):2--3, 1954.

\bibitem[Lev83]{Levitt83}
G.~Levitt.
\newblock Foliations and laminations on hyperbolic surfaces.
\newblock {\em Topology}, 22(2):119--135, 1983.

\bibitem[Mos62]{Moser62}
J.~Moser.
\newblock On invariant curves of area-preserving mappings of an annulus.
\newblock {\em Nachr. Akad. Wiss. G\"ottingen Math.-Phys. Kl. II}, 1962:1--20, 1962.

\bibitem[OU41]{OU1941}
J.~C. Oxtoby and S.~M. Ulam.
\newblock Measure-preserving homeomorphisms and metrical transitivity.
\newblock {\em Annals of Mathematics. Second Series}, 42:874--920, 1941.

\bibitem[Pal78]{Palmeira}
C.~F. Palmeira.
\newblock Open manifolds foliated by planes.
\newblock {\em Annals of Mathematics}, 107(1):109--131, 1978.

\bibitem[Pla83]{Plante83solvable}
J.~F. Plante.
\newblock Solvable groups acting on the line.
\newblock {\em Transactions of the American Mathematical Society}, 278(1):401--414, 1983.

\bibitem[PS72]{PughShub72}
C.~Pugh and M.~Shub.
\newblock Ergodicity of {A}nosov actions.
\newblock {\em Inventiones Mathematicae}, 15:1--23, 1972.

\bibitem[Ros68]{Rosenberg}
H.~Rosenberg.
\newblock Foliations by planes.
\newblock {\em Topology}, 7(2):131--138, 1968.

\bibitem[Rou71]{Roussarie71}
R.~Roussarie.
\newblock Sur les feuilletages des vari{\'e}t{\'e}s de dimension trois.
\newblock {\em Annales de l'institut Fourier}, 21(3):13--82, 1971.

\bibitem[Sul76]{Sullivan76}
D.~Sullivan.
\newblock Cycles for the dynamical study of foliated manifolds and complex manifolds.
\newblock {\em Inventiones mathematicae}, 36(1):225--255, 1976.

\bibitem[Thu88]{Thurston88}
W.~Thurston.
\newblock On the geometry and dynamics of diffeomorphisms of surfaces.
\newblock {\em Bulletin of the American mathematical society}, 19(2):417--431, 1988.

\bibitem[Wil98]{Wilkinson98}
A.~Wilkinson.
\newblock Stable ergodicity of the time-one map of a geodesic flow.
\newblock {\em Ergodic Theory and Dynamical Systems}, 18(6):1545--1587, 1998.

\bibitem[Zha21]{Zhang21}
J.~Zhang.
\newblock Partially hyperbolic diffeomorphisms with one-dimensional neutral center on 3-manifolds.
\newblock {\em Journal of Modern Dynamics}, 17(0):557--584, 2021.

\end{thebibliography}

\end{document}